\documentclass[10pt]{amsart}
\usepackage{tikz}

\usepackage[text={475pt,660pt},centering]{geometry}

\usepackage{color}
\usepackage{esint,amssymb}
\usepackage{graphicx}
\usepackage{MnSymbol}
\usepackage{amsmath, mathtools}
\usepackage[colorlinks=true, pdfstartview=FitV, linkcolor=blue, citecolor=blue, urlcolor=blue,pagebackref=false]{hyperref}
\usepackage{microtype}

\usepackage{bm}
\usepackage{scalerel} 
\usepackage{dsfont}
\usepackage{mathrsfs}
\usepackage[font={footnotesize}]{caption}
\usepackage{stmaryrd}

\definecolor{darkgreen}{rgb}{0,0.5,0}
\definecolor{darkblue}{rgb}{0,0,0.7}
\definecolor{darkred}{rgb}{0.9,0.1,0.1}

\makeatletter
\newtheorem*{rep@theorem}{\rep@title}
\newcommand{\newreptheorem}[2]{%
\newenvironment{rep#1}[1]{%
 \def\rep@title{#2 \ref{##1}}%
 \begin{rep@theorem}}%
 {\end{rep@theorem}}}
\makeatother

\newtheorem{theorem}{Theorem}
\newreptheorem{theorem}{Theorem}
\newreptheorem{proposition}{Proposition}
\newreptheorem{lemma}{Lemma}
\newtheorem{proposition}{Proposition}
\newtheorem{lemma}[proposition]{Lemma}
\newtheorem{corollary}[proposition]{Corollary}

\theoremstyle{remark}

\theoremstyle{definition}
\newtheorem{definition}[proposition]{Definition}
\newtheorem{remark}[proposition]{Remark}

\numberwithin{equation}{section}
\numberwithin{proposition}{section}

\newcommand{\Z}{\mathbb{Z}}
\newcommand{\N}{\mathbb{N}}
\newcommand{\R}{\mathbb{R}}

\newcommand{\E}{\mathbb{E}}

\newcommand{\Zd}{\mathbb{Z}^d}
\newcommand{\Rd}{{\mathbb{R}^d}}

\newcommand{\di}{\mathrm{d}}
\newcommand{\inte}[1]{%
 {\kern0pt#1}^{\mathrm{o}}%
}

\newcommand{\ep}{\varepsilon}

\renewcommand{\a}{\mathbf{a}}

\renewcommand{\subset}{\subseteq}

\DeclareMathOperator{\cov}{cov}
\DeclareMathOperator{\diam}{diam}

\DeclareMathOperator{\supp}{supp}

\renewcommand{\bar}{\overline}

\newcommand{\indc}{\mathds{1}}

\newcommand{\C}{C}

\newcommand{\G}{\mathcal{G}}

\newcommand{\f}{\mathbf{f}}
\renewcommand{\L}{\mathcal{L}}

\title{Parallel spin wave for the Villain model}
\author{Paul Dario}
\thanks{(P. Dario) LAMA, Universit\'e Paris-Est Cr\'eteil, CNRS UMR 8050, and CNRS, France.
{\footnotesize paul.dario@u-pec.fr.}
}
\author{Wei Wu}
\thanks{(W. Wu) Department of Mathematics, NYU Shanghai and NYU-ECNU Institute of Mathematical Sciences, China.
{\footnotesize wei.wu@nyu.edu.}
}

\begin{document}

\maketitle

\begin{abstract}
    In this paper, we study the Villain model in $\Zd$ in dimension $d\geq 3$. It is conjectured, that the parallel correlation function in the infinite volume Gibbs state, i.e., the map
    $$
    x \mapsto \langle \cos\theta(0) \cos\theta(x) \rangle_{\mu_{\mathrm{Vil}, \beta}} -\left( \langle \cos\theta(0) \rangle_{\mu_{\mathrm{Vil}, \beta}} \right)^2,
    $$
    decays like $|x|^{-2(d-2)}$ as $|x| \to \infty$ at low temperature. The results of~\cite{BFLLS} show that for the related XY model, this correlation decays at least as fast as $|x|^{2-d}$. We prove the optimal upper and lower bounds for the Villain model in $d=3$, up to a logarithmic correction, and also improve the upper bound in general dimensions.

    Our proof builds upon the approach developed in~\cite{DW}, which in turn is inspired by a key observation of Fr\"{o}hlich and Spencer~\cite{FS4d}: in the low temperature regime, a combination of duality transformation and renormalisation allows certain properties of the Villain model to be analysed in terms of a (vector-valued) $\nabla \varphi$ interface model. This latter model can be investigated using the Helffer-Sj\"{o}strand representation formula (originally introduced by Helffer and Sjöstrand~\cite{HS} and used by Naddaf and Spencer~\cite{NS} and Giacomin, Olla and Spohn~\cite{GOS} to identify the scaling limit of the model). In this paper, we go beyond the results of~\cite{DW} by deriving more refined estimates, combining the representation with tools from elliptic and parabolic regularity theory.
\end{abstract}

\section{Introduction}

In this paper we study the Villain model \cite{Vil}, a classical statistical mechanics model with two-component spins. For $L \in \N$, we consider the finite box $\Lambda_L := \left\{ - L , \ldots, L \right\}^d \subset \Zd$, where $\Zd$ is the Euclidean lattice in dimension $d \geq 3$. We denote by $\partial\Lambda_L$ its boundary, namely the set of vertices in $\Lambda_L$ that are connected to $\Zd\setminus \Lambda_L$ by some edge, and by $E(\Lambda_L)$ its edge set. We also set $\Lambda_L^\circ = \Lambda_L \setminus \partial \Lambda_L$. The Villain model at inverse temperature $\beta \in (0 , \infty)$ on the box $\Lambda_L$ with zero boundary condition is given by the following Gibbs measure on the set of functions $\Omega_{L} := \left\{ \theta : \Lambda_L \to [- \pi , \pi) \right\} \simeq [- \pi , \pi)^{\Lambda_L}$
\begin{equation} \label{e.finitedirichlet}
  \mu_{\mathrm{Vil}, \beta, L}(d\theta):= \frac{1}{Z_{L,\beta}} \prod_{(x,y) \subset E(\Lambda_L)} v_\beta(\theta(x) - \theta(y)) \prod_{x\in \partial \Lambda_L}\delta_0(\theta(x)) \prod_{x\in \Lambda_L^\circ}  \,d\theta(x),
\end{equation}
where the interaction $v_\beta$ is the heat kernel on $S^1$ defined by
\begin{equation*} 
v_\beta(\theta) := \sum_{m\in \Z} \exp{\left(-\frac{\beta}{2} (\theta+ 2\pi m)^2\right)},
\end{equation*}
and $Z_{L,\beta}$ is the normalization constant that makes \eqref{e.finitedirichlet} a probability measure. 
The expectation with respect to the Gibbs measure \eqref{e.finitedirichlet} is denoted by $\left\langle \cdot \right\rangle_{\mu_{\mathrm{Vil}, \beta, L}}$. We also define the spin variable, which takes values on the unit circle, by $S_x= (\cos{\theta(x)},\sin{\theta(x)})$. By the $\theta \to - \theta$ symmetry, we have
\begin{equation*}
\left\langle S_0 \cdot S_x \right\rangle_{\mu_{\mathrm{Vil}, \beta, L}} = \left\langle e^{i (\theta(0) - \theta(x))}\right\rangle_{\mu_{\mathrm{Vil}, \beta,L}}.
\end{equation*}
The Villain model is a fundamental example of the \textit{Abelian spin models}, and is believed to share the same large-scale properties of other models in this class, such as  the classical XY model. 

It is known that, as a consequence of the Ginibre correlation inequality \cite{Gi}, the thermodynamic limit of the measures \eqref{e.finitedirichlet} exists as $L \to \infty$ (i.e., the sequence of measures $\mu_{\mathrm{Vil}, \beta,L}$ converges as $L \to \infty$). We denote by $\mu_{\mathrm{Vil}, \beta}$ the corresponding infinite-volume Gibbs measure (N.B. this measure is invariant under translations). A natural question for the Villain model in infinite volume is whether the system exhibits long range order (or symmetry breaking), in the sense that $\left\langle  S_x \right\rangle_{\mu_{\mathrm{Vil}, \beta}} >0$. Indeed, spontaneous symmetry breaking only occurs when $d\ge 3$: in $d=2$ there is no spontaneous magnetization for any $\beta$ \cite{MW} (i.e., $\left\langle  S_x \right\rangle_{\mu_{\mathrm{Vil}, \beta}} = 0$ for any $\beta > 0$). In contrast, in $d\ge 3$ the breakthrough work of Fr\"{o}hlich, Simon and Spencer~\cite{FSS} shows that these models undergo an order/disorder phase transitions: there exists an critical inverse temperature $\beta_c > 0$ such that 
\begin{equation*}
    \mbox{for any}~ \beta > \beta_c, ~~\left\langle  S_x \right\rangle_{\mu_{\mathrm{Vil}, \beta}} \neq 0 \hspace{5mm} \mbox{and} \hspace{5mm} \mbox{for any} ~~ \beta < \beta_c,~~ \left\langle  S_x \right\rangle_{\mu_{\mathrm{Vil}, \beta}} = 0.
\end{equation*}

A fundamental question when the system has spontaneous symmetry breaking is to determine the precise asymptotics of the \textit{truncated spin correlations} $\left\langle S_0 \cdot S_x \right\rangle_{\mu_{\mathrm{Vil}, \beta}} -\left\langle S_0  \right\rangle_{\mu_{\mathrm{Vil}, \beta}}^2 $. It was conjectured in \cite{Dy} (see also \cite{MW}), known as the \textit{Gaussian spin-wave} prediction, that the asymptotic truncated spin correlation should behave like a massless free field, decaying as $|x|^{2-d}$. Heuristic spin-wave theory also leads to predictions about the asymptotic \textit{transversal} (namely, $\langle \sin\theta(0) \sin\theta(x) \rangle_{\mu_{\mathrm{Vil}, \beta}}$) and \textit{parallel} ($\langle \cos\theta(0) \cos\theta(x) \rangle_{\mu_{\mathrm{Vil}, \beta}}$) two-point functions. Progress was made in the 80s, when Fr\"{o}hlich and Spencer \cite{FS4d} and Bricmont, Fontaine, Lebowitz, Lieb and Spencer \cite{BFLLS} obtained upper and lower bounds for asymptotic spin correlations and transversal two-point functions.

The precise asymptotics, stated below, were recently obtained by the authors in \cite{DW}, thus rigorously justifying the spin wave conjecture for the spin correlations and the transversal two-point functions. 

\begin{theorem}[\cite{DW}]
\label{t.twopt}
Let $d\geq3$. There exist constants $\beta_0=\beta_0(d) \in (0 , \infty)$, $ c_1= c_1(\beta,d) > 0, c_2= c_2(\beta,d)> 0$, and $\alpha=\alpha(d)>0$ such that, for all $\beta>\beta_0$, the transversal two-point function has the asymptotics
\begin{equation} \label{e.2ptfirstfirst}
\langle \sin\theta(0) \sin\theta(x) \rangle_{\mu_{\mathrm{Vil}, \beta}}
= \frac{c_2}{|x|^{d-2}} + O\left(\frac{1}{|x|^{d-2+\alpha}}\right),
\end{equation}
and the spin-spin correlation function satisfies 
\begin{equation}
\label{e.2ptfirst}
  \left\langle S_0 \cdot S_x \right\rangle_{\mu_{\mathrm{Vil}, \beta}} =
 \langle S_0\rangle_{\mu_{\mathrm{Vil}, \beta}}^2 + \frac{c_1}{|x|^{d-2}}+ O\left(\frac{1}{|x|^{d-2+\alpha}}\right).
 \end{equation} 
\end{theorem}

It remains to justify the spin wave prediction for the parallel two-point function. The best known rigorous bounds for the parallel two-point function in the XY model were obtained in \cite{BFLLS}. They proved that
\begin{equation*} 
    \frac{c}{|x|^{2(d-2)}} \leq \langle \cos\theta(0) \cos\theta(x) \rangle_{\mu_{\mathrm{XY} , \beta}} -\left( \langle \cos\theta(0) \rangle_{\mu_{XY, \beta}} \right)^2 
    \leq
    \frac{C}{|x|^{d-2}}.
\end{equation*}
It is believed that the lower bound above is optimal, as suggested by a spin-wave calculation. Indeed, a heuristic calculation by replacing the low temperature Villain model by a Gaussian free field, and using $\cos \theta \approx 1- \frac{\theta^2}{2}$ for $\theta$ small yields
\begin{equation} \label{eq:backoftheenvelop}
    \langle \cos\theta(0) \cos\theta(x) \rangle_{\mu_{\mathrm{Vil}, \beta}} -\left( \langle \cos\theta(0) \rangle_{\mu_{\mathrm{Vil}, \beta}} \right)^2 
    \approx
    \langle \varphi(0)^2 ; \varphi(x)^2 \rangle_{\mathrm{GFF}}
    = 
    c \langle \varphi(0) \varphi(x) \rangle_{\mathrm{GFF}}^2,
\end{equation}
and therefore decays as $\frac{C}{|x|^{2(d-2)}}$.  In this paper we obtain an almost optimal upper bound for the Villain model in $d=3$.

\begin{theorem} \label{th:mainth}
Let $d=3$, then there exists an inverse temperature $\beta_0 $ and an exponent $\kappa \in \N$, such that for any $\beta \geq \beta_0$, there exist constants $C_1 := C_1( \beta) <\infty, C_2 = C_2( \beta) >0$ such that, for any $x \in \Zd \setminus \{ 0 \}$,
\begin{equation*}
\frac{C_2}{|x|^{2}}
\le
    \langle \cos\theta(0) \cos\theta(x) \rangle_{\mu_{\mathrm{Vil}, \beta}} -\left( \langle \cos\theta(0) \rangle_{\mu_{\mathrm{Vil}, \beta}} \right)^2 
    \leq
    \frac{C_1 \left( \ln |x|_+ \right)^\kappa}{|x|^{2}}.
\end{equation*}
More generally, for $d \geq 3$, there exists an inverse temperature $\beta_0 := \beta_0(d) $ and an exponent $\kappa := \kappa(d) \in \N$, such that for any $\beta \geq \beta_0$, there exist constants $C_1 := C_1(d , \beta) < \infty, C_2 = C_2(d , \beta) > 0$ such that, for any $x \in \Zd \setminus \{ 0 \}$,
\begin{equation*}
   \frac{C_2}{|x|^{2(d-2)}}
   \leq
   \langle \cos\theta(0) \cos\theta(x) \rangle_{\mu_{\mathrm{Vil}, \beta}} -\left( \langle \cos\theta(0) \rangle_{\mu_{\mathrm{Vil}, \beta}} \right)^2 
   \leq \frac{C_1 \left( \ln |x|_+ \right)^\kappa}{|x|^{d-1}}.
\end{equation*}
\end{theorem}
An immediate consequence of Theorem \ref{th:mainth} is that we can identify the leading coefficients in \eqref{e.2ptfirst} and \eqref{e.2ptfirstfirst}. 
\begin{corollary}
    In the setting of Theorem \ref{t.twopt}, $c_1 = c_2$.
\end{corollary}

\begin{remark}
Let us make two remarks about Theorem~\ref{th:mainth}:
\begin{itemize}
\item The proof gives the value $\kappa = 4d + 10$ (and in particular $\kappa = 22$ in dimension $3$). We mention that the argument could certainly be optimised to reduce the value of this exponent, but proving the upper bound with $\kappa = 0$ (which would be optimal in dimension $d = 3$) using the techniques developed in this article seems to (at least) require a (much) more technical argument.
\item Regarding the dependency of the constants in the inverse temperature $\beta$, we note that an inspection of the proof shows that $C_1 \leq \frac{A_1}{\beta^2}$ for some constant $ A_1 := A_1(d) < \infty$, and $C_2 \geq \frac{A_2}{\beta^2}$ for some constant $ A_2 := A_2(d) >0$. We additionally note that, the proof of this paper can be adapted (with some nontrivial modifications)  to show that 
\begin{equation*}
\langle \cos\theta(0) \cos\theta(x) \rangle_{\mu_{\mathrm{Vil}, \beta}} -\left( \langle \cos\theta(0) \rangle_{\mu_{\mathrm{Vil}, \beta}} \right)^2 
   \leq 
   \frac{C_0}{\beta^2|x|^{2(d-2)}}
   + 
   e^{-c\beta }\frac{ \left( \ln |x|_+ \right)^\kappa}{|x|^{d-1}}.
\end{equation*}
For $\beta \gg 1$, this inequality yields the optimal bound as long as $|x|\le e^{c'\beta}$ (and $|x|\le e^{e^{c'\beta}}$ for $d=3$).
\end{itemize}
\end{remark}

Understanding spin wave prediction in systems with continuous symmetry is one of the famous problems for spin systems in three dimensions. In the symmetry-breaking phase, the emergence of massless excitations--known as the Goldstone bosons-- is predicted by the celebrated \textit{Goldstone Theorem} \cite{gold}. This predicts, in particular, that in systems with an $O(N)$ symmetry, there exist $N-1$ distinct Goldstone bosons in the symmetry-breaking phase, and the connected $\phi^2-\phi^2$  correlation decays asymptotically like $|x|^{-2(d-2)}$. The study of these massless modes was first carried out by Symanzik in the setting of the Euclidean field theories~\cite{Sym1, Sym2}. An alternative approach was developed by Balaban in the 1980s and 1990s, through an elaborated renormalization group analysis for the $O(N)$ vector models with $N\ge 2$ (where the $N\ge 3$ case corresponds to non-Abelian spins) and culminated in \cite{BOC}. The results in \cite{BOC} were related to those in Theorem \ref{t.twopt}, and in the $N=2$ case, for the longitudinal two-point function, they were able to show 
\[
 \langle \cos\theta(0) \cos\theta(x) \rangle_{\mu_{\mathrm{XY},\beta}} - (\langle \cos\theta(0) \rangle_{\mu_{\mathrm{XY},\beta}})^2  
   \leq \frac{C_1}{|x|^{d-2+\rho}},
\]
for some $\rho >0$. Theorem \ref{th:mainth} improves upon this result, and our approach to the $N=2$ case differs from Balaban's approach, in that it is based on the duality between the Abelian spin model and height function, the Helffer-Sj{\"o}strand representation, and tools from elliptic and parabolic regularity theory. 

Other recent progress toward understanding the Goldstone theorem and massless fluctuations in three dimensions includes the work of Giuliani and Ott \cite{GO}, who extended the method of \cite{BFLLS} to $O(N)$ model with $N\ge 3$, and justified the asymptotic expansion for the Heisenberg magnetization at low temperature. Bauerschmidt, Crawford and Helmuth \cite{BCH} established the spin wave prediction for the aboreal gas in $d \ge 3$ via a spin representation. Massless fluctuations in the symmetric breaking phase have also been studied, in different forms, in the context of the Abelian Higgs model \cite{KK}, the Abelian gauge theory \cite{guth, FS4d, GS}, and quantum spin models \cite{DLS,CGS}. Related phenomena in two-dimensional settings, where massless behaviour arises from a different mechanism (the Kosterlitz-Thouless phase \cite{KT}), have been analysed, for instance, in the work of Falco on the 2D Coulomb gas \cite{Falco1, Falco2}, in the multiscale analysis by Fr\"{o}hlich, and Spencer \cite{FrSp}, and in the rigorous renormalization group analysis by Bauerschmidt, Park, and Rodriguez for the 2D integer valued height functions \cite{BPR1,BPR2}.

\subsection{Proof strategy and organisation}

The proof of Theorem~\ref{th:mainth} relies on a combination of duality between the spin model and the Coulomb gas, renormalisation techniques, the Helffer–Sj\"{o}strand representation formula, and regularity theory for parabolic systems with small ellipticity contrast. As in~\cite{DW}, the starting point is a duality transformation introduced by Fr\"{o}hlich and Spencer~\cite{FrSp, FS4d} (see also~\cite[Chapter 5]{Bau}) which, together with a cluster expansion, is used to map the Villain model on $\mathbb{Z}^d$, for $d \geq 3$, to a vector-valued random interface model (whose Gibbs measure is denoted by $\mu_\beta$ below; see Section~\ref{sec:randominterfaces}). This model features an infinite-range potential which is a small, smooth perturbation of the Gaussian free field on~$\Zd$.  The key observation is that this second model is more analytically tractable than the original Villain model. In particular, it belongs to the well-studied class of $\nabla \varphi$-random interface models with uniformly convex potentials (see~\cite{F05, V06, Sh} for an overview of the literature and more results on this model). A wide array of analytical tools is available for studying such models, including the Helffer–Sj\"{o}strand representation~\cite{HS, NS, GOS}, which plays a central role in the proofs of this article.

As a result, the general strategy of this article is to study the truncated parallel correlation by first using trigonometric identities to write
\begin{multline}
\label{e.1}
\langle \cos\theta(0) \cos\theta(x) \rangle_{\mu_{\mathrm{Vil}, \beta}} -\left( \langle \cos\theta(0) \rangle_{\mu_{\mathrm{Vil}, \beta}} \right)^2 \\ = \frac{1}{2} \langle \cos (\theta(0) -\theta(x)) \rangle_{\mu_{\mathrm{Vil}, \beta}}+
\frac{1}{2} \langle \cos (\theta(0) +\theta(x)) \rangle_{\mu_{\mathrm{Vil}, \beta}}
- \left( \langle \cos (\theta(0)) \rangle_{\mu_{\mathrm{Vil}, \beta}} \right)^2,
\end{multline}
by then using the duality transformation of~\cite{FS4d} to map each term on the right-hand side above to a non-linear and non-local observable in the interface model (see Proposition~\ref{p.dual} for the definitions of these observables). Understanding their behaviour requires a precise and quantitative theory that captures the large-scale behaviour of the random interface. This is achieved by making use of the Helffer-Sj\"{o}strand formula, following~\cite{HS, NS, GOS}, as explained in the following paragraphs.

To study the non-local observable in the random interface model, a main tool that we use is the Helffer-Sj{\"o}strand representation \cite{HS}, which has been implemented by Naddaf and Spencer \cite{NS} and Giacomin, Olla and Spohn \cite{GOS}, to represent the correlation of the random interface model to the solutions of an infinite-dimensional elliptic equation (see also \cite{Sj}). In order to define the elliptic equation, we let $\Omega$ be the infinite-dimensional vector space (N.B. this space is the probability space which is at the centre of this article; an element of $\Omega$ will be denoted by $\varphi$ and will be referred to as an interface)
$$
\Omega := \left\{ \varphi : \Zd \to \R^{d \choose 2} \right\}.
$$
The elliptic equation is defined via the Helffer-Sj{\"o}strand operator which acts on functions defined in the space $\Omega \times \Zd$ and valued in $\R^{d \choose 2}$ and takes the following form
\begin{equation*}
  \mathcal{L} := -\Delta_\varphi - \mathcal{L}_{\mathrm{spat}},
\end{equation*}
where:
\begin{itemize}
    \item The operator $\Delta_\varphi$ is the (infinite-dimensional) Laplacian computing derivatives with respect to the interfaces of $\Omega$.
    \item The operator $\mathcal{L}_{\mathrm{spat}}$ acts on the spatial variable $x \in \Zd$ and is a uniformly elliptic operator with infinite range on the discrete lattice $\Zd$ (N.B. All these operators act on functions valued in $\R^{d \choose 2}$ and we are dealing here with \textit{systems of equations}).
\end{itemize}
We refer to Section~\ref{sec:HSrepre} for a more detailed description of these operators. A key insight from~\cite{HS} is that the properties of the random interface model can be effectively understood by analysing the behaviour of solutions to the Helffer–Sj\"{o}strand equation. In particular, in~\cite{NS, GOS}, the scaling limit of the model was identified by applying techniques of stochastic homogenization to the Helffer-Sj\"{o}strand operator. This approach has proven to be a fruitful line of inquiry for the model with nearest neighbor interactions and uniformly convex potential, leading to numerous significant contributions~\cite{DD05, DGI00, DG00, Mil, CS14} (and we refer to~\cite{BK07, CDM09, BS11, CD12, AT21} for some extensions of the theory to non-uniformly convex potentials). We specifically mention the recent contribution of Deuschel and Rodriguez~\cite{DR24} who derived isomorphism theorems for the $\nabla \phi$-model and identified the scaling limit of the square of the field (N.B. the question studied there is related to the one studied in this article due to the expansion $\cos \theta \simeq 1 - \theta^2/2$ for $\theta$ small and to the heuristic~\eqref{eq:backoftheenvelop} above, so that the $\nabla\phi$-model can be seen as a simple version of our setting).

Recently, the development of a quantitative theory of stochastic homogenization—initiated through distinct approaches by Gloria and Otto~\cite{GO1, GO2} and by Armstrong and Smart~\cite{AS}, and subsequently extended by various contributors~\cite{GO15, GO115, GNO, AKM1, AKMInvent, AKM, AT2022}, among many other references—has stimulated renewed quantitative progress on the $\nabla \varphi$ interface model~\cite{AW, AD, WuLCLT, AWSOS, dario2024hydrodynamic}, and the derivation of asymptotic expansions for the two-point function in the Villain model in dimension 
$d \ge 3$ ~\cite{DW}.

To illustrate the main ideas behind the proof of Theorem~\ref{th:mainth}, we introduce a simplified setting and consider the following toy model: suppose $\mu_\beta$ is a Gaussian free field in $\Zd$, $d\ge 3$, with covariance matrix $G$ (i.e., the Green's function on the lattice), and consider the covariance $\cov_{\mu_\beta} \left[  \varphi(0)^2; \cosh \varphi(x)\right]$ (this is an oversimplification of one term in our non-local observable \eqref{eq:09071202}). Then, it follows from two Gaussian integrations by parts that 
\begin{align*}
    \cov_{\mu_\beta} \left[  \varphi(0)^2 ; \cosh \varphi(x)\right]
    &= 
     G(0,x)  \langle \varphi(0) \sinh \varphi(x)\rangle_{\mu_\beta} \\
     &= 
     G(0,x)^2  \langle  \cosh \varphi(x)\rangle_{\mu_\beta}
      \leq
      CG(0,x)^2
      \approx \frac{C}{
      |x|^{2(d-2)}}. \notag
\end{align*}
The general strategy of the proof is to make sense of this heuristics in the case of the Villain model and a number of difficulties arise. First, $\varphi(0)$ and $\varphi(x)$ should be replaced by the (exponential of the) non-local observables $U_0$ and $U_x$, and their variants, defined in \eqref{eq:def3.5}. This requires a precise analysis of the contribution of various terms, which we carry out in Section \ref{sec:section3}. Second, in the case of a non-Gaussian measure $\mu_\beta$, the first step of the computation must be adapted. Specifically, the Helffer–Sj\"{o}strand representation must be employed to replace the Gaussian integration by parts, and the Green's function $G$ should be replaced by the Green's matrix $\mathcal{G}$ associated with the Helffer–Sj\"{o}strand operator (see Definition~\ref{eq:defHSGREEN} below). In contrast to the lattice Green's function $G$, the matrix $\mathcal{G}$ is random and depends on the random interface $\varphi \in \Omega$. Therefore, in the second line of the heuristic computation above, there are additional covariance terms, which turn out to be non-negligible. To estimate the covariance, one needs to compute the derivative $\partial \mathcal{G}$ with respect to the interface (see Section~\ref{sec:defWitten}). To study this function, we rely on the observation of Conlon and Spencer~\cite{CS14} that the derivative $\partial \mathcal{G}$ is a solution of the \textit{second-order Helffer-Sj{\"o}strand equation} (see Section~\ref{sec:secsecondorderHS} for more details in this direction), which can be studied using elliptic regularity estimates. 
The third difficulty is that the Helffer-Sj{\"o}strand equation and the second-order equation we study are elliptic \emph{systems of equations}. A number of properties which are valid for elliptic equations, and used to study the random interface models, are known to be false for elliptic systems. These includes the maximum principle, which is used to obtain a random walk representation, the De Giorgi-Nash-Moser regularity theory for uniformly elliptic and parabolic PDE (see~\cite[Section 8]{GT01}) and the Nash-Aronson estimate on the heat kernel (see~\cite{Ar}).

To resolve this lack of regularity, we make use of ideas from \emph{Schauder theory} (see~\cite[Chapter 6]{GT01}) and \emph{Calder\'{o}n-Zygmund theory} (see~\cite[Chapter 9]{GT01}); we leverage on the fact that the inverse temperature $\beta$ is chosen very large so that the elliptic operator $\mathcal{L}_{\mathrm{spat}}$ has a small ellipticity contrast, i.e., it can be written
\begin{equation*}
  \mathcal{L}_{\mathrm{spat}} := - \frac{1}{2\beta} \Delta + \mathcal{L}_{\mathrm{pert}},
\end{equation*}
where the operator $\mathcal{L}_{\mathrm{pert}}$ is a perturbative term; its typical size is of order $\beta^{-\frac 32} \ll\beta^{-1}$. One can thus prove that any solution $u$ of the equation $\mathcal{L}_{\mathrm{spat}} u = 0$ is well-approximated on every scale by a harmonic function $\bar u$ for which the regularity can be easily established. It is then possible to borrow the strong regularity properties of the function $\bar u$ and transfer them to the solution $u$. In this paper, in order to obtain the (almost) optimal upper bound for $d=3$ in Theorem \ref{th:mainth}, we need to utilise precise estimates to obtain cancellations in \eqref{e.1}, and we require the following two statements of regularity theory:
\begin{itemize}
    \item A $C^{0,1-\varepsilon}$ regularity for the solution to the Helffer-Sj{\"o}strand equation, which holds in $L^\infty$ stochastic integrability. This idea has been implemented in \cite{DW} and is used to obtain (optimal) decay estimates on the Green's matrix associated with the Helffer-Sj\"{o}strand operator (see Section~\ref{sec:sec263}).

    \item An optimal decay estimate for the gradient and mixed derivative of this Green's matrix in the $L^{p}$-annealed sense for some exponent $p >1$ depending on $d$ and $\beta$ and tending to infinity as $\beta$ tends to infinity (see Proposition \ref{prop.prop4.11chap4}).
\end{itemize}
The first regularity estimate relies on the Schauder regularity (and is only very briefly discussed in this article as it was extensively used in~\cite{DW}). The second regularity estimate is obtained by combining the arguments of Delmotte and Deuschel~\cite{DD05} and the Calder\'{o}n-Zygmund regularity for solutions of parabolic equations with a small ellipticity contrast, adapted to the setting of an infinite-range elliptic operator and systems of equations (see Appendix~\ref{app.CZreg}). These estimates play a crucial role in Section~\ref{sec:section3} and Section~\ref{sec:newsec4} to derive the upper bound of Theorem~\ref{th:mainth}.

We finally note, while most of the article is devoted to the proof of the upper bound of Theorem~\ref{th:mainth}, Section~\ref{sec:shortsec5} contains the proof of the lower bound and is based on a different technique (N.B. the same lower bound was proved in~\cite{BFL} but in the case of the XY model). The approach combines a correlation inequality of Dunlop and Newman~\cite{DN} (which holds for the XY model) together with the observation that the Villain model can be expressed as a \emph{metric graph limit} of XY models (see Section~\ref{sec:shortsec5} for details).

\medskip
\noindent
\textbf{Acknowledgments.} Part of this work was done when the second author was appointed as a visiting professor at Labex B\'ezout. We thank Labex B\'ezout and Universit\'e Paris-Cr\'eteil for their hospitality. We also thank S\'ebastien Ott and Tom Spencer for useful discussions. The work of W.W. was partially supported by MOST grant 2021YFA1002700 and NSFC grant 20220903 NYTP.

\section{Preliminaries}

This section contains the preliminary tools and results used in the proof of Theorem~\ref{th:mainth}. The technique used in the argument requires to introduce a certain amount of notation and results, we therefore encourage the reader to initially skim this section and refer back to it as a reference. These ingredients can nevertheless be classified into three categories:
\begin{itemize}
    \item \textbf{Discrete differential forms:} As mentioned above, the proof starts from the observation of Fr\"{o}hlich and Spencer~\cite{FS4d} (see also~\cite[Section 5.5]{Bau}) that a duality transform can be used to rewrite the correlation of the Villain model (in particular the cosine-cosine correlation) as the expectation of a certain observable of a vector-valued random interface model. In dimension $d \geq 3$, the natural mathematical object one needs to use to perform this duality transformation is the notion of discrete differential form which is presented (together with some of the results and tools associated with this notion) in the first four subsections.
    \medskip
    \item \textbf{Random interfaces and the Helffer-Sj\"{o}strand representation:} One of the crucial advantage of working with random interfaces (rather than studying directly the Villain model) is that they can be studied (very) precisely by using a tool known as the Helffer-Sj\"{o}strand representation~\cite{HS, Sj, NS, GOS}. In Sections~\ref{sec:randominterfaces} and~\ref{sec:HSrepre} below, we present a brief introduction to this interface model (N.B. this model is a vector-valued version of the $\nabla \varphi$-interface model which has been widely studied in the mathematical physics literature, and we refer to~\cite{F05, V06} for a detailed presentation of the model) as well as the Helffer-Sj\"{o}strand representation.
    \medskip
    \item \textbf{Elliptic and parabolic regularity:} The Helffer-Sj\"{o}strand representation is a powerful tool to study random interface model because it can combined with results of elliptic and parabolic regularity. In Proposition~\ref{prop.prop4.11chap4}, we collect some applications of the theory by stating some upper bounds on the Green's matrix associated with the Helffer-Sj\"{o}strand equation.
\end{itemize}

\subsection{General notation} \label{sec:generalnotation}

We fix a dimension $d \geq 3$, denote by $\Zd$ the lattice and let $\{e_1 , \ldots, e_d\}$ be the canonical basis of $\Rd$. For $x , y \in \Zd$, we write $x \sim y$ to mean that $x$ and $y$ are nearest neighbour in $\Zd$.
For $L \in \N$, we denote by $\Lambda_L := \left\{ -L , \ldots, L \right\}^d$ the finite box centred at $0$ and of side length $(2L+1)$. We then define a box to be a set of the form $\Lambda := (x + \Lambda_L)$ for some $x \in \Zd$ and $L\in \N$, and denote by $\left| \Lambda \right| := (2L+1)^d$ the cardinality of the box $\Lambda$. For $k \in \N$, we denote by $\left| \cdot \right|$ the Euclidean norm on $\R^k$ and set $\left| \cdot \right|_+ := \left| \cdot \right| + 1.$

We next introduce $E(\Zd)$, $F(\Zd)$ and $C(\Zd)$ the sets of edges, faces and cubes of $\Zd$, i.e., 
\begin{equation*}
     \begin{aligned}
         E(\Zd) & := \left\{x + \lambda_1 e_{i_1}  \in \Rd \, : \, x \in \Zd, \, 0 \leq \lambda_1 \leq 1, 1 \leq i_1  \leq d \right\}, \\
         F(\Zd) & := \left\{x + \lambda_1 e_{i_1} + \lambda_2 e_{i_2} \in \Rd \, : \, x \in \Zd, \, 0 \leq \lambda_1 ,\lambda_2 \leq 1, 1 \leq i_1 < i_2 \leq d \right\}, \\
         C(\Zd) & := \left\{x + \lambda_1 e_{i_1} + \lambda_2 e_{i_2} + \lambda_3 e_{i_3}  \in \Rd \, :  \, x \in \Zd, \, 0 \leq \lambda_1 ,\lambda_2 , \lambda_3 \leq 1, 1 \leq i_1 < i_2 < i_3 \leq d  \right\}.
     \end{aligned}
\end{equation*}
We then equip these sets with an orientation induced by the canonical orientation of the lattice $\Zd$ and denote by $\vec{E}(\Zd)$, $\vec{F}(\Zd)$ and $\vec{C}(\Zd)$ the set of oriented edges, faces and cubes of the lattice $\Zd$. Given an oriented edge $\vec{e} \in \vec{E}(\Zd)$ (resp. an oriented face $\vec{f} \in \vec{F}(\Zd)$ and an oriented cube $\vec{c} \in \vec{C}(\Zd)$), we denote by $\partial \vec{e}$ (resp. $\partial \vec{f}$, $\partial \vec{c}$) the boundary of the edge $\vec{e} \in \vec{E}(\Zd)$; it can be decomposed into a disjoint union of vertices (resp. of edges and faces). In the case of an oriented face $\vec{f}$, the orientation of $\vec{f} \in \vec{F}(\Zd)$ induces an orientation on $\partial \vec{f}$, i.e., $\partial \vec{f}$ is a union of oriented edges. Similarly, given an oriented cube $\vec{c} \in \vec{C}(\Zd)$, the orientation of $\vec{c}$ induces an orientation on $\partial \vec{c}$, i.e., $\partial \vec{c}$ is a union of oriented faces.

\subsection{Discrete differential forms}
In this section, we introduce some basic definitions of discrete differential forms. Specifically, we will introduce $0$-forms, $1$-forms and $2$-forms on $\Zd$, as well as the discrete versions of the exterior derivative and the codifferential.

\subsubsection{$0$-form (or functions)}

A $0$-form is a function $u : \Zd \to \R$ (i.e.,  a function defined on the vertices of the lattice $\Zd$). For arbitrary $0$-forms (or functions) $u, v : \Zd \to \R $ with finite support, we define the scalar product $(\cdot , \cdot)$ by the formula
\begin{equation*}
  (u , v) = \sum_{x \in \Zd} u(x) v(x).
\end{equation*}
Given a function $u : \Zd \to \R$, we define the discrete gradient and Laplacian of $u : \Zd \to \R$ according to the identities
\begin{equation} \label{eq:def.grad}
    \nabla u (x) := \left( u(x + e_i) - u(x) \right)_{1 \leq i\leq d} \in \Rd  ~\mbox{and}~  \Delta u(x) := \sum_{y \sim x} \left( u(y)  - u(x)\right) \in \R
\end{equation}
and we will use the notation for the norm of the gradient, for any $x \in \Zd$
\begin{equation*}
    \left| \nabla u (x) \right| = \sqrt{\sum_{i = 1}^d \left| u(x + e_i) - u(x) \right|^2}.
\end{equation*}
For $x \in \Zd$, we denote by $\delta_x : \Zd \to \R$ the discrete Delta function, i.e., the one satisfying $\delta_x(x) = 1$ and $\delta_x(y)= 0$ for $y \in \Zd \setminus \{x\}.$

\subsubsection{$1$-form (or vector fields)}

A $1$-form (or a vector field) is a function $h : \vec{E}(\Zd) \to \R$ defined on the oriented edges of the lattice $\Zd$ and satisfying the identity, 
\begin{equation*}
    \forall \vec{e} \in \vec{E}(\Zd), ~~ h(\vec{e}) = - h(-\vec{e}),
\end{equation*}
where $- \vec{e}$ denotes the edge $\vec{e}$ with reversed orientation.
For $i \in \{1 , \ldots, d\}$, we denote by $\vec{e}_i$ the edge $\{\lambda e_i \, : \, 0 \leq \lambda \leq 1\}$ with positive orientation.
A $1$-form can be equivalently defined as a function defined on the vertices on $\Zd$ and valued in $\R^d$ (i.e., $h : \Zd \to \Rd$) by using the bijection
\begin{equation} \label{eq:adefforh}
    \forall x \in \Zd, ~ h(x) := \left( h(x + \vec{e}_i) \right)_{1 \leq i \leq d} \in \Rd ~\mbox{and}~ \Delta h(x) := \sum_{y \sim x} \left( h(y)  - h(x)\right) \in \R^d.
\end{equation}
A typical example of $1$-form is the gradient of a function (as defined in~\eqref{eq:def.grad}).

For arbitrary $1$-forms $h, g : \Zd \to \R^d$ with finite support, we define the $L^2$-scalar product $(\cdot , \cdot)$ and the norm $\left| \cdot \right|$ by the formulae
\begin{equation*} 
  (h , g) = \sum_{x \in \Zd} h(x) \cdot g(x) ~~\mbox{and}~~ |h(x)| = \sqrt{\sum_{i = 1}^d \left| h(x + \vec{e}_i) \right|^2}
\end{equation*}
where $h(x) \cdot g(x)$ denotes the Euclidean scalar product between $h(x) , g(x) \in \Rd$. We further define the $L^1$ and $L^\infty$-norms of a $1$-form $h$ by the formulae
\begin{equation*}
     \left\| h \right\|_1 = \sum_{x \in \Zd} |h(x)| ~~\mbox{and}~~ \left\| h  \right\|_\infty = \sup_{x \in \Zd} |h(x)|.
\end{equation*}
Finally, we define the discrete gradient of the $1$-form $h$ according to the formula (N.B. since for each $x \in \Zd$, $h(x) \in \Rd$, we have that $\nabla h(x) \in \R^{d \times d}$, i.e., it is a matrix-valued function; additionally $h(x + e_i)$ is the value of the function $h$ at the vertex $(x + e_i) \in \Zd$)
\begin{equation*}
    \nabla h(x) = \left( h(x + e_i) - h(x) \right)_{1 \leq i\leq d} \in \R^{d \times d}.
\end{equation*}

\subsubsection{2-forms (or functions defined on faces)}

We recall the notation $\vec{F}(\Zd)$ for the oriented faces of the lattice $\Zd$. A $2$-form is then a function $k : \vec{F}(\Zd) \to \R$ satisfying the identity
\begin{equation*}
    \forall \vec{f} \in \vec{F}(\Zd), ~ k(\vec{f}) = - k( - \vec{f}),
\end{equation*}
where $- \vec{f}$ denotes the face $\vec{f}$ with reversed orientation.
We next define the support of a $2$-form $k$ to be the collection of all the faces on which it takes a non-zero value, i.e.,
\begin{equation*}
    \supp k := \left\{ \vec{f} \in \vec{F}(\Zd) \, : \, k(\vec{f}) \neq 0 \right\}.
\end{equation*}
For $i , j \in \{1 , \ldots, d\}$ with $i < j$, we denote by $\vec{f}_{ij}$ the face $\{\lambda_1 e_i + \lambda_2 e_j \, : \, 0 \leq \lambda_1 , \lambda_2 \leq 1\}$ with positive orientation.
A $2$-form can be equivalently defined as a function defined on the vertices on $\Zd$ and valued in $\R^{\binom d2}$ (i.e., $k : \Zd \to \R^{\binom d2}$) by using the bijection
\begin{equation*}
    \forall x \in \Zd, ~ k(x) := \left( k(x + \vec{f}_{ij}) \right)_{1 \leq i < j \leq d} \in \Rd.
\end{equation*}
For arbitrary $2$-forms $k, l : \Zd \to \R^d$ with finite support, we define the $L^2$-scalar product $(\cdot , \cdot)$ and the norm $\left| \cdot \right|$ according to the formulae
\begin{equation} \label{eq:TV07319}
  (k , l) = \sum_{x \in \Zd} k(x) \cdot l(x) ~~\mbox{and}~~ |k(x)| = \sqrt{\sum_{1 \leq i < j \leq d} \left| k(x + \vec{f}_{ij}) \right|^2}
\end{equation}
where $k(x) \cdot l(x)$ denotes the Euclidean scalar product between $k(x) , l(x) \in \R^{\binom d2}$. We further define the $L^1$ and $L^\infty$-norms of a $1$-form $h$ by the formulae
\begin{equation*}
     \left\| k \right\|_1 = \sum_{x \in \Zd} |k(x)| ~~\mbox{and}~~ \left\| k  \right\|_\infty = \sup_{x \in \Zd} |k(x)|.
\end{equation*}
We then define the discrete gradient  and Laplacian of the $2$-form $k$ according to the formulae (N.B. since for each $x \in \Zd$, $k(x) \in \R^{\binom d2}$, we have that $\nabla k(x) \in \R^{d \times \binom d2}$)
\begin{equation*}
    \nabla k(x) = \left( k(x + e_i) - k(x) \right)_{1 \leq i\leq d} \in \R^{d \times \binom d2} ~\mbox{and}~  \Delta k(x) := \sum_{y \sim x} \left( k(y)  - k(x)\right) \in \R^{\binom d2}.
\end{equation*}
In Section~\ref{sec:HSrepre} (and specifically when we need to use the second-order Helffer-Sj\"{o}strand equation in Section~\ref{sec:newsec4}), we will need to use the notion of tensor product of $2$ forms. To this end, we will make use of the tensor notation $\otimes$ defined as follows: given two $2$-forms $k , l : \Zd \to \R^{\binom d2}$, we denote by
\begin{equation} \label{def:tensorproduct}
    k \otimes l: \left\{ \begin{aligned}
        \Zd \times \Zd & \to \R^{\binom d2 \times \binom d2} \\
        (x,y) & \mapsto (k_{ij}(x) l_{i'j'}(y) )_{1 \leq i < j \leq d, 1 \leq i' < j' \leq d}.
    \end{aligned} \right.
\end{equation}

\subsubsection{3-forms (or functions defined on cubes)}
We recall the notation $\vec{C}(\Zd)$ for the oriented cubes of the lattice $\Zd$. A $3$-form is then a function $m : \vec{C}(\Zd) \to \R$ satisfying the identity
\begin{equation*}
    \forall \vec{c} \in \vec{C}(\Zd), ~ m(\vec{c}) = - m( - \vec{c}),
\end{equation*}
where $- \vec{c}$ denotes the cube $\vec{c}$ with reversed orientation 
(N.B. we only need this definition for a $3$-form, but it would of course be possible to extend the definitions of the previous sections to the $3$-forms).

\subsubsection{Exterior derivative and codifferential}
In this section, we introduce two (fundamental) objects related to the (discrete) differential calculus: the exterior derivative and the codifferential. They are defined in this section in a somewhat restrictive setting (i.e., only for the $1$ and $2$ forms) but we mention that more general versions of these definitions and statements have been obtained in the literature.

\begin{definition}[Discrete exterior derivative for $1$ and $2$-forms]
    Given a $1$-form $h : \vec{E}(\Zd) \to \R$, we define its exterior derivative $\di h$ to be the $2$-form defined according to the formula, for each oriented face $\vec{f} \in \vec{F}(\Zd)$,
\begin{equation*}
\di h \left( \vec{f} \right) = \sum_{\vec{e} \in \partial \vec{f}} h(\vec{e}),
\end{equation*}
where we used the convention introduced in Section~\ref{sec:generalnotation} for the orientation of the edges on the right-hand side.
Similarly, given a $2$-form $k : \vec{E}(\Zd) \to \R$, we define its exterior derivative $\di k$ to be the $3$-form defined according to the formula, for each oriented cube $\vec{c} \in \vec{C}(\Zd)$,
\begin{equation*}
\di k \left( \vec{c} \right) = \sum_{\vec{f} \in \partial \vec{c}} k(\vec{f}).
\end{equation*}
\end{definition}

We similarly define the (discrete) codifferential of a $2$-form (N.B. As mentioned below, this operator is the adjoint of the exterior derivative with respect to the $L^2$-scalar product $(\cdot , \cdot)$).

\begin{definition}[Codifferential of a $2$-forms] \label{def:codifferential2form}
    Given a $2$-form $k : \vec{F}(\Zd) \to \R$, we define its codifferential $\di^* k$ to be the $1$-form given by the formula, for each oriented edge $\vec{e} \in \vec{E}(\Zd)$,
\begin{equation*}
\di^* k \left( \vec{e} \right) = \sum_{\partial \vec{f} \ni \vec{e}} k(\vec{f}),
\end{equation*}
where we used the convention introduced in Section~\ref{sec:generalnotation} for the orientation of the faces on the right-hand side.
\end{definition}

\begin{remark} \label{rem:remark25}
    Let us fix a $2$-form $k : \vec{F}(\Zd) \to \R$ and make two remarks about the previous definitions:
        \begin{itemize}
            \item The exterior derivative $\di k$ and the codifferential $\di^* k$ are both linear combinations of the components of the gradient $\nabla k$.
            \item If the form $k$ satisfies $\di k =0$, then we will say that it is \emph{closed}.
            \item The codifferential is the adjoint of the exterior derivative: for any $1$-form $h :  \vec{E}(\Zd) \to \R$ with finite support, one has the identity
\begin{equation} \label{identity:dandd*}
    \left( \di h , k \right) = \left( h , \di^* k \right).
\end{equation}
(N.B. Note that the $L^2$-scalar product on the left-hand side is the one on $2$-forms while the one on the right-hand side is the one on $1$-forms).
        \end{itemize}
\end{remark}

\subsection{The set of charges $\mathcal{Q}$}

In the rest of this article, an important role is played by the set of closed $2$-forms which are integer-valued and finitely supported with connected support (N.B. they will be referred to as charges). We denote this set by $\mathcal{Q}$ and formally define it below (N.B. for the connectedness of the support, we say that $2$ faces of $\Zd$ are connected if they are at distance less than $1$ from each other).
\begin{equation*}
    \mathcal{Q} := \left\{ q : \Zd \to \Z^{\binom d2} \, : \, \di q = 0 ~~\mbox{and}~~ \supp q \mbox{ is connected and finite}   \right\}.
\end{equation*}
Given a charge $q \in \mathcal{Q}$, we denote $z_q \in \Zd$ the vertex which is in the support of $q$ and which minimises the lexicographical order (N.B. any other arbitrary criterion to select a vertex in the support of $q$ is admissible). We also let $\Lambda(q)$ be the intersection of all the boxes containing the support of $q$.

\begin{lemma}[Discrete Poincar\'e Lemma, Lemma 1 of~\cite{FS4d} or Lemma 2.2 of~\cite{Ch}] \label{lem:lem2.4}
    There exists a constant $C := C(d) < \infty$ such that for each $q \in \mathcal{Q}$, there exists an integer-valued $1$-form $n_q : \vec{E}(\Zd) \to \Z$ satisfying the following properties 
    \begin{equation} \label{eq:critnq}
        \di n_q = q, ~~ \supp \, n_q \subseteq \Lambda(q)~~ \mbox{and}~~ \left\| n_q \right\|_\infty \leq C \left\| q \right\|_1.
    \end{equation}
\end{lemma}

\begin{remark}
Let us make a few remarks about the previous lemma:
\begin{itemize}
    \item As stated above, the form $n_q$ is not uniquely defined by the three criteria above. In the rest of this article, for each $2$-form $q \in \mathcal{Q}$, we select a form $n_q$ satisfying~\eqref{eq:critnq} (and break ties using an arbitrary criterion).
    \item It is important to note that both $q$ and $n_q$ are integer-valued.
    \item We will make use of the following inequality (which can be established using that the support of a $2$-form $q \in \mathcal{Q}$ is connected):
    \begin{equation} \label{eq:somepropofcharges}
        \diam q \leq C \left\| q \right\|_1.
    \end{equation}
\end{itemize}
\end{remark}
Given $\beta \in (0,\infty)$ and a $2$-form $q \in \mathcal{Q}$, we define the activity of $q$ according to the formula (see~\cite[(5.71)]{Bau})
\begin{equation*}
    z(\beta , q) = \sum_{n = 1}^\infty \frac{1}{n !} I(G (\supp q_1 , \ldots, \supp q_n)) \sum_{q_1 + \ldots + q_n = q} e^{- \frac 12  \sqrt{\beta} \sum_i (q_i , q_i)},
\end{equation*}
where the sum runs over all the charges $q_1 , \ldots , q_n$ with connected support satisfying $\di q_i =0$, and the combinatorial factor $I(G (\supp q_1 , \ldots, \supp q_n))$ is defined as follows: we let $G(\supp q_1 , \ldots, \supp q_n)$ be the connection graph of the sets $\supp q_1 , \ldots, \supp q_n$ (i.e., the graph whose vertices are $\supp q_1 , \ldots, \supp q_n$, and with an edge between $\supp q_i$ and $\supp q_j$ if and only the distance between the two sets is less or equal to $1$), and for a connected graph $G$, we define
\begin{equation*}
    I(G ) := \sum_{H \subseteq G } (-1)^{\left| E(H)\right|},
\end{equation*}
where the sum runs over all the connected spanning subgraphs of $G$. These definitions and formulae are the ones of~\cite[Section 5.5.3]{Bau}. By~\cite[Lemma 5.15]{Bau}, one has the estimate
\begin{equation*}
  |z(\beta , q)| \leq e^{-c \sqrt{\beta} \|q\|_1}, \quad \mbox{for some } c :=c(d) > 0.
\end{equation*}
(N.B. To be precise, this definition is not exactly the same as the one of~\cite[(5.71)]{Bau}: for a technical, and not particularly important, reason, we choose the small constant $c$ of~\cite[(5.71)]{Bau} to be equal to $1/\sqrt{\beta}$).

We next state and prove an important summation property which is used multiple times in the proofs below.

\begin{lemma}[Summation over the $2$-forms of $\mathcal{Q}$] \label{lemma.lemma2.5}
    Fix an integer $k \in \N$. There exists a constant $C := C(d , \beta , k) < \infty$ such that, for any function $h : \Zd \to [0, \infty)$,
    \begin{equation*}
        \sum_{q \in \mathcal{Q}} |z(\beta, q)| \left\| q \right\|_1^k h(z_q) \leq C \sum_{z \in \Zd} h(z).
    \end{equation*}
\end{lemma}

\begin{remark} \label{rem:remark2.6}
This inequality will be used to obtain the following statement:
   if we let $H_0, H_1 : \vec{E}(\Zd) \to \R$ be two $1$-forms which decay at most polynomially fast (with potentially a logarithmic correction), i.e., there exist a constant $C < \infty$, four exponents $k_0 , k_1, \kappa_0, \kappa_1 \in [0 , \infty)$ and a vertex $x \in \Zd$ such that
    \begin{equation} \label{assumption.polydecay}
        \forall z \in \Zd, ~ \left| H_0(z) \right|  \leq \frac{C \left(\ln |z|_+ \right)^{\kappa_0}}{|z|^{k_0}_+} ~~\mbox{and}~~ \left| H_1 (z) \right| \leq \frac{C \left(\ln |x - z|_+ \right)^{\kappa_1}}{|x - z|^{k_1}_+},
    \end{equation}
    then one has the inequality
    \begin{equation} \label{ineq:2.11technical}
        \sum_{q \in \mathcal{Q}} z(\beta, q) \left| \left(H_0 , n_q \right) \right| \left| \left(H_1 , n_q \right) \right| \leq C \sum_{z \in \Zd} \frac{\left(\ln |z|_+ \right)^{\kappa_0}}{|z|^{k_0}_+} \frac{\left(\ln |x - z|_+ \right)^{\kappa_1}}{|x - z|^{k_1}_+}.
    \end{equation}
    In words, the inequality~\eqref{ineq:2.11technical} says that if $H_0$ and $H_1$ are two functions which decay polynomially fast, then summing the scalar products $\left(H_0 , n_q \right)$ and $\left(H_1 , n_q \right)$ over all the charges $q \in \mathcal{Q}$ is (roughly) equivalent to summing the values $\left(\ln |z|_+ \right)^{\kappa_0} |z|^{-k_0} \left(\ln |x - z|_+ \right)^{\kappa_1} |x - z|^{-k_1}$ over all the vertices $z \in \Zd$. This is an important simplification (which will be used many times in the proof below) since the sum over the vertices of $\Zd$ is easier to estimate than the sum over all the $2$-forms of $\mathcal{Q}$.
\end{remark}

\begin{remark}
In the rest of this article, we may denote by $C_q$ a constant which depends on $d , \beta$ and $q$ but shall only grow polynomially fast in $\left\| q \right\|_1$ (i.e., we assume that there exists a constant $C := C(d , \beta)< \infty$ depending only on $d$ and $\beta$ and an exponent $k := k(d) < \infty$ depending only on $d$ such that $C_q \leq C \left\| q \right\|_1^k$). This notation will be typically used to write inequalities of the form
\begin{equation*}
    \left| \left(H_0 , n_q \right) \right| \leq \frac{C_q \left(\ln |z|_+\right)^{\kappa_0} }{|z|^{k_0}_+},
\end{equation*}
when $H_0$ is a $1$-form satisfying~\eqref{assumption.polydecay}.
\end{remark}

\begin{proof}[Proof of Lemma~\ref{lemma.lemma2.5}]
We first use the Cauchy-Schwarz inequality and write
\begin{align*}
 \sum_{q \in \mathcal{Q}} |z(\beta, q)| \left\| q \right\|_1^k h(z_q) \leq 
 \sum_{q \in \mathcal{Q}} e^{-c \sqrt{\beta} \left\| q \right\|_1}\left\| q \right\|_1^k h(z_q)  & = \sum_{y \in \Zd} \left( \sum_{\substack{q \in \mathcal{Q} \\ z_q = y}} e^{-c \sqrt{\beta} \left\| q \right\|_1}  \left\| q \right\|_1^k \right)  h(y)  .
\end{align*}
It is thus sufficient to show that there exists a constant $C := C(\beta, d , k) < \infty$ such that, for any $y \in \Zd$,
\begin{equation*}
    \sum_{\substack{q \in \mathcal{Q} \\ z_q = y}} e^{-c \sqrt{\beta} \left\| q \right\|_1} \left\| q \right\|_1^k \leq C.
\end{equation*}
The rest of the argument is devoted to the proof of this inequality. We first observe that, the polynomial factor $\left\| q \right\|_1^k$ can be absorbed in the exponential decay (at the cost of slightly reducing the exponent in the exponential). Specifically, we have that
\begin{equation*}
     e^{-c \sqrt{\beta} \left\| q \right\|_1} \left\| q \right\|_1^k \leq C e^{-\frac{c}{2} \sqrt{\beta} \left\| q \right\|_1}.
\end{equation*}
We then decompose over the supports of the charges. To this end, let us denote by $\mathcal{A}_y$ the set of the finite connected subsets of faces of $\Zd$ containing the vertex $y$. We then write, for some $c' < c$,
\begin{align*}
    \sum_{\substack{q \in \mathcal{Q} \\ z_q = y}} e^{-c' \sqrt{\beta} \left\| q \right\|_1} & \leq \sum_{X \in \mathcal{A}_y} \sum_{\substack{q \in \mathcal{Q} \\ \supp q = X}} e^{-c' \sqrt{\beta} \left\| q \right\|_1} = \sum_{X \in \mathcal{A}_y} \sum_{\substack{q \in \mathcal{Q} \\ \supp q = X}} \left( \prod_{x \in X } e^{-c' \sqrt{\beta} \left| q(x) \right|} \right).
\end{align*}
Exchanging the sum and the product, we see that
\begin{align*}
        \sum_{\substack{q \in \mathcal{Q} \\ \supp q = X}} \left( \prod_{x \in X } e^{-c' \sqrt{\beta} \left| q(x) \right|} \right) & \leq \prod_{x \in X } \left( \sum_{q(x) = 1}^\infty   e^{-c' \sqrt{\beta} \left| q(x) \right|} \right) = \left( \frac{e^{-c' \sqrt{\beta}}}{1- e^{-c' \sqrt{\beta}}} \right)^{|X|}.
\end{align*}
We thus obtain
\begin{equation*}
    \sum_{q \in \mathcal{Q}_y} e^{-c' \sqrt{\beta} \left\| q \right\|_1} \leq C \sum_{X \in \mathcal{A}_y} e^{-c' \sqrt{\beta} |X|} = C \sum_{n = 1}^\infty \left| \left\{ X \in \mathcal{A}_y \, : \, \left|  X \right| = n \right\} \right| e^{-c' \sqrt{\beta} n }.
\end{equation*}
We next note that
\begin{equation} \label{eq:110754}
    \left| \left\{ X \in \mathcal{A}_y \, : \, \left|  X \right| = n \right\} \right| \leq e^{ C n}.
\end{equation}
The inequality~\eqref{eq:110754} can be established by associating each connected set of $n$ faces with one of its spanning trees, and then bounding the number of such spanning trees.
Choosing the inverse temperature $\beta$ large enough (i.e., such that $\frac{c}{4} \sqrt{\beta} \geq 2C$), we deduce that
\begin{equation*}
    \sum_{q \in \mathcal{Q}_y} e^{-\frac{c}{2} \sqrt{\beta} \left\| q \right\|_1} \leq C e^{-\frac{c}{4} \sqrt{\beta}},
\end{equation*}
where we have reduced the value of the exponent $c$ in the right-hand side.
\end{proof}

We complete this section by giving a proof of the inequality~\eqref{ineq:2.11technical} based on Lemma~\ref{lemma.lemma2.5}.

\begin{proof}[Proof of~\eqref{ineq:2.11technical}] 
To simplify the notation in the proof, we only show the result in the case $\kappa_0 = \kappa_1 = 0$. Let us first fix a charge $q \in \mathcal{Q}$ and start from the observation that the support of $n_q$ is included in the box $\Lambda(q)$. This allows to write
    \begin{align*}
        \left| \left( H_0 , n_q \right) \right| \leq \sum_{y \in \Lambda(q)} \left| H_0(y)\right| \left| n_q(y) \right| \leq \sum_{y \in \Lambda(q)} \left| H_0(y)\right| \left\| n_q \right\|_\infty & \leq C \left\| q \right\|_1 \sum_{y \in \Lambda(q)} \left| H(y)\right| .
    \end{align*}
    Using the assumption~\eqref{assumption.polydecay} (and since $\Lambda(q)$ has a diameter which is comparable to the diameter of~$q$), we see that the following inequality holds
    \begin{equation*}
        \forall y \in \Lambda(q), ~~ \left| H_0(y)\right| \leq \frac{C (\diam q)^{k_0}}{\left| z_q \right|^{k_0}_+}.
    \end{equation*}
    Combining the two previous inequalities with~\eqref{eq:somepropofcharges} and noting that the cardinality of $\Lambda(q)$ is comparable to $(\diam q)^d$, we further deduce that 
    \begin{equation*}
        \left| \left( H_0 , n_q \right) \right| \leq \left\| q \right\|_1 \sum_{y \in \Lambda(q)} \left| H_0(y)\right| \leq  \frac{\left\| q \right\|_1 \left| \Lambda(q) \right| (\diam q)^{k_0}}{\left| z_q \right|^{k_0}_+} \leq \frac{C \left\| q \right\|_1^{d + 1  + k_0} }{\left| z_q \right|^{k_0}_+} \leq \frac{C_q}{\left| z_q \right|^{k_0}_+}.
    \end{equation*}
    Writing the same computation with the function $H_1$ instead of the function $H_0$, we deduce that
    \begin{equation*}
        \left| \left( H_1 , n_q \right) \right| \leq \frac{C \left\| q \right\|_1^{d + 1  + k_1} }{\left| x -  z_q \right|^{k_1}_+} \leq \frac{C_q}{\left| z_q \right|^{k_1}_+}.
    \end{equation*}
    We then combine the two previous inequalities with Lemma~\ref{lemma.lemma2.5} (with $k := 2d + 2  + k_0+k_1$) to obtain that
    \begin{equation*}
        \sum_{q \in \mathcal{Q}} |z(\beta, q)| \left| \left(H_0 , n_q \right) \right| \left| \left(H_1 , n_q \right) \right| \leq C \sum_{q \in \mathcal{Q}} z(\beta , q) \left\| q \right\|_1^{2d + 2  + k_1+k_2}  \frac{1}{\left| z_q \right|^{k_0}_+  |x - z_q|_+^{k_1}} \leq C \sum_{z \in \Zd} \frac{1}{|z|^{k_0}_+} \frac{1}{|x - z|^{k_1}_+}.
    \end{equation*}
\end{proof}

\subsection{The discrete Green's function}

In this section, we introduce the discrete Green's function on the lattice $\Zd$ as well as some of its elementary properties (regarding its decay and the one of its derivative).

\begin{definition}[Green's function on the lattice $\Zd$] \label{def:defgreenlattice}
    Fix a dimension $d \geq 3$, then there exists a unique function $G : \Zd \to \R$ that satisfies the two properties
    \begin{equation*}
        - \Delta G = \delta_0 ~~\mbox{on}~~\Zd ~~\mbox{and} ~~ G(x) \underset{|x| \to \infty}{\longrightarrow} 0.
    \end{equation*}
\end{definition}

\begin{proposition}[Chapter 4 of~\cite{lawler2010random}] \label{prop:greenfunction}
    The Green's function $G$ satisfies the properties: \begin{itemize}
        \item Decay of the Green's function: There exist two constants $c := c(d) > 0$ and $C := C(d) < \infty$ such that, for any $x \in \Zd$,
        \begin{equation*}
            \frac{c}{|x|^{d-2}_+} \leq G(x) \leq \frac{C}{|x|^{d-2}_+}.
        \end{equation*}
        \item Decay of the gradient of the Green's function: There exists a constant $C := C(d) < \infty$ such that, for any $x \in \Zd$,
        \begin{equation*}
            \left| \nabla G(x) \right| \leq \frac{C}{|x|^{d-1}_+}.
        \end{equation*}
    \end{itemize}
\end{proposition}

\subsection{The random interface measure $\mu_\beta$} \label{sec:randominterfaces}

This section is devoted to the introduction of the random interface model (see the measure $\mu_\beta$ below) which appears when the duality transform of~\cite{FS4d} is applied to the Villain model (at low temperature) together with the Helffer-Sj\"{o}strand representation formula. We mention that we will make use of the definitions of the previous sections.

\subsubsection{Definitions} \label{sec:defWitten}

Given an integer $L \in \N$, we introduce the finite dimensional vector space of $2$-forms 
\begin{equation*}
    \Omega_L := \left\{ \varphi : \Zd \to \R^{d \choose 2} \, : \, \varphi = 0 ~\mbox{in}~ \Zd \setminus \Lambda_L  \right\}.
\end{equation*}
We equip the finite-dimensional vector space $\Omega_L$ with the Borel $\sigma$-algebra, the scalar product~\eqref{eq:TV07319} and the Lebesgue measure associated with this scalar product. We then introduce the (infinite-dimensional) vector space of real-valued $2$-forms, i.e.,
\begin{equation*}
    \Omega := \left\{ \varphi : \Zd \to \R^{d \choose 2}  \right\},
\end{equation*}
and equip this space with the $\sigma$-algebra generated by the projections. We next introduce the set of smooth local and compactly supported functions of the set $\Omega$
\begin{multline*}
    C^\infty_c(\Omega) := \left\{ F : \Omega \to \R \, : \, \exists \, n \in \N, \exists \, x_1, \ldots, x_n \in \Zd ~\mbox{and}~ \exists \, f \in \C^\infty_c(\R^n) \right. \\ \left. ~\mbox{such that} ~~ F(\varphi) = f(\varphi (x_1) , \ldots , \varphi (x_n)) \right\}.
\end{multline*} 
Given a smooth function $F \in C^\infty_c(\Omega)$, a vertex $x \in \Zd$ and an integer $i \in \{1 , \ldots, d\}$, we define the partial derivative $\partial_{x , i}$ according to the identity: for any $\varphi \in \Omega$,
\begin{equation*}
    \partial_{x , i} F(\varphi) := \lim_{\ep \to 0} \frac{F(\varphi + \ep e_i \delta_{x} ) -  F(\varphi)}{\ep},
\end{equation*}
and further set
\begin{equation*}
    \partial_{x} F(\varphi) := \left( \partial_{x , 1} F(\varphi) , \ldots, \partial_{x , d} F(\varphi) \right) \in \Rd.
\end{equation*}
Given a function $F \in C^\infty_c(\Omega),$ we denote by $\partial F : \Zd \times \Omega \to \R^{\binom d2}$ the function defined by the formula $\partial F : (x , \varphi) \mapsto \partial_{x} F(\varphi) \in \R^{\binom d2}.$

For $k \in \N$, we extend the previous notation to vector-valued functions $F : \Omega \to \R^k$ and write $F \in C^\infty_c(\Omega \, ; \, \R^k)$ if all the components of $F$ belong to $C^\infty_c(\Omega)$.

We next introduce the space of smooth and compactly supported functions defined on $\Zd \times \Omega$
\begin{multline*} 
    C^\infty_c( \Zd \times \Omega \, ; \,  \R^{d \choose 2} ) \\
    := \left\{ F : \Zd \times \Omega \to \R^{\binom d2} \, : \, \forall x \in \Zd, \, F(x , \cdot) \in C^\infty_c(\Omega) ~\mbox{and}~ F(x , \cdot) = 0 ~\mbox{for all but finitely many}~ x \in \Zd\ \right\}.
\end{multline*}
This space is defined so that the following property is satisfied: for any $F \in C^\infty_c(\Omega)$, one has $\partial F \in C^\infty_c( \Zd \times \Omega \, ; \, \R^{d \choose 2}).$

We extend the definitions for the partial derivatives $\partial_{x , i}, \partial_x$ and $\partial$ to the functions of $C^\infty_c( \Zd \times \Omega \, ; \, \R^{d \choose 2} )$. Given a function $F = (F_1 , \ldots, F_{\binom d2}) \in C^\infty_c( \Zd \times \Omega \, ; \, \R^{d \choose 2} ),$ we denote by $\partial F : \Zd \times \Zd \times \Omega \to \R^{\binom d2 \times \binom d2}$ the function defined by the formula $\partial F : (x , y, \varphi) \mapsto \left( \partial_{x} F_1(y, \varphi) , \ldots, \partial_{x} F_{\binom d2}(y, \varphi) \right) \in \R^{\binom d2 \times \binom d2}.$ Finally, for any function $F \in C^\infty_c(\Omega)$, we let $\partial \partial F$ be the function defined by $\partial \partial F : (x , y, \varphi) \mapsto \left( \partial_{x , i} \partial_{y,j} F( \varphi) \right)_{1 \leq i,j \leq {d \choose 2}} \in \R^{\binom d2 \times \binom d2}$.

\subsubsection{The random interface measure $\mu_\beta$}

We first equip the measurable space $(\Omega_L , \mathcal{F}(\Omega_L))$ with a probability measure which plays an important role in the study of the low temperature Villain model in this article.

\begin{definition}[Finite-volume Gibbs measure]
    Given a (large) inverse temperature $\beta \in (1,\infty)$, we equip the space $\Omega_L$ with the following probability measure 
    \begin{equation} \label{def:finite-volumeGibbs}
        \mu_{L , \beta}(d \varphi) := \frac{1}{Z_{L , \beta}}  \exp \left( - \frac{1}{2 \beta} \sum_{x \in \Zd} \left| \nabla \varphi(x) \right|^2 -  \sum_{n \geq 2}  \sum_{x \in \Zd} \frac{1}{2\beta} \frac{1}{\beta^{\frac{n}{2}}} \left| \nabla^n \varphi(x) \right|^2 + \sum_{q \in \mathcal{Q}} z(\beta, q) \cos \left( 2\pi \left( \di^* \varphi , n_q \right) \right)  \right)
    \end{equation}
    where $Z_{L , \beta}$ is the constant chosen so that $\mu_{L , \beta}$ is a probability measure.
\end{definition}

\begin{remark} 
Let us introduce the three functions, for any $\varphi \in \Omega_L$,
        \begin{align*}
            H_1(\varphi) & := \frac{1}{\beta} \sum_{x \in \Zd} \left| \nabla \varphi(x) \right|^2, \\
            H_2(\varphi) & := \sum_{n \geq 2}  \sum_{x \in \Zd} \frac{1}{\beta} \frac{1}{\beta^{\frac{n}{2}}} \left| \nabla^n \varphi(x) \right|^2, \\
            H_3(\varphi) &:= \sum_{q \in \mathcal{Q}} z(\beta, q) \cos \left( 2\pi \left( \di^* \varphi , n_q \right) \right).
        \end{align*}
Despite the technical complexity of the formula~\eqref{def:finite-volumeGibbs}, a few important observations can be made:
    \begin{itemize}
        \item[(i)] When $\beta \gg 1$, the sums in the definitions of the terms $H_2$ and $H_3$ converge absolutely for any value of $\varphi\in \Omega_L$ (N.B. In fact the function $\varphi \mapsto H_2(\varphi)$ is quadratic and the function $\varphi \mapsto H_3(\varphi)$ can be shown to be smooth).
        \item[(ii)] The Hessians of the functions $H_2$ and $H_3$ are much smaller than the one of the function $H_1$. To be more specific, we have the identities: for any $\varphi , \psi \in \Omega_L$,
        \begin{align*}
             \left( \psi , \mathrm{Hess}\,  H_1(\varphi) \psi \right) & = \frac{1}{2\beta} \sum_{x \in \Zd} \left| \nabla \psi(x) \right|^2,  \\
            \left( \psi , \mathrm{Hess} \,H_2(\varphi) \psi \right) & = \sum_{n \geq 2}  \sum_{x \in \Zd} \frac{1}{2\beta} \frac{1}{\beta^{\frac{n}{2}}} \left| \nabla^n \psi(x) \right|^2,\\
            \left( \psi , \mathrm{Hess} \, H_3(\varphi) \psi \right) & = - (4 \pi)^2 \sum_{q \in \mathcal{Q}} z(\beta, q) \cos  \left( 2\pi \left( \di^* \varphi , n_q \right) \right) \left| \left( \di^* \psi , n_q \right) \right|^2,
        \end{align*}
        and thus, for $\beta \gg 1$ and any $\varphi , \psi \in \Omega_L$,
        \begin{align}
            \left( \psi , \mathrm{Hess} \,H_2(\varphi) \psi \right) & \leq \frac{C}{\sqrt{\beta}} \left( \psi , \mathrm{Hess}\,  H_1(\varphi) \psi \right), \label{H2small} \\
            \left| \left( \psi , \mathrm{Hess} \, H_3(\varphi) \psi \right) \right| & \leq C e^{-c \beta} \left( \psi , \mathrm{Hess}\,  H_1(\varphi) \psi \right) \label{H3small}.
        \end{align}
         \item[(iii)] A consequence of \eqref{H2small} and \eqref{H3small} is that the term inside the exponential in the definition~\eqref{def:finite-volumeGibbs} is a function which is uniformly convex in $\varphi$.
    \end{itemize}
    From all these observations, it can be deduced that the model~\eqref{def:finite-volumeGibbs} falls into the category of the (extensively) studied $\nabla \varphi$ interface model with uniformly convex interaction potential (see~\cite{F05, V06}) with two important features:
    \begin{itemize}
        \item The functions $\varphi \in \Omega_L$ are assumed to be valued in $\R^{d \choose 2}$ (while in the ``standard" $\nabla \varphi$-interface model they are usually assumed to be real-valued).
        \item The model has long-range interactions but with exponential decay of the strength of the interaction.
    \end{itemize} 
\end{remark}

These observations imply that the techniques and tools developed to study the $\nabla \varphi$-interface model (e.g., Brascamp-Lieb inequality, Helffer-Sj\"{o}strand representation formula) can be used to study the measure~\eqref{def:finite-volumeGibbs} and a number of results proved in the case the $\nabla \varphi$-interface model can be established for the measure~\eqref{def:finite-volumeGibbs}.

In this direction, a first important result (which will be needed in this article) is the existence of a thermodynamic limit, i.e., the convergence of the finite-volume measures $(\mu_{L , \beta})_{L \in \N}$ toward an infinite-volume measure. In the case of the $\nabla \varphi$-interface model, this question is usually answered using the Brascamp-Lieb inequality (see~\cite{BL, F05}) or elliptic regularity estimates (see~\cite{AW}), and the proof can be extended (without much difficulties) to the case of the measure $\mu_{L , \beta}$.

In order to state the property, we note that we always have the inclusion $\Omega_L \subseteq \Omega$. In particular, we can see the measure $\mu_{L , \beta}$ as a probability distribution on the space $(\Omega, \mathcal{F}(\Omega))$.

\begin{proposition}[Thermodynamic limit for $\mu_{L , \beta}$]\label{p.mubeta}
For $\beta \in (1 , \infty)$ sufficiently large, the series of measures $\left( \mu_{L, \beta} \right)_{L \in \N}$ converges as $L \to \infty$ to an infinite-volume Gibbs measure denoted by $\mu_\beta$. Additionally, the measure $\mu_\beta$ is invariant under translations and under the involution $\varphi \mapsto - \varphi$.
\end{proposition}

\begin{remark}
Let us make some remarks about the previous definition:
\begin{itemize}
    \item The (weak) convergence is defined as follows: for any function $\mathbf{f} \in C^\infty_c(\Omega)$, one has the convergence $$\left\langle \mathbf{f} \right\rangle_{\mu_{L , \beta}} \underset{L \to \infty}{\longrightarrow} \left\langle \mathbf{f} \right\rangle_{\mu_{ \beta}}.$$ We denote by $\left\langle \cdot \right\rangle_{\mu_\beta}$ and $\cov_{\mu_\beta}$ the expectation and covariance with respect to $\mu_\beta$.
    \item For $p \in (1, \infty]$ and $k \in \N$, we will make use of the following $L^p$ spaces:
    \begin{itemize}
        \item We let $L^p(\Omega \, ; \, \R^k)$ be the space of measurable functions $\mathbf{f}: \Omega \to \R^{k}$ satisfying (with $p = \infty$ being the essential supremum norm) $$\left\| \mathbf{f} \right\|_{L^p(\Omega \, ; \, \R^k)}^p := \left\langle \left| \mathbf{f} \right|^p \right\rangle_{\mu_\beta} < \infty;$$
        \item We let $L^p(\Zd \times \Omega \, ; \, \R^k)$ the space of measurable functions $G : \Zd \times \Omega \to \R^{k}$ satisfying (with $p = \infty$ being the essential supremum norm) $$\left\| G \right\|_{L^p(\Zd \times \Omega \, ; \, \R^k)}^p := \sum_{x \in \Zd} \left\langle \left| G(x , \cdot) \right|^p \right\rangle_{\mu_\beta} < \infty.$$
    \end{itemize}
    \item From now on, we will only work with the measure $\mu_\beta$.
\end{itemize}
\end{remark}

\subsection{The Helffer-Sj\"{o}strand representation} \label{sec:HSrepre}

This section is devoted to the Helffer-Sj\"{o}strand representation formula. This formula was introduced by Helffer and Sj\"{o}strand~\cite{HS} and then used by Naddaf-Spencer~\cite{NS} and Giacomin-Olla-Spohn~\cite{GOS} to identify the scaling limit of the $\nabla \varphi$-interface model. We will introduce it in this section in a few steps:
\begin{itemize}
    \item In Section~\ref{sec:wittenetHS}, we introduce the Helffer-Sj\"{o}strand operator.
    \item In Section~\ref{sec:HSrepresetnationformula}, we state the Helffer-Sj\"{o}strand representation formula.
    \item In Section~\ref{sec:sec263}, we collect some estimates on the decay of the fundamental solution (or Green's matrix) associated with the Helffer-Sj\"{o}strand operator.
\end{itemize}

\subsubsection{The Helffer-Sj\"{o}strand operator} \label{sec:wittenetHS}

We first introduce the Laplacian~$\Delta_\varphi$ defined on the set of functions $F \in C^\infty_c(\Omega)$ by the formula: for $\varphi \in \Omega$,
\begin{multline} \label{eq:TV13000102}
\Delta_\varphi F (\varphi) 
\\ : = 
\sum_{x\in \Zd} \partial_x^2 F(\varphi) 
+ 
\sum_{x\in \Zd} 
 \left[\frac{1}{2\beta} \Delta \varphi(x) - \sum_{n\geq 1} \frac{1}{2\beta} \frac{1}{\beta^{n/2}}(-\Delta)^{n+1} \varphi(x)-\sum_{q \in \mathcal{Q}} 2\pi z(\beta, q) q(x) \sin\left( 2\pi \left( \di^* \varphi , n_q \right)\right)  \right] \cdot \partial_xF(\varphi),
\end{multline}
where:
\begin{itemize}
    \item For the first term on the right-hand side, the notation $\partial_x^2$ means $\sum_{i=1}^{\binom d2} \partial_{x,i}^2$;
    \item The second term on the right-hand side is made of two terms: the (long) term inside the brackets and the partial derivative $\partial_xF$. We note that these two terms are both valued in $\R^{\binom d2}$ (since both the function $\varphi \in \Omega$, the charge $q \in \mathcal{Q}$ and the partial derivative  $\partial_xF$ are valued in $\R^{\binom d2}$) and that we take the scalar product between the two terms on the right side of~\eqref{eq:TV13000102}. 
\end{itemize} 
The operator $\Delta_\varphi$ is defined so as to satisfy the following identities
\begin{equation*}
\left\langle F \Delta_\varphi G \right \rangle_{\mu_{\beta}} 
=
\left\langle G \Delta_\varphi F \right \rangle_{\mu_{\beta}} 
=
- \sum_{x\in\Zd} \langle \partial_x F \cdot \partial_x G \rangle_{\mu_{\beta}},
\quad \forall F,G \in C^\infty_c(\Omega  )
\end{equation*}
 (N.B. there is a scalar product on the right-most term because both functions $\partial_x F$ and $\partial_x G$ are valued in~$\R^{d \choose 2}$).
We extend this definition to the functions of $C^\infty_c( \Zd \times \Omega \, ; \,  \R^{d \choose 2} )$ by setting, for $F = (F_1 , \ldots, F_{d \choose 2}) \in C^\infty_c (\Zd \times  \Omega \, ; \, \R^{d \choose 2} )$ and $(x , \varphi) \in \Zd \times \Omega$,
\begin{equation*}
    \Delta_\varphi F(x , \varphi) = \left( \Delta_\varphi F_1(x , \varphi) , \ldots, \Delta_\varphi F_{d \choose 2}(x , \varphi) \right) \in \R^{d \choose 2}.
\end{equation*}
We next introduce the Helffer-Sj\"{o}strand operator. To this end, we extend the definition of the discrete Laplacian $\Delta$ to functions of $C^\infty_c( \Zd \times \Omega )$ by setting, for any $(x,\varphi) \in \Zd \times \Omega$,
\begin{equation*}
     \Delta F( x , \varphi) := \sum_{y \sim x}  ( F( y , \varphi) - F( x , \varphi) ),
\end{equation*}
and define the iteration of the Laplacian $(-\Delta)^k$ by iterating the previous definition. For $q \in \mathcal{Q}$ and $\varphi \in \Omega$, we define the coefficient
\begin{equation*}
    \a_q (\varphi):=  4\pi^2 z\left( \beta , q \right) \cos \left(2\pi\left( \varphi , q \right)\right) \in \R.
\end{equation*}
Given a function $F \in C^\infty_c(\Zd \times \Omega \, ; \, \R^{d \choose 2})$ and $\varphi \in \Omega$, we introduce the notation (N.B. note that this quantity depends only on the gradient of the function $F$ since the codifferential $\di^*$ can be written as a linear combination of the components of the gradient of $F$)
\begin{equation*}
    \nabla_q F(\varphi) := (q , F(\cdot, \varphi)) = (n_q , \di^* F(\cdot, \varphi)) \in \R.
\end{equation*}
Combining the two previous definitions, we introduce the (long-range) operator 
\begin{align*}
 \nabla_q^* \cdot \a_q \nabla_q F(x, \varphi) & = 4\pi^2 z\left( \beta , q \right) \cos \left(2\pi\left( \varphi , q \right)\right) \left( \nabla_q F(\varphi) \right) q(x) \\
 & = 4\pi^2 z\left( \beta , q \right) \cos \left(2\pi\left( \varphi , q \right)\right) \left( \di^* F(\cdot , \varphi) , n_q \right) q(x) \in \R^{d \choose 2}. \notag
\end{align*}
The notation is motivated by the following identity: for any $F , G \in C^\infty_c( \Zd \times \Omega \, ; \, \R^{d \choose 2} )$ and any $q \in \mathcal{Q},$
\begin{align*}
   \sum_{x \in \Zd} \left( \nabla_q^* \cdot \a_q \nabla_q F \right)(x , \varphi) \cdot G(x, \varphi)  & = \sum_{x \in \Zd} \left(\nabla_q^* \cdot \a_q \nabla_q \right) G(x, \varphi) \cdot F(x, \varphi) \\
   & = \a_q(\varphi) \left( \di^* F(\cdot , \varphi)  , n_q \right) \left( \di^* G(\cdot , \varphi) , n_q \right) \in \R.
\end{align*}
Equipped with these definitions, we introduce the spatial operator $\mathcal{L}^{\varphi}$ and the Helffer-Sj{\"o}strand operator $\mathcal{L}$ acting on functions $F \in C^\infty_c( \Zd \times \Omega  \, ; \, \R^{d \choose 2})$
\begin{equation*}
    \mathcal{L}^{\varphi}_{\mathrm{spat}} F(x , \varphi) : =  \frac{1}{2\beta} (-\Delta) F(x , \varphi) + \frac{1}{2\beta}\sum_{n \geq 1} \frac{1}{\beta^{ \frac n2}} \left(-\Delta\right)^{n+1} F(x , \varphi) + \sum_{q \in \mathcal{Q}} \nabla_q^* \cdot \a_q \nabla_q F(x , \varphi)
\end{equation*}
and
\begin{equation*}
\L F(x , \varphi) : = -\Delta_\varphi F(x , \varphi) +  \mathcal{L}^{\varphi}_{\mathrm{spat}} F(x , \varphi).
\end{equation*}
We remark that the spatial operator $\mathcal{L}^{\varphi}_{\mathrm{spat}}$ can be applied to functions depending only on the spatial variable (i.e. to functions $F : \Zd \to \R^{d \choose 2}$) once a function $\varphi \in \Omega$ is specified (N.B. The function $\varphi \in \Omega$ appears in the coefficient $\mathbf{a}_q$).

In this article, we will be interested in the solutions of the Helffer-Sj\"{o}strand equation $\mathcal{L} F = G$ (for a given function $G : \Zd \times \Omega \to \R^{d \choose 2}$) as these functions can be used to obtain precise information on the random interface measure $\mu_\beta$ (see Proposition~\ref{prop:HSrepresetnation} below). We first introduce a definition of a solution for the Helffer-Sj\"{o}strand equation together with a statement ensuring the existence and uniqueness of solutions for a large class of functions $G : \Zd \times \Omega \to \R^{d \choose 2}$.

\begin{definition}[Solution of the Helffer-Sj\"{o}strand equation] \label{def.solHSeq}
Let $G:  \Zd \times \Omega \to \R^{\binom d2}$ be a function that satisfies $G(x , \cdot)\in L^1(\Omega ; \R^{\binom d2})$ for any $x \in \Zd$. A function $\mathcal{U}  : \Zd \times \Omega \to \R^{\binom d2}$ is called a weak solution of the Helffer-Sj\"{o}strand equation 
$$\mathcal{L} \mathcal{U} = G ~~\mbox{in}~~  \Zd \times \Omega,$$
if $\mathcal{U}(x , \cdot)\in L^1\left(\Omega; \R^{\binom d2}\right)$ for any $x \in \Zd$ and if, for any function $F \in C^\infty_c\left(\Zd \times \Omega \, ; \,  \R^{d \choose 2} \right)$,
\begin{equation} \label{eq:solHSeqdef}
 \sum_{x \in \Zd} \left\langle \mathcal{U}(x , \cdot) \cdot \mathcal{L} F(x , \cdot) \right\rangle_{\mu_\beta} 
 = \sum_{x\in \Zd} \left\langle  G(x , \cdot) \cdot F(x , \cdot) \right\rangle_{\mu_\beta}.
\end{equation}
\end{definition}

\begin{remark}
    On both sides of the identity~\eqref{eq:solHSeqdef}, the functions $\mathcal{U}$,  $\mathcal{L} F$, $G$ and $F$ are valued in $\R^{d \choose 2}$, there is thus a scalar product (in $\R^{d \choose 2}$) on both sides.
\end{remark}

The next proposition ensures the existence and uniqueness of solutions of the Helffer-Sj\"{o}strand equation. The proof of this result can be found, for instance, in the article of Naddaf-Spencer~\cite{NS} (using the Lax-Milgram Theorem).

\begin{proposition}[Solvability of the Helffer-Sj\"{o}strand equation, Section 2.1 of~\cite{NS} and Section 3.4.2 of~\cite{DW}]
For any $G \in L^2(\Zd \times  \Omega \, ; \, \R^{d \choose 2} )$, there exists a unique weak solution $\mathcal{U}: \Zd \times \Omega \to \R^{\binom d2}$ of the Helffer-Sj\"{o}strand equation
\begin{equation*}
 \mathcal{L} \mathcal{U}
 =  G
\quad \text{in } \Zd \times \Omega,
\end{equation*}
which satisfies, for some $C(d,\beta)<\infty$,
\begin{equation}
\label{e.HSsolest}
\sum_{x , y \in \Zd} \left\langle \left| \partial_y \mathcal{U}(x , \cdot) \right|^2\right\rangle_{\mu_\beta} + \sum_{x \in \Zd} \left\langle \left| \nabla \mathcal{U}(x , \cdot) \right|^2\right\rangle_{\mu_\beta} + \sum_{x \in \Zd} \left\langle \left| \mathcal{U}(x , \cdot) \right|^{\frac{2d}{d-2}}\right\rangle_{\mu_\beta}
\leq 
C \sum_{x \in \Zd} \left\langle \left| G(x , \cdot) \right|^2\right\rangle_{\mu_\beta}. 
\end{equation}
\end{proposition}

\begin{remark}
    We will omit the proof of this result and refer to~\cite[Section 2.1]{NS} (N.B. The proof of~\cite{NS} is written in the case of the $\nabla \varphi$ interface model, some minor adaptations are needed to treat the case the random interface measure $\mu_\beta$, they can be found in~\cite[Section 3.4.2]{DW}). We mention that the upper bound on the first two terms on the left-hand side of~\eqref{e.HSsolest} is a consequence of the Lax-Milgram Theorem and the bound on the third term on the left-hand side of~\eqref{e.HSsolest} is a consequence of the Gagliardo-Nirenberg-Sobolev inequality (and only holds in dimension $d \geq 3$).
\end{remark}

\subsubsection{The Helffer-Sj\"{o}strand representation formula} \label{sec:HSrepresetnationformula}

This section is devoted to the presentation of the Helffer-Sj\"{o}strand representation formula which provides an identity relating the covariance of two functions $\mathbf{f} , \mathbf{g} : \Omega \to \R$ under the measure $\mu_\beta$ to the solutions of the Helffer-Sj\"{o}strand equation.

\begin{proposition}[Helffer-Sj\"{o}strand representation~\cite{HS, Sj, NS, GOS}] \label{prop:HSrepresetnation}
    Let $\mathbf{f}, \mathbf{g} \in \C_c^\infty(\Omega)$, then one has the identity
    \begin{equation} \label{eq:HSrep}
        \cov_{\mu_\beta} \left[ \mathbf{f}, \mathbf{g} \right] = \sum_{x \in \Zd} \left\langle \partial_x \mathbf{f} \cdot \mathcal{H}(x , \cdot)  \right\rangle_{\mu_\beta},
    \end{equation}
    where the function $\mathcal{H} : \Zd \times \Omega \to \R^{d \choose 2}$ is the solution of the Helffer-Sj\"{o}strand equation
    \begin{equation*}
        \mathcal{L} \mathcal{H}(x , \varphi) = \partial_x \mathbf{g}(\varphi) \hspace{3mm}\mbox{for} \hspace{3mm} (x , \varphi) \in \Zd \times \Omega.
    \end{equation*}
\end{proposition}

\begin{remark}
    In the previous statement, we assumed that the functions $\mathbf{f}$ and $\mathbf{g}$ belong to the space $\C_c^\infty(\Omega)$, but the formula can be extended (using approximation arguments) to more general functions; in particular to the functions appearing in~\eqref{eq:def3.5} of Section~\ref{sec:section3}.
\end{remark}

The identity~\eqref{eq:HSrep} is one in the key tools of this article (and in the study of the $\nabla \varphi$-interface model in general). In particular, it will be interesting to us to obtain upper bounds on the solutions of the Helffer-Sj\"{o}strand equation, as they can be used to derive upper bounds on the covariance of two random variables under the measure $\mu_\beta$ (N.B. This has already been observed and used in the literature, see for instance~\cite{NS, GOS}). Such upper bounds are usually obtained by using tools of elliptic regularity and the result in the case of the measure $\mu_\beta$ are collected in the following section.

\subsubsection{The Green's matrix for the Helffer-Sj{\"o}strand equation}
\label{sec:sec263}
In order to obtain (flexible) upper bounds on the solution of the Helffer-Sj\"{o}strand equation, we will introduce the fundamental solution of the Helffer-Sj\"{o}strand equation and state some inequalities on its decay (Proposition~\ref{prop.prop4.11chap4}). From these results, it is possible to derive upper bounds on general solutions of the Helffer-Sj\"{o}strand equation.

Before giving the precise results and statements, we mention that, since the solutions of the Helffer-Sj\"{o}strand equation are functions valued in $\R^{\binom d2}$, we are working with a system of equations. As a consequence the fundamental solution is not a function but a matrix valued in the space $\R^{\binom d2 \times \binom d2}$.

\begin{definition}[Green's matrix for the Helffer-Sj\"{o}strand equation] \label{eq:defHSGREEN}
For any $y \in \Zd$, we let $\delta_{y} : \Zd \to \R^{\binom d2 \times \binom d2}$ be the discrete Dirac mass defined by the formula
$$\delta_{y}(x) := \left( \indc_{\{x = y\}} \indc_{\{i=j\}} \right)_{1 \leq i, j \leq \binom d2} \in \R^{{d \choose 2 }  \times \binom d2}.$$
   For any function $\f : \Omega \to \R$ satisfying $\f \in L^2 \left( \Omega \, ; \, \R \right)$ and any $y \in \Zd$, we let $\G_\f(\cdot ; y):= \left( \G_{\f, ij}(\cdot ; y) \right)_{1\leq i,j\leq \binom d2} : \Zd \times \Omega \to \R^{\binom d2 \times \binom d2}$ be the solution of the Helffer-Sj\"{o}strand equation
\begin{equation} \label{eq:defGreensmatrix}
    \mathcal{L} \G_\f(\cdot ; y) = \mathbf{f} \delta_{y} ~~\mbox{in}~~ \Zd \times \Omega.
\end{equation}
\end{definition}

\begin{remark} \label{rem:remark2.21}
    Let us make a few remarks on the previous definition:
    \begin{itemize}
    \item The identity~\eqref{eq:defGreensmatrix} is understood as follows: for any $j \in \{ 1 , \ldots, {d \choose 2} \}$, the function $\G_{\f, \cdot j}(\cdot ; y):= \left( \G_{\f, ij}(\cdot ; y) \right)_{1\leq i \leq \binom d2} : \Zd \times \Omega \to \R^{\binom d2}$ solves
    \begin{equation*}
        \mathcal{L} \mathcal{G}_{\mathbf{f}, \cdot j}(x , \varphi; y) = 
        \begin{pmatrix}
           0 \\
           \vdots \\
           \textbf{f} (\varphi)  \indc_{\{x = y\}} \\
           \vdots \\
           0
         \end{pmatrix} 
         \hspace{3mm} \mbox{for}~ (x , \varphi) \in \Zd \times \Omega,
    \end{equation*}
    where the non-zero entry is in the $j$-th position.
    \item The mapping $\mathbf{f} \mapsto \mathcal{G}_{\mathbf{f}}$ is linear.
    \item Using the notation of Definition~\ref{def.solHSeq} and writing $G :=  (G_j)_{1 \leq j \leq {d \choose 2}} : \Zd \times \Omega \to \R^{d \choose 2}$, one has the identity
    \begin{equation} \label{id:decomgeneralsolutions}
        \mathcal{U}(x , \varphi) := \sum_{y \in \Zd} \sum_{j = 1}^{d \choose 2} \G_{G_j(y , \cdot), \cdot j}(x , \varphi ; y).
    \end{equation}
    \item In this article, we will make extensive use of the codifferential of the Green's matrix. It is formally defined as follows. For $\mathbf{f} : \Omega \to \R$, we define the codifferential in the first variable, and denote it by $\di^*_1 \mathcal{G}_{\mathbf{f}}$, as follows: for each fixed index $j \in \{1 , \ldots, \binom d2 \}$, vertex $y \in \Zd$ and each $\varphi \in \Omega$, we consider the function $x \mapsto \mathcal{G}_{\mathbf{f}, \cdot j}(x , \varphi ; y) \in \R^{d \choose 2}$. This function can be seen as a $2$-form and one can apply the codifferential~$\di^*$ as in Definition~\ref{def:codifferential2form}. We thus obtain a $1$-form $x \mapsto \di^*_1 \mathcal{G}_{\mathbf{f}, \cdot j}(x , \varphi ; y) \in \R^{d}$. We then define the function 
    $$\di^*_1 \mathcal{G}_{\mathbf{f}} : (x , \varphi ; y) \mapsto \left( \di^*_1 \mathcal{G}_{\mathbf{f}, i j}(x , \varphi ; y) \right)_{1 \leq i \leq d, 1 \leq j \leq \binom d2 } \in \R^{d \times \binom d2}.$$
    We may similarly define the codifferential in the second variable, denoted by $\di^*_2 \mathcal{G}_{\mathbf{f}}$, by freezing the first variable and the first component $i \in \{1 , \ldots, {d \choose 2} \}$, and the mixed codifferential, denoted by $\di^*_1 \di^*_2 \mathcal{G}_{\mathbf{f}}$, by applying this procedure to both variables. We have in particular
    \begin{equation*}
        \begin{aligned}
            \di_2^* \mathcal{G}_{\mathbf{f}} & := \left( \di_2^* \mathcal{G}_{\mathbf{f}, i j} \right)_{1 \leq i \leq \binom d2, 1 \leq j \leq  d} : \Zd \times \Omega \times \Zd \to \R^{{d \choose 2} \times d}, \\
            \di^*_1 \di_2^* \mathcal{G}_{\mathbf{f}} & := \left( \di^*_1 \di_2^*   \mathcal{G}_{\mathbf{f}, i j} \right)_{1 \leq i \leq d, 1 \leq j \leq  d} : \Zd \times \Omega \times \Zd \to \R^{d \times d}.
        \end{aligned}
    \end{equation*}
    \item Even though this will be used less frequently in the proofs below, we may also similarly define the exterior derivative with respect to the first and second variables (and denote them by $\di_1$ and $\di_2$).
    \item We will frequently use the identity~\eqref{id:decomgeneralsolutions} when the function $G$ takes the form $G(x , \varphi) = \mathbf{f}(\varphi) q(x)$ for some function $\mathbf{f} : \Omega \to \R$ and some charge $q := (q_j)_{1 \leq j \leq \binom d2} \in \mathcal{Q}$. In that case, we first introduce the two pieces of notation, for each pair $(x , \varphi) \in \Zd \times \Omega$,
    \begin{equation*}
        \begin{aligned}
            \left( \G_{\mathbf{f}}(x , \varphi ; \cdot),  q \right) & =  \sum_{y \in \Zd} \sum_{j = 1}^{d \choose 2} \G_{\mathbf{f}, \cdot j}(x , \varphi ; y) q_j(y) \in \R^{\binom d2} \\
            \left( \di_2^* \G_{\mathbf{f}}(x , \varphi ; \cdot),  n_q \right) & =  \sum_{y \in \Zd} \sum_{j = 1}^{d} \di^*_2 \G_{\mathbf{f}, \cdot j}(x , \varphi ; y) n_{q,j}(y) \in \R^{\binom d2}
        \end{aligned}
    \end{equation*}
    (N.B. The sums are in fact finite because both $q$ and $n_q$ have finite support and these two terms are equal due to Lemma~\ref{lem:lem2.4} and~\eqref{identity:dandd*}). The identity~\eqref{id:decomgeneralsolutions} then becomes
    \begin{equation*}
    \mathcal{U}(x , \varphi) = \left( \G_{\mathbf{f}}(x , \varphi ; \cdot),  q \right) = \left( \di_2^* \G_{\mathbf{f}}(x , \varphi ; \cdot),  n_q \right) \in \Rd.
    \end{equation*}
    It will finally be useful to consider (and estimate) the codifferential of the function $\mathcal{U}$. To this end, we introduce two additional pieces of notation
    \begin{equation*}
        \begin{aligned}
            \left( \di^*_1 \G_{\mathbf{f}}(x , \varphi ; \cdot),  q \right) & =  \sum_{y \in \Zd} \sum_{j = 1}^{d \choose 2} \di^*_1 \G_{\mathbf{f}, \cdot j}(x , \varphi ; y) q_j(y) \in \R^{d}, \\
            \left( \di_1^* \di_2^* \G_{\mathbf{f}}(x , \varphi ; \cdot),  n_q \right) & =  \sum_{y \in \Zd} \sum_{j = 1}^{d} \di^*_1 \di^*_2 \G_{\mathbf{f}, \cdot j}(x , \varphi ; y) n_{q,j}(y) \in \R^{d},
        \end{aligned}
    \end{equation*}
    so that we have
    \begin{equation*}
        \di^* \mathcal{U}(x , \varphi) = \left( \di^*_1 \G_{\mathbf{f}}(x , \varphi ; \cdot),  q \right) = \left(\di^*_1 \di_2^* \G_{\mathbf{f}}(x , \varphi ; \cdot),  n_q \right) \in \R^d.
    \end{equation*}
    \end{itemize}
\end{remark}

The following proposition collects some decay estimates on the Green's matrix and its codifferential.

\begin{proposition}[Upper bounds on the Green's matrix and its codifferential, Proposition 3.17 of \cite{DW} and Appendix~\ref{app.CZreg}] \label{prop.prop4.11chap4}

For any exponent $p \in (1 , \infty)$, there exists an inverse temperature $\beta_1 (d , p) < \infty$ such that for any $\beta > \beta_1$ the following result holds. There exists a constant $C(d, \beta, p ) < \infty$ such that for any $x,y \in \Zd$, one has the inequality
\begin{equation} \label{eq:estimateGreensfunction}
\left\| \G_\f (x , \cdot;y ) \right\|_{L^p \left( \Omega \, ; \, \R^{{\binom d2} \times {\binom d2}} \right)} \leq \frac{C \left\| \f \right\|_{L^{p} \left( \Omega \, ; \, \R \right)}}{|x-y|_+^{d-2}},
\end{equation}
together with the regularity estimates
\begin{equation} \label{eq:decaygradgreen}
\left\| \di^*_1 \G_\f (x, \cdot ; y ) \right\|_{L^p\left( \Omega \, ; \, \R^{d \times {\binom d2}} \right)} + \left\| \di_2^* \G_\f (x, \cdot ; y ) \right\|_{L^p\left( \Omega \, ; \, \R^{  {\binom d2}\times d} \right)}  \leq \frac{C (\ln |x-y|_+)^{d+2} \left\| \f \right\|_{L^{2p} \left( \Omega \, ; \, \R \right)} }{|x-y|^{d - 1}_+}
\end{equation}
and
\begin{equation}  \label{eq:decaygradgradgreen}
    \left\| \di^*_1 \di^*_2 \G_\f (x ,\cdot ; y) \right\|_{L^p\left( \Omega \, ; \, \R^{d \times d} \right)} \leq \frac{C (\ln |x-y|_+)^{d+2} \left\| \f \right\|_{L^{2p} \left( \Omega \, ; \, \R \right)}}{|x-y|^{d}_+}.
\end{equation}
\end{proposition}

\begin{remark}
Let us make a few remarks about this statement:
\begin{itemize}
    \item An estimate of the form of~\eqref{eq:estimateGreensfunction} was originally proved by Naddaf and Spencer~\cite[Section 2.2.2]{NS} in the case of the $\nabla \varphi$-interface model and the argument relies on elliptic regularity estimates (specifically the off-diagonal Nash-Aronson upper bound for heat kernels). In the case of the measure $\mu_\beta$, the proof can be found in~\cite[Proposition 3.17]{DW}. It essentially follows the same strategy as in~\cite[Section 2.2.2]{NS} with an important difference: since we are dealing here with a system of equations, the off-diagonal Nash-Aronson estimate cannot be applied directly. This obstruction is fixed by noticing that, due to the assumption $\beta \gg 1$, the Helffer-Sj\"{o}strand operator has a small ellipticity contrast (of order $(1 + e^{-c\sqrt{\beta}})$). This allows to make use of the Schauder regularity theory (see, e.g.,~\cite[Chapter 6]{GT01}) to recover the results of elliptic regularity needed to derive the upper bound~\eqref{eq:estimateGreensfunction}. 
    \item Regarding the upper bound on the codifferential~\eqref{eq:decaygradgreen} and mixed codifferential~\eqref{eq:decaygradgradgreen}, we rely once again on the small ellipticity contrast of the Helffer-Sj\"{o}strand operator but make use this time of the \emph{Calder\'{o}n-Zygmund regularity theory} (see~\cite[Chapter 9]{GT01}) and combine it with the ideas of Delmotte-Deuschel~\cite{DD05} on the annealed regularity of heat kernels associated with parabolic equations with random stationary coefficients. The details of the proofs can be found in Appendix~\ref{app.CZreg}. We remark that the result would be optimal if the logarithm on the right-hand side were removed. We believe that reaching optimality would be possible with a more careful analysis but decided not to do it to simplify the technical nature of the argument (and because, even with the optimal result at hand, the proof in this article would still yield a logarithmic correction on the right-hand side of Theorem~\ref{th:mainth}).
\end{itemize}
\end{remark}

\section{Upper bound for the parallel correlation function} \label{sec:section3}

This section contains (most of) the proof of Theorem~\ref{th:mainth} and is structured as follows. In Section~\ref{sec:sec3.1}, we recall the duality transform of~\cite{FS4d} which allows to rewrite the cosine-cosine correlation as the expectation of a certain random variable under the measure $\mu_\beta$ (see Proposition~\ref{p.dual} below). We then study this random variable and split the argument into several sections. In Section~\ref{sec:sec3.2}, we decompose this random variable into a collection of six random variables and then show, in Sections~\ref{sec:3.3} and~\ref{sec:section3.4}, that four of them are lower-order terms which can be estimated fairly easily. The core of the analysis is then in Section~\ref{sec:section3.4} where the Helffer-Sj\"{o}strand representation formula is applied to estimate the contributions of the remaining two random variables (i.e., the random variables $U_0$ and $U_{\mathrm{cos},0}$). The estimate of one of the terms of Section~\ref{sec:section3.4} (specifically, the identity~\eqref{eq:identitydstarGHW}) requires some specific attention and the details of this computation are postponed to Section~\ref{sec:newsec4}.

Throughout this section (and the next one), we will make use of a number of sums on the lattice which are collected in Appendix~\ref{App.sumonlattice}, and we will frequently bound the sines and the cosines using the inequalities $\left| \sin a \right| \leq |a|$ and $1- \cos a \leq \frac{a^2}{2}$ for $a \in \R$.

\subsection{Duality formulae} \label{sec:sec3.1}

The starting points of the proof are the following three identities relating the Villain model, the discrete Green's function and a (technically involved) random variable defined in terms of the random interface measure $\mu_\beta$ (N.B. it would be possible to derive other identities of a similar nature, but we will only need these three identities in the proofs below).

We denote by $\mu_\beta$ the infinite volume Gibbs measure obtained in Proposition \ref{p.mubeta}, $G= (-\Delta)^{-1}$ be the lattice Green's function in $\Zd$ (see Definition~\ref{def:defgreenlattice}), and $\mu_{\mathrm{Vil}, \beta}$ the thermodynamic limit of the Villain model with zero boundary conditions (i.e., the ``+" state constructed in \cite{BFLLS}). 

\begin{proposition}
\label{p.dual}
    There exists an inverse temperature $\beta_1:= \beta_1(d)<\infty$, such that for all $\beta\geq \beta_1$,
    \begin{align} \label{eq:10541102}
      \left\langle \cos \left( \theta(x) \right) \right\rangle_{\mu_{\mathrm{Vil}, \beta}}
  \exp \left(\frac{1}{2\beta}  G (0) \right)
  & = \left\langle \exp \left( \sum_{q \in \mathcal{Q}} z(\beta , q) \sin \left(2\pi (\di^* \varphi , n_q)\right) \sin \left(2\pi(\nabla G , n_q)\right)  \right) \right.  \\ 
  & \qquad  \times \left.  \exp \left( \sum_{q \in \mathcal{Q}} z(\beta , q) \cos \left(2\pi(\di^* \varphi , n_q)\right) \left( \cos  \left(2\pi(\nabla G , n_q)\right) - 1 \right) \right) \right\rangle_{\mu_{\beta}}, \notag
\end{align}
\begin{multline} \label{eq:10541103}
      \left\langle \cos \left( \theta(x) - \theta(0) \right) \right\rangle_{\mu_{\mathrm{Vil}, \beta}}
  \exp \left(\frac{1}{\beta} \left( G (0) - G(x) \right)\right) \\
  = \left\langle \exp \left( \sum_{q \in \mathcal{Q}} z(\beta , q) \sin \left(2\pi (\di^* \varphi , n_q)\right) \sin \left(2\pi(\nabla G - \nabla G_x , n_q)\right)  \right) \right. \\  \times \left.  \exp \left( \sum_{q \in \mathcal{Q}} z(\beta , q) \cos \left(2\pi(\di^* \varphi , n_q)\right) \left( \cos  \left(2\pi(\nabla G - \nabla G_x , n_q)\right) - 1 \right) \right) \right\rangle_{\mu_{\beta}},
\end{multline}
and 
\begin{multline} \label{eq:10541104}
       \left\langle \cos \left( \theta(x) + \theta(0) \right) \right\rangle_{\mu_{\mathrm{Vil}, \beta}}
  \exp \left(\frac{1}{\beta} \left( G (0) + G(x) \right)\right) \\
  = \left\langle \exp \left( \sum_{q \in \mathcal{Q}} z(\beta , q) \sin \left(2\pi (\di^* \varphi , n_q)\right) \sin \left(2\pi(\nabla G + \nabla G_x , n_q)\right)  \right) \right. \\  \times \left.  \exp \left( \sum_{q \in \mathcal{Q}} z(\beta , q) \cos \left(2\pi(\di^* \varphi , n_q)\right) \left( \cos  \left(2\pi(\nabla G + \nabla G_x , n_q)\right) - 1 \right) \right) \right\rangle_{\mu_{\beta}}.
\end{multline}
\end{proposition}

\begin{remark}
    The infinite sums on the right-hand sides are well-defined random variables and we refer to Appendix~\ref{app.notation} for additional justification.
\end{remark}

The second and the third identities are obtained using the techniques of~\cite{FS4d} and they appear in~\cite[Section 5.5]{Bau} or~\cite[Proposition 3.3.2]{DW}. The first identity follows from the same argument: one starts with a finite volume identity (see the line below (3.1.6) of \cite{DW}), and apply the cluster expansion and thermodynamic limit along the line of Chapter 3 of \cite{DW}.

The strategy of the proof is then to rewrite the covariance of the cosine using the trigonometric identity,
\begin{align} \label{eq:14141102}
    \cov_{\mu_{\mathrm{Vil}, \beta}} \left[ \cos \theta(0) , \cos \theta(x) \right] & =  \langle \cos \theta(0) \cos\theta(x) \rangle_{\mu_{\mathrm{Vil}, \beta}} - \langle \cos \theta(0)\rangle\langle \cos\theta(x) \rangle_{\mu_{\mathrm{Vil}, \beta}}
    \\
    & = 
    \frac{1}{2}  \langle \cos (\theta(0) -\theta(x))  \rangle_{\mu_{\mathrm{Vil}, \beta}}+
    \frac{1}{2}  \langle  \cos (\theta(0) +\theta(x))   \rangle_{\mu_{\mathrm{Vil}, \beta}}
    - \left( \langle \cos (\theta(0)) \rangle_{\mu_{\mathrm{Vil}, \beta}} \right)^2, \notag
\end{align}
and to use the three identities~\eqref{eq:10541102},~\eqref{eq:10541103} and~\eqref{eq:10541104} to reformulate the statement of Theorem~\ref{th:mainth} into a problem involving the discrete Green's function $G$ and a random variable defined in terms of the random interface measure $\mu_\beta$. We may then study these two quantities using the properties of the discrete Green's function (Proposition~\ref{prop:greenfunction}) and the Helffer-Sj\"{o}strand representation formula.

\subsection{Setting up the problem} \label{sec:sec3.2}
In order to ease the formalism, we introduce the following notation: for $x \in \Zd$ and $\varphi \in \Omega$,
\begin{equation} \label{eq:def3.5}
  \begin{aligned}
  U_{x}(\varphi) &:=  \sum_{q \in \mathcal{Q}} z(\beta , q)  \sin \left(2\pi(\di^* \varphi , n_q)\right)  \sin \left(2\pi(\nabla G_x , n_q)\right),  \\ 
  U_{\cos , x}(\varphi) &:=  \sum_{q \in \mathcal{Q}} z(\beta , q) \cos \left(2\pi(\di^* \varphi , n_q)\right) \left( \cos \left(2\pi(\nabla G_x , n_q)\right) - 1 \right), \\
  U_{\sin \cos, x}(\varphi) & :=  \sum_{q \in \mathcal{Q}} z(\beta , q) \sin \left(2\pi(\di^* \varphi , n_q)\right) \sin\left(2\pi( \nabla G , n_q)\right) \left( \cos \left(2\pi(\nabla G_x , n_q)\right) - 1 \right),  \\  
  U_{\cos \sin,x}(\varphi) & := \sum_{q \in \mathcal{Q}} z(\beta , q) \sin \left(2\pi(\di^* \varphi , n_q)\right) \left( \cos \left(2\pi(\nabla G , n_q)\right) - 1 \right) \sin \left(2\pi(\nabla G_x , n_q)\right)  ,\\
  U_{\cos \cos,x}(\varphi) & := \sum_{q \in \mathcal{Q}} z(\beta , q) \cos \left(2\pi(\di^* \varphi , n_q)\right)  \left( \cos \left(2\pi (\nabla G , n_q)\right) - 1 \right) \left(\cos \left(2\pi(\nabla G_x , n_q)\right) -1 \right) ,\\
  U_{\sin \sin,x}(\varphi) & := \sum_{q \in \mathcal{Q}} z(\beta , q) \cos \left(2\pi(\di^* \varphi , n_q)\right) \sin \left(2\pi(\nabla G , n_q)\right) \sin \left(2\pi(\nabla G_x , n_q)\right) .
  \end{aligned}
\end{equation}

\begin{remark} \label{rem:remark3.3}
Let us state some remarks and properties of these random variables (and refer to Appendix~\ref{app.notation} for the proofs): 
\begin{itemize}
\item All these random variables are real-valued.
\item To ease the notation, we may drop the dependency in the variable $\varphi$ in the proof below (e.g., we may write $U_{x}$ instead of $U_{x}(\varphi)$).
\item Among these variables, $U_x$ is the only one for which the sum does not converge absolutely for every value of $\varphi \in \Omega$ but the sum converges in $L^2(\Omega \, ; \, \R)$.
\item For the five other random variables, the sum over the $2$-forms of $\mathcal{Q}$ converges absolutely for any realisation of the function $\varphi \in \Omega$ (and in fact uniformly over $\varphi \in \Omega$, they are thus in the space $L^\infty(\Omega \, ; \, \R)$).
\item The exponentials of these random variables belong to the space $L^{p}(\Omega \, ; \, \R)$ for some exponent $p := p(d , \beta) \in [1 , \infty)$ satisfying $p \to \infty$ as $\beta \to \infty$ (we will assume that $\beta$ is sufficiently large so that this estimate holds with an exponent $p$ sufficiently large, e.g., larger than $8$).  Additionally, the bound is uniform over $x \in \Zd$ (N.B. this result is not formally proved in the article but can be obtained by a generalisation of the proof of Appendix~\ref{app.notation} for which we refer to~\cite[Theorem 4.9]{F05} or to~\cite[Proposition 3.4]{DW}).
\item For later purposes, we note that we have the following formula for the derivative of $U_x$ (N.B. due to the term $q(y)$, which is equal to $0$ for all the forms whose support does not include the vertex $y$, one can show that the sum on the right-hand side converges absolutely for any $\varphi \in \Omega$), for any $y \in \Zd$,
\begin{equation} \label{formuladerUx}
    \partial_y U_x(\varphi) = 2 \pi \sum_{q \in \mathcal{Q}} z(\beta , q) \cos \left( 2\pi (\di^* \varphi , n_q) \right) \sin \left( 2\pi(\nabla G_x , n_q) \right) q(y) \in \R^{d \choose{2}}
\end{equation}
and that similar formulae hold for the five other random variables.
\end{itemize}
\end{remark}

Using this formalism, we have
\begin{equation*}
    \left\langle \cos \left( \theta(x) \right) \right\rangle_{\mu_{\mathrm{Vil}, \beta}} \exp \left(\frac{1}{2\beta}  G (0) \right) = \left\langle \exp \left( U_0 + U_{\cos , 0} \right) \right\rangle_{\mu_\beta}.
\end{equation*}
Regarding the identity~\eqref{eq:10541103}, we first expand the sine and cosine on the right-hand sides of~\eqref{eq:10541103} using the trigonometric identities:
    \begin{multline*}
  \sin \left(2\pi(\nabla G - \nabla G_x , n_q)\right)= \sin \left(2\pi(\nabla G , n_q)\right) - \sin \left(2\pi(\nabla G_x , n_q)\right) \\ + \left( \cos \left(2\pi(\nabla G_x , n_q)\right) -1 \right) \sin \left(2\pi(\nabla G , n_q)\right) - \left( \cos \left(2\pi(\nabla G , n_q)\right) -1 \right) \sin \left(2\pi(\nabla G_x , n_q)\right),
    \end{multline*}
and
    \begin{multline*}
  \cos \left(2\pi(\nabla G - \nabla G_x , n_q)\right) - 1 = \left( \cos \left(2\pi(\nabla G , n_q)\right) - 1 \right) \left( \cos \left(2\pi(\nabla G_x , n_q)\right) - 1 \right) \\ + \left( \cos \left(2\pi(\nabla G , n_q)\right) - 1 \right) + \left(\cos \left(2\pi(\nabla G_x , n_q)\right) - 1 \right)+ \sin \left(2\pi(\nabla G , n_q)\right) \sin \left(2\pi(\nabla G_x , n_q)\right)
    \end{multline*}
so as to obtain the equality
\begin{multline} \label{eq:twopointsimplified}
    \left\langle \cos \left( \theta(x) - \theta(0) \right) \right\rangle_{\mu_{\mathrm{Vil}, \beta}}
  \exp \left(\frac{1}{\beta} \left( G (0) - G(x) \right)\right) \\ 
  = \left\langle \exp \left( (U_0+ U_{\cos , 0}) - (U_x - U_{\cos , x}) + U_{\sin , \cos , x} - U_{\cos , \sin , x} + U_{\cos \cos , x} +  U_{\sin \sin , x}\right) \right\rangle_{\mu_\beta}.
\end{multline}
Similarly for the identity~\eqref{eq:10541104}, we may expand the sine and the cosine so as to obtain
\begin{multline*}
\left\langle \cos \left( \theta(x) + \theta(0) \right) \right\rangle_{\mu_{\mathrm{Vil}, \beta}}
  \exp \left(\frac{1}{\beta} \left( G (0) + G(x) \right)\right) \\ 
  = \left\langle \exp \left( (U_0+ U_{\cos , 0}) + (U_x + U_{\cos , x}) + U_{\sin , \cos , x} + U_{\cos , \sin , x} + U_{\cos \cos , x} -  U_{\sin \sin , x}\right) \right\rangle_{\mu_\beta}.
\end{multline*}
We next perform a series of simplifications: in Section~\ref{sec:3.3}, we show that the terms $U_{\sin , \cos , x}, U_{\cos , \sin , x}$ and $U_{\cos \cos , x}$ are lower order terms whose contribution can be removed, and in Section~\ref{sec:section3.4}, we treat the terms involving the Green's function and the term $U_{\sin \sin , x}$.

\subsection{Removing the lower order terms} \label{sec:3.3}

As mentioned above, we show in this section that the terms $U_{\sin , \cos , x}, U_{\cos , \sin , x}$ and $U_{\cos \cos , x}$ are lower order error terms. The specific statement we will use is the following upper bound on the supremum of these random variables:
\begin{equation} \label{eq:estlowerorderterms}
    \left\| U_{\sin , \cos , x} \right\|_{L^\infty(\Omega \, ; \, \R)} + \left\| U_{\cos , \sin , x} \right\|_{L^\infty(\Omega \, ; \, \R)} + \left\| U_{\cos \cos , x} \right\|_{L^\infty(\Omega \, ; \, \R)} \leq \frac{C}{|x|^{d-1}_+}.
\end{equation}
We refer to Appendix~\ref{app.notation} for the proof of these results. The typical size of these terms is smaller than the upper bound we wish to obtain in Theorem~\ref{th:mainth}. We thus show in this section that they can be discarded from the analysis. Combining the identity~\eqref{eq:twopointsimplified} with the upper bound~\eqref{eq:estlowerorderterms}, we see that
\begin{align*}
    \left\langle \cos \left( \theta(x) - \theta(0) \right) \right\rangle_{\mu_{\mathrm{Vil}, \beta}}
   & \leq \left(1+ \frac{C}{|x|^{d-1}} \right) \exp \left(-\frac{1}{\beta} \left( G (0) - G(x) \right)\right) \left\langle \exp \left( (U_0+ U_{\cos , 0}) - (U_x - U_{\cos , x}) +  U_{\sin \sin , x}\right) \right\rangle_{\mu_\beta} \\
   & \leq \exp \left(-\frac{1}{\beta} \left( G (0) - G(x) \right)\right) \left\langle \exp \left( (U_0+ U_{\cos , 0}) - (U_x - U_{\cos , x}) +  U_{\sin \sin , x}\right) \right\rangle_{\mu_\beta} + \frac{C}{|x|^{d-1}_+},
\end{align*}
and similarly
\begin{equation*}
    \left\langle \cos \left( \theta(x) + \theta(0) \right) \right\rangle_{\mu_{\mathrm{Vil}, \beta}} \leq  \exp \left(-\frac{1}{\beta} \left( G (0) + G(x) \right)\right) \left\langle \exp \left( (U_0+ U_{\cos , 0}) + (U_x + U_{\cos , x}) -  U_{\sin \sin , x}\right) \right\rangle_{\mu_\beta} + \frac{C}{|x|^{d-1}_+}.
\end{equation*}
A combination of the two previous inequalities with~\eqref{eq:14141102} implies that 
\begin{align} \label{eq:0949}
     \cov_{\mu_{\mathrm{Vil}, \beta}} \left[ \cos \theta(0) , \cos \theta(x) \right] \exp \left(\frac{G (0)}{\beta}  \right) & \leq \frac12 \exp \left(\frac{G(x)}{\beta} \right) \left\langle \exp \left( (U_0+ U_{\cos , 0}) - (U_x - U_{\cos , x}) +  U_{\sin \sin , x}\right) \right\rangle_{\mu_\beta} \\
     & \quad + \frac12 \exp \left(-\frac{G(x)}{\beta} \right) \left\langle \exp \left( (U_0+ U_{\cos , 0}) + (U_x + U_{\cos , x}) -  U_{\sin \sin , x}\right) \right\rangle_{\mu_\beta} \notag \\
     & \quad  -  \left( \left\langle \exp \left( U_0 + U_{\cos , 0} \right) \right\rangle_{\mu_\beta}\right)^2 + \frac{C}{|x|^{d-1}_+}. \notag
\end{align}

\subsection{Removing the exponential of the Green's function and the term $U_{\sin \sin}$} \label{sec:section3.4}

In this section, we further simplify the right-hand side of~\eqref{eq:0949} by removing the term involving the exponential of the Green's function and the random variable $U_{\sin \sin, x}$. Specifically, we start from the identity: for any $a ,  a' , b ,  b' \in \R$,
\begin{equation} \label{id.doubleproduct}
    a b +  a'  b' = \frac{1}{2}(a +  a') (b +  b') + \frac{1}{2} (a - a') (b - b'),
\end{equation}
together with the upper bound (proved in Appendix~\ref{app.notation})
\begin{equation} \label{ineq:varXsinsin}
    \left\| U_{\sin \sin, x} \right\|_{L^{\infty}(\Omega \, : \, \R)} \leq \frac{C}{|x|^{d-2}_+} \hspace{5mm} \mbox{and} \hspace{5mm} \mathrm{var}_{\mu_\beta}\left[ U_{\mathrm{sin \, sin}, x} \right] \leq \frac{C}{|x|_+^{2d-2}}.
\end{equation}
Let us note that a consequence of the first inequality is the estimate on the expectation
\begin{equation*}
    \left| \left\langle U_{\sin \sin, x} \right\rangle_{\mu_\beta} \right| \leq \frac{C}{|x|^{d-2}}.
\end{equation*}
The strategy is then to apply the identity~\eqref{id.doubleproduct} with (N.B. with this choice $a , a' , b, b'$ depend on $\varphi \in \Omega$ and are thus random variables and we note that they are non-negative and bounded in $L^2(\Omega \, ; \, \R)$ uniformly over $x \in \Zd$)
\begin{equation*}
    a = \exp \left(\frac{G(x)}{\beta} + U_{\sin \sin , x}\right) ~~\mbox{and}~~ b =  \exp \left( (U_0+ U_{\cos , 0}) - (U_x - U_{\cos , x}) \right),
\end{equation*}
as well as
\begin{equation*}
    a' = \exp \left(-\frac{G(x)}{\beta} -  U_{\sin \sin , x} \right)  ~~\mbox{and}~~  b' =  \exp \left( (U_0+ U_{\cos , 0}) + (U_x + U_{\cos , x}) \right).
\end{equation*}
Let us note that, by performing a Taylor expansion of the exponential to the second order, we have the inequality
\begin{equation*}
    a + a' = \exp \left(\frac{G(x)}{\beta} + U_{\sin \sin , x}\right) + \exp \left( - \frac{G(x)}{\beta} - U_{\sin \sin , x}\right) \leq 2 + \frac{C G(x)^2}{\beta^2} + C \left\| U_{\sin \sin} \right\|_{L^\infty(\Omega \, : \, \R)}^2 \leq 2 + \frac{C}{|x|^{2d-4}_+}.
\end{equation*}
Taking the expectation with respect to the measure $\mu_\beta$ in the previous inequalities, we deduce that
\begin{align*}
     \lefteqn{\frac{1}{2} \left\langle(a +  a') (b +  b') \right\rangle_{\mu_\beta}} \qquad & \\ & \leq \left(1 + \frac{C}{|x|^{2d-4}_+} \right) \left(\left\langle \exp \left( (U_0+ U_{\cos , 0}) - (U_x - U_{\cos , x})\right) \right\rangle_{\mu_\beta} + \left\langle \exp \left( (U_0+ U_{\cos , 0}) + (U_x + U_{\cos , x}) \right) \right\rangle_{\mu_\beta} \right) \\
     & \leq \left\langle \exp \left( (U_0+ U_{\cos , 0}) - (U_x - U_{\cos , x})\right) \right\rangle_{\mu_\beta} + \left\langle \exp \left( (U_0+ U_{\cos , 0}) + (U_x + U_{\cos , x}) \right) \right\rangle_{\mu_\beta} + \frac{C}{|x|^{2d-4}_+}.
\end{align*}
Similarly, we have, using a second Taylor expansion of the exponential,
\begin{align*}
    a - a' & = \exp \left(\frac{G(x)}{\beta} + U_{\sin \sin , x}\right) - \exp \left( - \frac{G(x)}{\beta} - U_{\sin \sin , x}\right) \\
    & =\frac{2 G(x)}{\beta} + 2 U_{\sin \sin} + O \left( \frac{1}{|x|^{2d-4}_+} \right) \\
    & = \underset{\mathrm{Deterministic \, part}}{\underbrace{\frac{2 G(x)}{\beta} + 2 \left\langle U_{\sin \sin} \right\rangle_{\mu_\beta}}} + \underset{\mathrm{Random \, part \, with \, small \,} L^2(\Omega \, ; \, \R) \, \mathrm{norm}}{\underbrace{ 2 \left( U_{\sin \sin}  -  \left\langle U_{\sin \sin} \right\rangle_{\mu_\beta} \right) + O \left( \frac{1}{|x|^{2d-4}_+} \right).}}
\end{align*}
Applying the Cauchy-Schwarz inequality, the upper bound on the variance~\eqref{ineq:varXsinsin} and the inequality $2d-4 \geq d-1$ (since $d \geq 3$), we may estimate
\begin{align*}
    \frac{1}{2} \left\langle(a -  a') (b -  b') \right\rangle_{\mu_\beta} & \leq \left(  \frac{2 G(x)}{\beta} + 2 \left\langle U_{\sin \sin} \right\rangle_{\mu_\beta} \right) \left\langle b - b' \right\rangle_{\mu_\beta} + C \sqrt{\mathrm{var}_{\mu_\beta}\left[ U_{\mathrm{sin \, sin}, x} \right]} + \frac{C}{|x|_+^{2d-4}} \\
    & \leq \left(  \frac{2 G(x)}{\beta} + 2 \left\langle U_{\sin \sin} \right\rangle_{\mu_\beta} \right) \left\langle b - b' \right\rangle_{\mu_\beta} +\frac{C}{|x|^{d-1}_+}.
\end{align*}
Using that the Green's function $G$ and the expectation of the random variable $U_{\sin \sin}$ decay like $|x|^{2-d}_+$, we deduce that
\begin{multline*}
    \frac{1}{2} \left\langle(a -  a') (b -  b') \right\rangle_{\mu_\beta} \\ \leq \frac{C}{|x|^{d-2}_+} \left| \left\langle \exp \left( (U_0+ U_{\cos , 0}) - (U_x - U_{\cos , x})\right) \right\rangle_{\mu_\beta} - \left\langle \exp \left( (U_0+ U_{\cos , 0}) + (U_x + U_{\cos , x}) \right) \right\rangle_{\mu_\beta} \right| + \frac{C}{|x|^{d-1}_+}.
\end{multline*}
A combination of the two previous displays implies that
\begin{multline*}
    \exp \left(\frac{G(x)}{\beta} \right) \left\langle \exp \left( (U_0+ U_{\cos , 0}) - (U_x - U_{\cos , x}) +  U_{\sin \sin , x}\right) \right\rangle_{\mu_\beta} \\ + \exp \left(-\frac{G(x)}{\beta} \right) \left\langle \exp \left( (U_0+ U_{\cos , 0}) + (U_x + U_{\cos , x}) -  U_{\sin \sin , x}\right) \right\rangle_{\mu_\beta} \\
    \leq \left\langle \exp \left( (U_0+ U_{\cos , 0}) - (U_x - U_{\cos , x})\right) \right\rangle_{\mu_\beta} + \left\langle \exp \left( (U_0+ U_{\cos , 0}) + (U_x + U_{\cos , x}) \right) \right\rangle_{\mu_\beta}  \\
    + \frac{C}{|x|^{d-2}_+} \left| \left\langle \exp \left( (U_0+ U_{\cos , 0}) - (U_x - U_{\cos , x})\right) \right\rangle_{\mu_\beta} - \left\langle \exp \left( (U_0+ U_{\cos , 0}) + (U_x + U_{\cos , x}) \right) \right\rangle_{\mu_\beta} \right| + \frac{C}{|x|_+^{d-1}}.
\end{multline*}
Combining the previous inequality with~\eqref{eq:0949}, we obtain that
\begin{align*}
\lefteqn{\cov_{\mu_{\mathrm{Vil}, \beta}} \left[ \cos \theta(0) , \cos \theta(x) \right] \exp \left(\frac{G (0)}{\beta}  \right)} \qquad &\\ &
\leq \frac12 \left\langle \exp \left( (U_0+ U_{\cos , 0}) - (U_x - U_{\cos , x})\right) \right\rangle_{\mu_\beta} + \frac12 \left\langle \exp \left( (U_0+ U_{\cos , 0}) + (U_x + U_{\cos , x}) \right) \right\rangle_{\mu_\beta} -  \left\langle \exp \left( U_0 + U_{\cos , 0} \right) \right\rangle_{\mu_\beta}^2 \\
& \quad + \frac{C}{|x|^{d-2}_+} \left| \left\langle \exp \left( (U_0+ U_{\cos , 0}) - (U_x - U_{\cos , x})\right) \right\rangle_{\mu_\beta} - \left\langle \exp \left( (U_0+ U_{\cos , 0}) + (U_x + U_{\cos , x}) \right) \right\rangle_{\mu_\beta} \right| \\
& \quad + \frac{C}{|x|^{d-1}_+}.
\end{align*}
From this inequality, we see that the upper bound in Theorem~\ref{th:mainth} is a consequence of the following two inequalities (to which the rest of the article is devoted):
\begin{multline} \label{eq:untreatedcov}
    \frac12 \left\langle \exp \left( (U_0+ U_{\cos , 0}) - (U_x - U_{\cos , x})\right) \right\rangle_{\mu_\beta} + \frac12 \left\langle \exp \left( (U_0+ U_{\cos , 0}) + (U_x + U_{\cos , x}) \right) \right\rangle_{\mu_\beta} - \left\langle \exp \left( U_0 + U_{\cos , 0} \right) \right\rangle_{\mu_\beta}^2  \\ \leq \frac{C \left(\ln  |x|_+\right)^{4d+10}}{ |x|^{d-1}_+}
\end{multline}
and
\begin{equation} \label{eq:untreatedremaining}
    \left| \left\langle \exp \left( (U_0+ U_{\cos , 0}) - (U_x - U_{\cos , x})\right) \right\rangle_{\mu_\beta} - \left\langle \exp \left( (U_0+ U_{\cos , 0}) + (U_x + U_{\cos , x}) \right) \right\rangle_{\mu_\beta} \right| \leq \frac{C(\ln |x|_+)^{d+3}}{|x|^{d-2}_+}.
\end{equation}
Throughout the following computations, we keep track of the exponent on the logarithmic term for completeness. However, we do not attempt to optimise its precise value, and the reader is encouraged to regard it simply as ``a logarithm raised to some power,” without placing too much emphasis on the specific exponent.

\subsection{Applying the Helffer-Sj{\"o}strand representation formula}
The next step of the argument is to prove the inequalities~\eqref{eq:untreatedcov} and~\eqref{eq:untreatedremaining}. These two inequalities only involve the random interface measure $\mu_\beta$, and they can be established by using the Helffer-Sj\"{o}strand representation formula introduced in Section~\ref{sec:HSrepre}. Let us first note that
\begin{multline*}
    \left\langle \exp \left( (U_0+ U_{\cos , 0}) - (U_x - U_{\cos , x})\right) \right\rangle_{\mu_\beta} + \left\langle \exp \left( (U_0+ U_{\cos , 0}) + (U_x + U_{\cos , x}) \right) \right\rangle_{\mu_\beta} \\
    = 2 \left\langle \exp \left( U_0+ U_{\cos , 0} \right) \mathrm{ch}(U_x) \exp(U_{\cos , x})) \right\rangle_{\mu_\beta}
\end{multline*}
and that, due to the $\varphi \mapsto - \varphi$ symmetry of the measure $\mu_\beta$ (since $U_x$ is an odd function of $\varphi$ and $U_{\cos , 0}$ is an even function of $\varphi$) together with the translation invariance of the measure $\mu_\beta$, we have the identities
\begin{equation} \label{eq:3.10}
    \left\langle \exp \left( - U_x + U_{\cos , x})\right) \right\rangle_{\mu_\beta} = \left\langle \exp \left( U_x + U_{\cos , x})\right) \right\rangle_{\mu_\beta}  = \left\langle \exp \left( U_0 + U_{\cos , 0})\right) \right\rangle_{\mu_\beta}.
\end{equation}
From the two previous identities, we see that~\eqref{eq:untreatedcov} is equivalent to
\begin{equation} \label{eq:09071202}
    \cov_{\mu_\beta} \left[ \exp \left( U_0 + U_{\cos , 0} \right),  \cosh{(U_x)} \exp \left( U_{\cos , x} \right) \right] \leq \frac{C \left(\ln  |x|_+\right)^{4d+10}}{ |x|^{d-1}_+}.
\end{equation}
Similarly (and using~\eqref{eq:3.10} a second time), we see that~\eqref{eq:untreatedremaining} is implied by the two upper bounds
\begin{align} \label{eq:09081202}
    \left| \cov_{\mu_\beta}  \left[ \exp \left( U_0+ U_{\cos , 0} \right) ,  \exp \left( - U_x + U_{\cos , x} \right) \right] \right| & \leq \frac{C(\ln |x|_+)^{d+3}}{|x|^{d-2}_+},  \\\
    \left| \cov_{\mu_\beta}  \left[ \exp \left( U_0+ U_{\cos , 0} \right) ,  \exp \left(  U_x + U_{\cos , x} \right) \right] \right| & \leq \frac{C(\ln |x|_+)^{d+3}}{|x|^{d-2}_+}. \notag 
\end{align}

Therefore, the upper bound in Theorem ~\ref{th:mainth} follows from the three upper bounds stated in~\eqref{eq:09071202} and~\eqref{eq:09081202}.
We devote the rest of this section to the proof of these upper bounds using the Helffer-Sj\"{o}strand representation formula, from which we conclude the  upper bound in Theorem ~\ref{th:mainth}. The inequalities~\eqref{eq:09081202} are strictly simpler to show than the inequality~\eqref{eq:09071202}, and we thus start with their proofs.

\subsubsection{Proof of the upper bound~\eqref{eq:09081202}}

We will apply the Helffer-Sj\"{o}strand representation formula (Proposition~\ref{prop:HSrepresetnation}) with the functions:
\begin{equation*}
    \mathbf{f}(\varphi) := \exp \left( U_0 (\varphi) + U_{\cos , 0} (\varphi) \right) ~~\mbox{and}~~ \mathbf{h}(\varphi) := \exp \left( - U_x(\varphi) + U_{\cos , x}(\varphi) \right).
\end{equation*}
Let us note that the partial derivative of the functions $\mathbf{f}$ and $\mathbf{g}$ can be explicitly computed, and that we have the identities, for any $y \in \Zd$,
\begin{equation*}
    \begin{aligned}
    \partial_y \mathbf{f}(\varphi) & = 2\pi \exp \left( U_0(\varphi) + U_{\cos , 0}(\varphi)\right)\sum_{q \in \mathcal{Q}} z(\beta , q)  \cos \left( 2\pi (\di^* \varphi , n_q)\right) \sin \left( 2\pi(\nabla G , n_q) \right)  q(y) \\
    &  \quad -  2\pi \exp \left( U_0(\varphi) + U_{\cos , 0}(\varphi)\right) \sum_{q \in \mathcal{Q}} z(\beta , q)  \sin \left( 2\pi (\di^* \varphi , n_q)\right) (\cos \left( 2\pi(\nabla G , n_q) \right) - 1 ) q(y), \\
    \partial_y \mathbf{h}(\varphi) &  = - 2\pi \exp \left( - U_x(\varphi) + U_{\cos , x}(\varphi)\right)\sum_{q \in \mathcal{Q}} z(\beta , q)  \cos \left( 2\pi (\di^* \varphi , n_q)\right) \sin \left( 2\pi(\nabla G_x , n_q) \right)  q(y) \\
    &  \quad -  2\pi \exp \left( U_x(\varphi) + U_{\cos , x}(\varphi)\right) \sum_{q \in \mathcal{Q}} z(\beta , q)  \sin \left( 2\pi (\di^* \varphi , n_q)\right) (\cos \left( 2\pi(\nabla G_x , n_q) \right) - 1 ) q(y).
    \end{aligned}
\end{equation*}
Applying the Helffer-Sj\"{o}strand representation formula (Proposition~\ref{prop:HSrepresetnation}), we thus obtain that
\begin{align} \label{eq:102412021}
    \lefteqn{\cov_{\mu_\beta}  \left[ \exp \left( U_0+ U_{\cos , 0} \right) ,  \exp \left( - U_x + U_{\cos , x} \right) \right]} \qquad & \\ &
    = - 2\pi \sum_{q \in \mathcal{Q}}   z(\beta, q) \sin \left( 2\pi(\nabla G_x , n_q) \right) \langle \exp {(-U_x(\varphi) + U_{\cos , x}(\varphi))} \cos \left( 2\pi(\di^* \varphi , n_q) \right)  ((\mathcal{G}_0 + \mathcal{G}_1)(\cdot ,\varphi) , q) \rangle_{\mu_\beta}  \notag \\
    & \quad  - 2\pi \sum_{q \in \mathcal{Q}}   z(\beta, q) (\cos \left( 2\pi(\nabla G_x , n_q) \right) -1) \langle \exp{(-U_x(\varphi) + U_{\cos , x}(\varphi))} \sin \left( 2\pi(\di^* \varphi , n_q) \right)  ((\mathcal{G}_0 + \mathcal{G}_1)(\cdot ,\varphi) ,q) \rangle_{\mu_\beta}, \notag
\end{align}
where $\mathcal{G}_0, \mathcal{G}_1 : \Zd \times \Omega \to \R^{2 \choose d}$ are the solutions of the Helffer-Sj\"{o}strand equations
\begin{align} \label{def.G0andG1}
    \mathcal{L} \mathcal{G}_0 & = 2\pi \exp \left( U_0(\varphi) + U_{\cos , 0}(\varphi)\right)\sum_{q \in \mathcal{Q}} z(\beta , q)  \cos \left( 2\pi (\di^* \varphi , n_q)\right) \sin \left( 2\pi(\nabla G , n_q) \right)  q, \\
    \mathcal{L} \mathcal{G}_1 & =  - 2\pi \exp \left( U_0(\varphi) + U_{\cos , 0}(\varphi) \right)\sum_{q \in \mathcal{Q}} z(\beta , q)  \sin \left( 2\pi (\di^* \varphi , n_q)\right) (\cos \left( 2\pi(\nabla G , n_q) \right) - 1 )  q. \notag 
\end{align}
Using the identity $q = \di n_q$ (see Lemma~\ref{lem:lem2.4}) together with~\eqref{identity:dandd*}, we may rewrite~\eqref{eq:102412021} as follows
\begin{align} \label{eq:10241202}
    \lefteqn{\cov_{\mu_\beta}  \left[ \exp \left( U_0+ U_{\cos , 0} \right) ,  \exp \left( - U_x + U_{\cos , x} \right) \right]} \qquad & \\ &
    = \underset{\eqref{eq:10241202}-(i)}{\underbrace{-2\pi \sum_{q \in \mathcal{Q}}   z(\beta, q) \sin \left( 2\pi(\nabla G_x , n_q) \right) \langle \exp {(-U_x(\varphi) + U_{\cos , x}(\varphi))} \cos \left( 2\pi(\di^* \varphi , n_q) \right)  (\di^* (\mathcal{G}_0 + \mathcal{G}_1)(\cdot ,\varphi) , n_q) \rangle_{\mu_\beta}}}  \notag \\
    & \quad  \underset{\eqref{eq:10241202}-(ii)}{\underbrace{- 2\pi \sum_{q \in \mathcal{Q}}   z(\beta, q) (\cos \left( 2\pi(\nabla G_x , n_q) \right) -1) \langle \exp{(-U_x(\varphi) + U_{\cos , x}(\varphi))} \sin \left( 2\pi(\di^* \varphi , n_q) \right)  (\di^*  (\mathcal{G}_0 + \mathcal{G}_1)(\cdot ,\varphi) , n_q) \rangle_{\mu_\beta}}}. \notag
\end{align}
The rest of this section is decomposed in two steps. In the first one, we prove the following inequalities on the $L^2(\Omega \, ; \, \R^d)$-norm of the codifferential of the functions $\mathcal{G}_0 $ and $\mathcal{G}_1$, for any $y \in \Zd$,
\begin{equation}
\label{e.gradG0G1}
    \left\| \di^* \mathcal{G}_0(y , \cdot) \right\|_{L^2(\Omega \, ; \, \R^d)} \leq \frac{C ( \ln |y|_+ )^{d+3}}{|y|^{d-1}_+} ~~ \mbox{and}~~ \left\| \di^* \mathcal{G}_1(y , \cdot) \right\|_{L^2(\Omega \, ; \, \R^d)} \leq \frac{C ( \ln |y|_+ )^{d+2} }{|y|^{d}_+},
\end{equation}
and, in the second step, we deduce from these inequalities an upper bound on the terms~\eqref{eq:10241202}-(i) and~\eqref{eq:10241202}-(ii) (which then implies the inequalities~\eqref{eq:09081202}).

\medskip

\textit{Step 1. Proof of the inequality~\eqref{e.gradG0G1}.}

\medskip

Let us first set, for $q \in \mathcal{Q}$,
\begin{equation*}
    \mathbf{f}_q(\varphi) := \exp \left( U_0(\varphi) + U_{\cos , 0}(\varphi)\right) \cos \left( 2\pi (\di^* \varphi , n_q)\right) \in \R.
\end{equation*}
Note that this random variable belongs to $L^2(\Omega \, ; \, \R)$ (see Remark~\ref{rem:remark3.3}).
We next write an identity relating the function $\mathcal{G}_0$ and the Green's matrix associated with the Helffer-Sj\"{o}strand operator introduced in Definition~\ref{eq:defHSGREEN}. Specifically, we use the identity~\eqref{id:decomgeneralsolutions} of Remark~\ref{rem:remark2.21} (and the linearity of the map $\mathbf{f} \mapsto \mathcal{G}_{\mathbf{f}}$) which in this case implies  
\begin{align*}
    \mathcal{G}_0(x , \varphi)   & =  2\pi \sum_{q \in \mathcal{Q}} z(\beta , q)  \sin \left( 2\pi(\nabla G , n_q) \right) \left( \mathcal{G}_{\mathbf{f}_q}(x , \varphi ; \cdot), q \right) \\
    & =  2\pi \sum_{q \in \mathcal{Q}} z(\beta , q)  \sin \left( 2\pi(\nabla G , n_q) \right) \left( \di^*_2 \mathcal{G}_{\mathbf{f}_q}(x , \varphi ; \cdot), n_q \right) \in \R^{{d \choose 2}}.
\end{align*}
We further deduce that
\begin{equation*}
    \di^* \mathcal{G}_0(x , \varphi) =  2\pi \sum_{q \in \mathcal{Q}} z(\beta , q)  \sin \left( 2\pi(\nabla G , n_q) \right) \left( \di^*_1 \di^*_2 \mathcal{G}_{\mathbf{f}_q}(x , \varphi ; \cdot), n_q \right) \in \Rd.
\end{equation*}
Using the regularity estimate on the Green's matrix stated in Proposition~\ref{prop.prop4.11chap4} together with Lemma~\ref{lemma.lemma2.5}, the first inequality of Proposition~\ref{prop:appCdiscretesum} and the inequality
$$|\sin \left( 2\pi(\nabla G , n_q) \right) |\leq C \left| (\nabla G , n_q) \right|, $$ we deduce that
\begin{align*}
    \left\| \di^* \mathcal{G}_0(x , \cdot) \right\|_{L^2(\Omega \, ; \, \R^d)} \leq C \sum_{q \in \mathcal{Q}} |z(\beta , q)| \left|(\nabla G , n_q) \right| \left\| \left( \di^*_1 \di^*_2 \mathcal{G}_{\mathbf{f}_q}(x , \varphi ; \cdot), n_q \right) \right\|_{L^2 \left( \Omega \, ; \, \R \right)} & \leq C \sum_{z \in \Zd} \frac{C \left( \ln |x - z|_+ \right)^{d+2} }{|z|^{d-1}_+ |x-z|^{d}_+} \\
    & \leq \frac{C  \left( \ln |x|_+ \right)^{d+3}}{|x|^{d-1}_+}. \notag
\end{align*}
We proceed similarly to estimate the function $\mathcal{G}_1$, this time setting
\begin{equation*}
    \mathbf{g}_q(\varphi) :=  \exp \left( U_0(\varphi) + U_{\cos , 0}(\varphi)\right) \sin \left( 2\pi (\di^* \varphi , n_q)\right) 
\end{equation*}
so that 
\begin{equation*}
    \di^* \mathcal{G}_1(x , \varphi) = - 2\pi \sum_{q \in \mathcal{Q}} z(\beta , q)  \left( \cos \left( 2\pi(\nabla G , n_q) \right) - 1\right) \left( \di^*_1 \di^*_2 \mathcal{G}_{\mathbf{g}_q}(x , \varphi ; \cdot), n_q \right) \in \Rd.
\end{equation*}
Using the regularity estimate on the Green's matrix associated with the Helffer-Sj\"{o}strand operator (Proposition~\ref{prop.prop4.11chap4}) together with Lemma~\ref{lemma.lemma2.5} and the upper bound $$1-\cos \left( 2\pi(\nabla G , n_q) \right) \leq C \left|(\nabla G , n_q) \right|^2,$$ we deduce that
\begin{align*}
    \left\| \di^* \mathcal{G}_1(x , \cdot) \right\|_{L^2(\Omega \, ; \, \Rd)} \leq C \sum_{q \in \mathcal{Q}} |z(\beta , q)| \left|(\nabla G , n_q) \right|^2 \left\| \left( \di^*_1 \di^*_2 \mathcal{G}_{\mathbf{g}_q}(x , \cdot ; \cdot), n_q \right) \right\|_{L^2(\Omega \, ; \, \Rd)} & \leq C \sum_{z \in \Zd} \frac{C (\ln |x - z|_+)^{d+2} }{|z|^{2d-2}_+ |x-z|^{d}_+} \\
    &  \leq \frac{C(\ln |x|_+)^{d+2}}{|x|^{d}_+}.
\end{align*}

\medskip

\textit{Step 2. Upper bounds on the terms~\eqref{eq:10241202}-(i) and~\eqref{eq:10241202}-(ii).}

\medskip

Combining this inequality with Lemma~\ref{lemma.lemma2.5}, we may estimate the first term on the right-hand side of~\eqref{eq:10241202}
\begin{align*}
   \left| \eqref{eq:10241202}-(i) \right| 
    \leq C \sum_{q \in \mathcal{Q}} z(\beta , q) \left| (\nabla G_x , n_q) \right| \left\| (\di^* \mathcal{G}_0 + \di^* \mathcal{G}_1 , n_q)  \right\|_{L^2(\Omega \, ; \, \R)}  & \leq C \sum_{z \in \Zd} \frac{1}{|x - z|^{d-1}_+} \frac{(\ln |z|_+)^{d+3}}{|z|^{d-1}_+} \\ 
    & \leq \frac{C(\ln |x|_+)^{d+3} }{|x|^{d-2}_+}.
\end{align*}
The second term on the right-hand side of~\eqref{eq:10241202} is estimated in a similar manner. We obtain
\begin{align*}
   \left| \eqref{eq:10241202}-(ii) \right|  
    \leq C \sum_{q \in \mathcal{Q}} |z(\beta , q)| \left| (\nabla G_x , n_q) \right|^2 \left\| (\di^* \mathcal{G}_0 + \di^* \mathcal{G}_1 , n_q)  \right\|_{L^2(\Omega \, ; \, \R)}  & \leq C \sum_{z \in \Zd} \frac{1}{|x - z|^{2d-2}_+} \frac{(\ln |z|_+)^{d+3}}{|z|^{d-1}_+} \\
    & \leq \frac{C(\ln |x|_+)^{d+3}}{|x|^{d-1}_+}.
\end{align*}
Combining the two previous inequalities with~\eqref{eq:10241202}, we deduce that
\begin{equation*}
    \cov_{\mu_\beta}  \left[ \exp \left( U_0+ U_{\cos , 0} \right) ,  \exp \left( - U_x + U_{\cos , x} \right) \right] \leq \frac{C(\ln |x|_+)^{d+3} }{|x|^{d-2}_+}
\end{equation*}
which is the first inequality of~\eqref{eq:09081202}. The proof of the second inequality of~\eqref{eq:09081202} is essentially identical (N.B. the only difference is that we need to replace $-U_x$ by $U_x$ in the expressions above and add a minus sign in front of the term~\eqref{eq:10241202}-(i)) and the details are thus omitted.

\subsubsection{Proof of the upper bound~\eqref{eq:09071202}- Part I: A first application of the Helffer-Sj\"{o}strand representation}

We first apply the Helffer-Sj{\"o}strand representation (exactly as in~\eqref{eq:102412021}) to obtain
    \begin{align*}
& \cov_{\mu_\beta} \left[ \exp \left( U_0 + U_{\cos , 0} \right),  \cosh{(U_x)} \exp \left( U_{\cos , x} \right) \right]
   \\ & \qquad = 
    2\pi \sum_{q \in \mathcal{Q}}   z(\beta, q) \sin \left( 2\pi(\nabla G_x , n_q) \right) \langle \sinh{(U_x)}  \exp \left( U_{\cos , x} \right) \cos \left( 2\pi(\di^* \varphi , n_q) \right)  (\di^* (\mathcal{G}_0 + \mathcal{G}_1)(\cdot ,\varphi) , n_q) \rangle_{\mu_\beta} \notag \\
    & - 2\pi \sum_{q \in \mathcal{Q}}   z(\beta, q) (\cos \left( 2\pi(\nabla G_x , n_q) \right) -1) \langle \cosh{(U_x)} \exp \left( U_{\cos , x} \right) \sin \left( 2\pi(\di^* \varphi , n_q) \right)  (\di^*  (\mathcal{G}_0 + \mathcal{G}_1)(\cdot ,\varphi) , n_q) \rangle_{\mu_\beta} \notag
    \end{align*}
where $\mathcal{G}_0$ and $\mathcal{G}_1$ are the two functions defined in~\eqref{def.G0andG1}. For later purposes, we split the first line of the previous inequality into two (isolating two terms: one involving the function $\mathcal{G}_0$ and one involving the function~$\mathcal{G}_1$)
\begin{align}
        \label{e.2pt1.5}
& \cov_{\mu_\beta} \left[ \exp \left( U_0 + U_{\cos , 0} \right),  \cosh{(U_x)} \exp \left( U_{\cos , x} \right) \right]
   \\ & \qquad = 
    2\pi \underset{\eqref{e.2pt1.5}-(i)}{\underbrace{\sum_{q \in \mathcal{Q}}   z(\beta, q) \sin \left( 2\pi(\nabla G_x , n_q) \right) \langle \sinh{(U_x)}  \exp \left( U_{\cos , x} \right) \cos \left( 2\pi(\di^* \varphi , n_q) \right)   (\di^* \mathcal{G}_0(\cdot ,\varphi) , n_q) \rangle_{\mu_\beta}}} \notag \\
    & + 2\pi \underset{\eqref{e.2pt1.5}-(ii)}{\underbrace{\sum_{q \in \mathcal{Q}}   z(\beta, q) \sin \left( 2\pi(\nabla G_x , n_q) \right) \langle \sinh{(U_x)}  \exp \left( U_{\cos , x} \right) \cos \left( 2\pi(\di^* \varphi , n_q) \right)  (\di^* \mathcal{G}_1(\cdot ,\varphi) , n_q) \rangle_{\mu_\beta}}} \notag \\
    & + 2\pi \underset{\eqref{e.2pt1.5}-(iii)}{\underbrace{\sum_{q \in \mathcal{Q}}   z(\beta, q) (\cos \left( 2\pi(\nabla G_x , n_q) \right) -1) \langle \cosh{(U_x)} \exp \left( U_{\cos , x} \right) \sin \left( 2\pi(\di^* \varphi , n_q) \right)  (\di^*  (\mathcal{G}_0 + \mathcal{G}_1)(\cdot ,\varphi) , n_q) \rangle_{\mu_\beta}}.} \notag
    \end{align}
Let us then fix a 2-form $q \in \mathcal{Q}$ and study the term (which appears in~\eqref{e.2pt1.5}-(i))
\begin{equation} \label{eq:eq5.4}
    \langle \sinh{(U_x)} \exp \left( U_{\cos , x} \right) \cos \left( 2\pi(\di^* \varphi , n_q) \right)  (\di^* \mathcal{G}_0(\cdot ,\varphi) , n_q) \rangle_{\mu_\beta}.
\end{equation}
We will prove in Section~\ref{subsec5.3} below the upper bound: for any $q \in \mathcal{Q}$, 
\begin{multline} \label{ineq:maintechnical}
    \left| \langle \sinh{(U_x)} \exp \left( U_{\cos , x} \right) \cos \left( 2\pi(\di^* \varphi , n_q) \right)  (\di^* \mathcal{G}_0(\cdot ,\varphi) , n_q) \rangle_{\mu_\beta} \right| \\ 
    \leq C_q \frac{\left(\ln  |z_q|_+\right)^{3d+7}}{ |z_q|^{d-1}_+} \frac{\left( \ln |x - z_q|_+ \right)^{d+2}}{|x - z_q|^{d-2}_+} +  C_q \frac{\left( \ln |z_q|_+ \right)^{d+7/2}}{|z_q|^{d-1}_+} \frac{\left( \ln |x|_+ \right)^{2d+13/2}}{|x|^{d-2}_+},
\end{multline}
where $C_q:= C_q(d , \beta , q) < \infty$ is a constant which depends on $d , \beta$ and $q$ and shall only grow polynomially fast in $\left\| q \right\|_1$ (i.e., there exist a constant $C := (d , \beta)< \infty$ depending only on $d$ and $\beta$ and an exponent $k := k(d) < \infty$ depending only on $d$ such that $C_q \leq C \left\| q \right\|_1^k$) and $z_q \in \Zd$ is a vertex of the support of $q$ (minimising the lexicographical order).

The inequality~\eqref{ineq:maintechnical} is the core of the argument, and we first show how the inequality \eqref{eq:09071202} can be deduced from it. We first note that
\begin{align*}
    \left| \eqref{e.2pt1.5}-(i) \right| &
    \leq  \sum_{q \in \mathcal{Q}} C_q |z(\beta, q)| \frac{1}{|x - z_q|^{d-1}_+} \frac{\left(\ln  |z_q|_+\right)^{3d+7}}{ |z_q|^{d-1}_+} \frac{\left( \ln |x - z_q|_+ \right)^{d+2}}{|x - z_q|^{d-2}_+}  \\
    & \quad +  \sum_{q \in \mathcal{Q}} C_q |z(\beta, q)| \frac{1}{|x - z_q|^{d-1}_+} \frac{\left( \ln |z_q|_+ \right)^{d+7/2}}{|z_q|^{d-1}_+} \frac{\left( \ln |x|_+ \right)^{2d+13/2}}{|x|^{d-2}_+} \\
    & = \sum_{q \in \mathcal{Q}} C_q |z(\beta, q)| \frac{\left(\ln  |z_q|_+\right)^{3d+7}}{ |z_q|^{d-1}_+} \frac{\left( \ln |x - z_q|_+ \right)^{d+2}}{|x - z_q|^{2d-3}_+} \\
    & \quad + \frac{\left( \ln |x|_+ \right)^{2d+13/2}}{|x|^{d-2}_+}\sum_{q \in \mathcal{Q}}  C_q |z(\beta, q)| \frac{1}{|x - z_q|^{d-1}_+} \frac{\left( \ln |z_q|_+ \right)^{d+7/2}}{|z_q|^{d-1}_+} .
\end{align*}
Using Lemma~\ref{lemma.lemma2.5} together with the inequality $2d - 3 \geq d$ (since $d \geq 3$), we deduce that
\begin{align*}
        \left| \eqref{e.2pt1.5}-(i) \right| & \leq C \sum_{z \in \Zd}\frac{\left(\ln  |z|_+\right)^{3d+7}}{ |z|^{d-1}_+} \frac{\left( \ln |x - z|_+ \right)^{d+2}}{|x - z|^{2d-3}_+} + \frac{C \left( \ln |x|_+ \right)^{2d+13/2}}{|x|^{d-2}_+}\sum_{z \in \Zd} \frac{1}{|x - z|^{d-1}_+} \frac{\left( \ln |z|_+ \right)^{d+7/2}}{|z|^{d-1}_+} \\
    & \leq C \frac{\left(\ln  |x|_+\right)^{4d+10}}{ |x|^{d-1}_+} + C \frac{\left( \ln |x|_+ \right)^{3d+10}}{|x|^{2d-4}_+} \\
    & \leq C \frac{\left(\ln  |x|_+\right)^{4d+10}}{ |x|^{d-1}_+} .
\end{align*}
We next estimate the simpler terms~\eqref{e.2pt1.5}-(ii) and \eqref{e.2pt1.5}-(iii). For \eqref{e.2pt1.5}-(ii), we apply \eqref{e.gradG0G1} and Lemma~\ref{lemma.lemma2.5} to obtain
\begin{equation*}
    \left| \eqref{e.2pt1.5}-(ii) \right| \leq C \sum_{q \in \mathcal{Q}} |z(\beta , q)| \left| (\nabla G_x , n_q) \right| \left\| (\di^* \mathcal{G}_1 , n_q)  \right\|_{L^2(\Omega \, ; \, \R)}  \leq \sum_{z \in \Zd} \frac{C(\ln |z|_+)^{d+2}}{|x - z|^{d-1}_+|z|^{d}_+} \leq \frac{C(\ln |x|_+)^{d+3} }{|x|^{d-1}_+}.
\end{equation*}
Likewise, for the term~\eqref{e.2pt1.5}-(iii), we apply \eqref{e.gradG0G1}
and Lemma~\ref{lemma.lemma2.5} to obtain
\begin{equation*}
    \left| \eqref{e.2pt1.5}-(iii) \right| \leq C \sum_{q \in \mathcal{Q}} |z(\beta , q)| \left| (\nabla G_x , n_q) \right|^2 \left\| (\di^* \mathcal{G}_0+ \di^* \mathcal{G}_1 , n_q)  \right\|_{L^2(\Omega \, ; \, \R)}  \leq \sum_{z \in \Zd} \frac{C (\ln |z|_+)^{d+3}}{|x - z|^{2d-2}_+ |z|^{d-1}_+} \leq \frac{C(\ln |x|_+)^{d+3} }{|x|^{d-1}_+}.
\end{equation*}
A combination of the previous inequalities with~\eqref{e.2pt1.5} implies
\begin{equation*}
    \left| \cov_{\mu_\beta} \left[ \exp \left( U_0 + U_{\cos , 0} \right),  \cosh{(U_x)} \exp \left( U_{\cos , x} \right) \right] \right| \leq  C \frac{\left(\ln  |x|_+\right)^{4d+10}}{ |x|^{d-1}_+}.
\end{equation*}
This completes the proof of~\eqref{eq:09071202} conditionally on the inequality~\eqref{ineq:maintechnical}.

\subsubsection{Proof of the upper bound~\eqref{eq:09071202}- Part II: Proof of the inequality~\eqref{ineq:maintechnical} using a second application of the Helffer-Sj{\"o}strand representation} \label{subsec5.3}

In order to estimate the term~\eqref{eq:eq5.4}, we apply the Helffer-Sj{\"o}strand representation (Proposition~\ref{prop:HSrepresetnation}) with the functions
\begin{equation*}
    \mathbf{f}(\varphi) := \sinh{\left(U_x(\varphi)\right)} \exp \left( U_{\cos , x}(\varphi) \right) ~~\mbox{and}~~\mathbf{g}(\varphi) := \cos \left( 2\pi(\di^* \varphi , n_q) \right)  (\di^* \mathcal{G}_0(\cdot ,\varphi) , n_q).
\end{equation*}
We obtain (N.B. due to the $\varphi \mapsto -\varphi$ symmetry of the law $\mu_\beta$, the expectation of the function $\mathbf{f}$ is equal to $0$ and the covariance between $\mathbf{f}$ and $\mathbf{g}$ is equal to the expectation of their product)
\begin{align}
    \label{e.2pt2}
    \lefteqn{\langle \sinh{(U_x)} \exp \left( U_{\cos , x} \right)\cos \left( 2\pi(\di^* \varphi , n_q) \right)  (\di^* \mathcal{G}_0(\cdot ,\varphi) , n_q) \rangle_{\mu_\beta}} \qquad & 
    \\ &
    = 
    \sum_{y \in \Zd} \langle  \mathcal{G}_2(y,\varphi) \cdot \partial_{y} \left( (\di^* \mathcal{G}_0(\cdot ,\varphi) , n_q)  \cos \left(2\pi(\di^* \varphi , n_q)\right) \right) \rangle_{\mu_\beta} \notag
    \\
    & =  \underset{\eqref{e.2pt2}-(i)}{\underbrace{\sum_{y \in \Zd} \langle  \mathcal{G}_2(y,\varphi) \cdot  (\di^* \partial_{y} \mathcal{G}_0(\cdot ,\varphi) , n_q)  \cos \left(2\pi(\di^* \varphi , n_q) \right) \rangle_{\mu_\beta}}} \notag \\
    &
    \quad -2 \pi \underset{\eqref{e.2pt2}-(ii)}{\underbrace{\langle  \left( \di^* \mathcal{G}_2(\cdot,\varphi), n_q \right) (\di^* \mathcal{G}_0(\cdot ,\varphi) , n_q)  \sin \left(2\pi(\di^* \varphi , n_q) \right) \rangle_{\mu_\beta}}} \notag 
\end{align}
where $\mathcal{G}_2$ is the solution of the Helffer-Sj{\"o}strand equation, for $(y , \varphi) \in \Zd \times \Omega$,
\begin{align*}
    \mathcal{L}\mathcal{G}_2(y , \varphi) & = \partial_y \left( \sinh{\left(U_x\right)} \exp \left( U_{\cos , x} \right)  \right)
    \\
    & = 2\pi  \cosh{(U_x)}  \exp \left( U_{\cos , x} \right)
    \sum_{q' \in \mathcal{Q}} z(\beta , q')  \cos \left( 2\pi ( \di^* \varphi , n_{q'})\right) \sin \left( 2\pi(\nabla G_{x} , n_{q'})\right)  q'(y) \notag \\
    & \quad - 2\pi  \sinh{(U_x)}  \exp \left( U_{\cos , x} \right)  \sum_{q' \in \mathcal{Q}} z(\beta , q')  \sin \left( 2\pi ( \di^* \varphi , n_{q'})\right) \left( \cos \left( 2\pi(\nabla G_{x} , n_{q'})\right) - 1 \right) q'(y). \notag
\end{align*} 
In order to estimate the right-hand side of~\eqref{e.2pt2}, we collect two upper bounds. The first one is an upper bound on the partial derivative of the function $\mathcal{G}_0$ and reads as follows: there exist two functions $\mathcal{H} : \Zd \times \Zd \times \Omega \to \R^{\binom d2 \times \binom d2 }$ and $\mathcal{W} : \Zd \times \Zd \times \Omega \to \R^{d \times \binom d2}$ such that, for any $y_1 , y_2 \in \Zd$ and $\varphi \in \Omega$,
\begin{equation} \label{eq:identitydstarGHW}
    \partial_{y_2} \mathcal{G}_0(y_1 ,\varphi) = \mathcal{H}(y_1 , y_2 , \varphi) + \di_{2} \mathcal{W} (y_1 , y_2 , \varphi)
\end{equation}
(N.B. for the function $\mathcal{W}$, we consider the exterior derivative in the second variable) together with the estimates
\begin{equation} \label{eq:estverticalderivative}
    \left\| \di^*_{1} \mathcal{H}(y_1 , y_2 , \cdot)  \right\|_{L^2\left(\Omega \, ; \, \R^{d \times \binom d2} \right)} \leq C \frac{\left(\ln (|y_1|_+ + |y_2|_+) \right)^{3d+7} }{|y_1|_+^{d-1} + |y_2|_+^{d-1}}\frac{1}{|y_1 -y_2|_+^{d + 1}}
\end{equation}
and
\begin{equation} \label{eq:estverticalderivativebis}
    \left\| \di^*_{1} \mathcal{W}(y_1 , y_2 , \cdot)  \right\|_{L^2\left(\Omega \, ; \, \R^{d \times d} \right)} \leq  C \frac{(\ln |y_1|_+)^{d + 7/2}}{|y_1|_+^{d-1}} \times \frac{(\ln |y_2|_+)^{d + 7/2}}{|y_2|_+^{d-1}}.
\end{equation}
This inequality is the core of the argument, and its proof can be found in Section~\ref{sec:newsec4}. The second one is an upper bound on the functions $\mathcal{G}_0$ and $\mathcal{G}_2$ and their codifferential: for any $y \in \Zd$,
\begin{equation} \label{eq:boundonG0} 
    \left\| \di^* \mathcal{G}_0 \left( y , \cdot \right) \right\|_{L^2 \left(\Omega \, ; \, \R^{d} \right) } \leq \frac{C(\ln |y|_+)^{d+3}}{|y|^{d - 1}_+} 
\end{equation}
and
\begin{equation} \label{eq:boundonG2} 
     \left\| \mathcal{G}_2 \left( y , \cdot \right) \right\|_{L^2 \left(\Omega \, ; \, \R^{{d \choose 2}} \right) } \leq \frac{C(\ln |y - x|_+)^{d+2}}{|y - x|^{d - 2}_+} \hspace{5mm} \mbox{and} \hspace{5mm} \left\| \di^* \mathcal{G}_2 \left( y , \cdot \right) \right\|_{L^2 \left(\Omega \, ; \, \R^{d} \right)} \leq \frac{C (\ln |y - x|_+)^{d+3}}{|y - x|^{d - 1 }_+}.
\end{equation}
The inequality~\eqref{eq:boundonG0} has already been stated (and proved) in~\eqref{e.gradG0G1} above. In the rest of this section, we prove the inequalities~\eqref{eq:boundonG2} (in Step 1 below) and estimate the two terms~\eqref{e.2pt2}-(i) and~\eqref{e.2pt2}-(ii) by assuming that the four previous inequalities hold (in Steps 2 and 3). As mentioned above, the proof of the identity~\eqref{eq:identitydstarGHW} together with the inequalities~\eqref{eq:estverticalderivative} and~\eqref{eq:estverticalderivativebis} is the most challenging part of the argument and is postponed to Section~\ref{sec:newsec4}.

\medskip

\textit{Step 1. Proof of the inequalities~\eqref{eq:boundonG2}.}

\medskip

To estimate the function $\mathcal{G}_2$, we introduce the two functions
\begin{equation*}
    \mathbf{h}_q(\varphi) = \cosh{(U_x(\varphi))}  \exp \left( U_{\cos , x}(\varphi) \right)
   \cos \left( 2\pi ( \di^* \varphi , n_{q})\right)
\end{equation*}
and 
\begin{equation*}
    \mathbf{k}_q(\varphi)  = \sinh{(U_x(\varphi))}  \exp \left( U_{\cos , x} (\varphi)\right)  \sin \left( 2\pi ( \di^* \varphi , n_{q})\right),
\end{equation*}
so as to have the identities
\begin{align*}
     \mathcal{G}_2(y , \varphi) & = 2\pi
    \sum_{q' \in \mathcal{Q}} z(\beta , q')   \sin \left( 2\pi(\nabla G_{x} , n_{q'})\right)  \left( \di^*_2 \mathcal{G}_{\mathbf{h}_q}(y , \varphi ; \cdot) , n_{q'} \right) \notag \\
    & \quad - 2\pi \sum_{q' \in \mathcal{Q}} z(\beta , q') \left( \cos \left( 2\pi(\nabla G_{x} , n_{q'})\right) - 1 \right) \left( \di^*_2 \mathcal{G}_{\mathbf{k}_q}(y , \varphi ; \cdot) , n_{q'} \right), \notag
\end{align*}
and
\begin{align*}
    \di^* \mathcal{G}_2(y , \varphi) & = 2\pi
    \sum_{q' \in \mathcal{Q}} z(\beta , q')   \sin \left( 2\pi(\nabla G_{x} , n_{q'})\right)  \left( \di^*_1 \di^*_2 \mathcal{G}_{\mathbf{h}_q}(y , \varphi ; \cdot) , n_{q'} \right) \notag \\
    & \quad - 2\pi \sum_{q' \in \mathcal{Q}} z(\beta , q') \left( \cos \left( 2\pi(\nabla G_{x} , n_{q'})\right) - 1 \right) \left( \di^*_1 \di^*_2 \mathcal{G}_{\mathbf{k}_q}(y , \varphi ; \cdot) , n_{q'} \right). \notag
\end{align*}
Using the regularity estimate on the Green's matrix associated with the Helffer-Sj\"{o}strand operator (Proposition~\ref{prop.prop4.11chap4}) together with Lemma~\ref{lemma.lemma2.5}, we deduce that
\begin{align*}
     \left\| \mathcal{G}_2(y , \cdot) \right\|_{L^2\left(\Omega \, ; \, \R^{d \choose 2}\right)} & \leq \sum_{q' \in \mathcal{Q}}  |z(\beta , q')| \left|(\nabla G_x , n_{q'}) \right| \left\| \left( \di^*_2 \mathcal{G}_{\mathbf{h}_q}(y , \cdot ; \cdot), n_{q'} \right) \right\|_{L^2\left(\Omega \, ; \, \R^{d \choose 2}\right)}  
     \\& \quad + \sum_{q' \in \mathcal{Q}}  |z(\beta , q')| \left|(\nabla G_x , n_{q'}) \right|^2 \left\| \left( \di^*_2 \mathcal{G}_{\mathbf{k}_q}(y , \cdot ; \cdot), n_{q'} \right) \right\|_{L^2\left(\Omega \, ; \, \R^{d \choose 2}\right)}  \\
     & \leq  \sum_{z \in \Zd} \frac{C (\ln |y - z|_+)^{d+2}}{|x-z|^{d-1}_+ |y-z|^{d-1}_+} +  \sum_{z \in \Zd} \frac{C  (\ln |y - z|_+)^{d+2}}{|x-z|^{2d-2}_+ |y-z|^{d-1}_+} \\
     & \leq \frac{C  (\ln |x - y|_+)^{d+2}}{|x-y|^{d-2}_+} + \frac{C  (\ln |x - y|_+)^{d+2}}{|x-y|^{d-1}_+} \\
     & \leq \frac{C  (\ln |x - y|_+)^{d+2}}{|x-y|^{d - 2}_+}.
\end{align*}
Similarly for the codifferential of the function $\mathcal{G}_2$
\begin{align*}
     \left\| \di^* \mathcal{G}_2(y , \cdot) \right\|_{L^2(\Omega \, ; \, \Rd)} & \leq \sum_{q' \in \mathcal{Q}}  |z(\beta , q')| \left|(\nabla G_x , n_{q'}) \right| \left\| \left( \di^*_1 \di^*_2 \mathcal{G}_{\mathbf{h}_q}(y , \cdot ; \cdot), n_{q'} \right) \right\|_{L^2(\Omega \, ; \, \Rd)}  
     \\& \quad + \sum_{q' \in \mathcal{Q}}  |z(\beta , q')| \left|(\nabla G_x , n_{q'}) \right|^2 \left\| \left( \di^*_1 \di^*_2 \mathcal{G}_{\mathbf{k}_q}(y , \cdot ; \cdot), n_{q'} \right) \right\|_{L^2(\Omega \, ; \, \Rd)}  \\
     & \leq C \sum_{z \in \Zd} \frac{C (\ln |y - z|_+)^{d+2}}{|x-z|^{d-1}_+ |y-z|^{d}_+} +  \sum_{z \in \Zd} \frac{C (\ln |y - z|_+)^{d+2}}{|x-z|^{2d-2}_+ |y-z|^{d}_+} \\
     & \leq \frac{C (\ln |x - y|_+)^{d+3}}{|x-y|^{d-1}_+} + \frac{C (\ln |x- y|_+)^{d+2}}{|x-y|^{d}_+} \\
     & \leq \frac{C(\ln |x - y|_+)^{d+3}}{|x-y|^{d - 1}_+}.
\end{align*}
The proof of the estimates~\eqref{eq:boundonG2} is complete.

\medskip

\textit{Step 2. Estimating the term~\eqref{e.2pt2}-(i).}

\medskip

We first estimate the term~\eqref{e.2pt2}-(i) by using the identity~\eqref{eq:identitydstarGHW} together with an integration by parts in the second variable (i.e., we use~\eqref{identity:dandd*}) to write (N.B. since the operators $\di^*_1$ and $\di_2$ act on different variables, it can be proved that they commute)
\begin{align*}
     \eqref{e.2pt2}-(i) & = \sum_{y \in \Zd} \langle \mathcal{G}_2(y,\varphi) \cdot  ( \di^*_{1} \mathcal{H}( \cdot , y ,\varphi) , n_q)  \sin \left(2\pi(\di^* \varphi , n_q) \right) \rangle_{\mu_\beta} 
    \\ & \quad + \sum_{y \in \Zd} \langle \mathcal{G}_2(y,\varphi) \cdot  \left( \di_{2} \di^*_{1} \mathcal{W}( \cdot , y ,\varphi), n_q \right)  \sin \left(2\pi(\di^* \varphi , n_q) \right) \rangle_{\mu_\beta} \\
    &  = \sum_{y \in \Zd} \langle \mathcal{G}_2(y,\varphi) \cdot  (\di^*_{1} \mathcal{H}(\cdot , y ,\varphi) , n_q)  \sin \left(2\pi(\di^* \varphi , n_q) \right) \rangle_{\mu_\beta} 
    \\ & \quad + \sum_{y \in \Zd} \langle \di^* \mathcal{G}_2(y,\varphi) \cdot  \left( \di^*_{1} \mathcal{W}( \cdot , y ,\varphi), n_q \right)  \sin \left(2\pi(\di^* \varphi , n_q) \right) \rangle_{\mu_\beta}.
\end{align*}
From the inequalities~\eqref{eq:estverticalderivative},~\eqref{eq:estverticalderivativebis},~\eqref{eq:boundonG2}, we see that there exists a constant $C_q := C_q(d , \beta , q) < \infty$ which grows polynomially fast in $\left\| q\right\|_1$ such that
\begin{align*}
    \left| \eqref{e.2pt2}-(i) \right| & \leq \sum_{y \in \Zd} C_q \frac{\left( \ln |y - x|_+ \right)^{d+2}}{|y - x|^{d-2}_+} \frac{\left(\ln (|y|_+ + |z_q|_+)\right)^{3d+7}}{|y|^{d-1}_+ + |z_q|^{d-1}_+} \frac{1}{|y -z_q|^{d+1}_+} \\
    & \quad + \sum_{y \in \Zd} C_q \frac{\left( \ln |y - x|_+ \right)^{d+3}}{|y - x|^{d - 1}_+} \frac{\left( \ln |z_q|_+ \right)^{d+7/2}}{|z_q|^{d-1}_+} \frac{\left( \ln |y|_+ \right)^{d+7/2}}{|y|^{d-1}_+ }.
\end{align*}
We simplify the previous inequality by using the suboptimal (but sufficient for our purposes) inequality 
$$\frac{\left(\ln (|y|_+ + |z_q|_+)\right)^{3d+7}}{|y|_+^{d-1} + |z_q|_+^{d-1}} \leq C \frac{\left(\ln |z_q|_+\right)^{3d+7}}{|z_q|_+^{d-1}}.$$
 We obtain this way
\begin{align*}
     \left| \eqref{e.2pt2}-(i) \right|  & \leq C_q \sum_{y \in \Zd}  \frac{\left( \ln |y - x|_+ \right)^{d+2}}{|y - x|^{d-2}_+} \frac{\left(\ln  |z_q|_+\right)^{3d+7}}{ |z_q|^{d-1}_+} \frac{1}{|y -z_q|^{d+1}_+} \\
     & \qquad + C_q \sum_{y \in \Zd} \frac{\left( \ln |y - x|_+ \right)^{d+3}}{|y - x|^{d - 1}_+} \frac{\left( \ln |z_q|_+ \right)^{d+7/2}}{|z_q|^{d-1}_+} \frac{\left( \ln |y|_+ \right)^{d+7/2}}{|y|^{d-1}_+ }.
\end{align*}
We next estimate the two terms on the right-hand side. For the first one, we have
\begin{align*}
\sum_{y \in \Zd}\frac{\left( \ln |y - x|_+ \right)^{d+2}}{|y - x|^{d-2}_+} \frac{\left(\ln  |z_q|_+\right)^{3d+7}}{ |z_q|^{d-1}_+}  \frac{1}{|y -z_q|^{d+1}_+}  & = \frac{\left(\ln  |z_q|_+\right)^{3d+7}}{ |z_q|^{d-1}_+}  \sum_{y \in \Zd} \frac{\left( \ln |y - x|_+ \right)^{d+2}}{|y - x|^{d-2}_+} \frac{1}{|y -z_q|^{d+1}_+}  \\
& \leq C \frac{\left(\ln  |z_q|_+\right)^{3d+7}}{ |z_q|^{d-1}_+} \frac{\left( \ln |x- z_q|_+ \right)^{d+2}}{|x - z_q|^{d-2}_+}.
\end{align*}
For the second one, we have
\begin{align*}
     \sum_{y \in \Zd} \frac{\left( \ln |y - x|_+ \right)^{d+3}}{|y - x|^{d - 1}_+} \frac{\left( \ln |z_q|_+ \right)^{d+7/2}}{|z_q|^{d-1}_+} \frac{\left( \ln |y|_+ \right)^{d+7/2}}{|y|^{d-1}_+ } & = \frac{\left( \ln |z_q|_+ \right)^{d+7/2}}{|z_q|^{d-1}_+}  \sum_{y \in \Zd}  \frac{\left( \ln |y - x|_+ \right)^{d+3}}{|y - x|^{d - 1}_+}  \frac{\left( \ln |y|_+ \right)^{d+7/2}}{|y|^{d-1}_+ } 
     \\ & \leq  C \frac{\left( \ln |z_q|_+ \right)^{d+7/2}}{|z_q|^{d-1}_+} \frac{\left( \ln |x|_+ \right)^{2d+13/2}}{|x|^{d-2}_+} .
\end{align*}
A combination of the three previous displays imply the upper bound
\begin{align} \label{eq:estimate57i}
    \left| \eqref{e.2pt2}-(i) \right| & \leq  C_q \frac{\left(\ln  |z_q|_+\right)^{3d+7}}{ |z_q|^{d-1}_+} \frac{\left( \ln |x - z_q|_+ \right)^{d+2}}{|x - z_q|^{d-2}_+} +  C_q \frac{\left( \ln |z_q|_+ \right)^{d+7/2}}{|z_q|^{d-1}_+} \frac{\left( \ln |x|_+ \right)^{2d+13/2}}{|x|^{d-2}_+}.
\end{align}

\medskip

\textit{Step 3. Estimating the term~\eqref{e.2pt2}-(ii).}

\medskip

This term is simpler to estimate. Using the inequalities~\eqref{eq:boundonG0} and \eqref{eq:boundonG2}, we have
\begin{equation} \label{eq:estimate57ii}
    \left| \eqref{e.2pt2}-(ii) \right| = \left| \langle  \left( \di^* \mathcal{G}_2(\cdot,\varphi), n_q \right) (\di^* \mathcal{G}_0(\cdot ,\varphi) , n_q)  \sin \left(2\pi(\di^* \varphi , n_q) \right) \rangle_{\mu_\beta} \right| \leq C_q \frac{\left( \ln |x - z_q| \right)^{d+3}}{|x - z_q|_+^{d-1}} \frac{\left( \ln |z_q| \right)^{d+3}}{ |z_q|_+^{d-1}}.
\end{equation}
A combination of~\eqref{eq:estimate57i} and~\eqref{eq:estimate57ii} completes the proof of the inequality~\eqref{ineq:maintechnical}.

\section{Bounding the vertical derivative using the second-order Helffer-Sj\"{o}strand equation} \label{sec:newsec4}

This section is devoted to the proof of the identity~\eqref{eq:identitydstarGHW} and the inequalities~\eqref{eq:estverticalderivative} and~\eqref{eq:estverticalderivativebis}. Let us first recall the statement, we want to show that 
the following identity
\begin{equation*} 
    \partial_{y_2} \mathcal{G}_0(y_1 ,\varphi) = \mathcal{H}(y_1 , y_2 , \varphi) + \di_{2} \mathcal{W} (y_1 , y_2 , \varphi)
\end{equation*}
together with the inequalities
\begin{equation*}
    \left\| \di^*_{1} \mathcal{H}(y_1 , y_2 , \cdot)  \right\|_{L^2\left(\Omega \, ; \, \R^{d \times \binom d2} \right)} \leq  \frac{\left(\ln (|y_1|_+ + |y_2|_+) \right)^{3d+7} }{|y_1|_+^{d-1} + |y_2|_+^{d-1}}\frac{1}{|y_1 -y_2|_+^{d + 1}}
\end{equation*}
and
\begin{equation*} 
    \left\| \di^*_{1} \mathcal{W}(y_1 , y_2 , \cdot)  \right\|_{L^2\left(\Omega \, ; \, \R^{d \times d} \right)} \leq    C \frac{(\ln |y_1|_+)^{d + 7/2}}{|y_1|_+^{d-1}} \times \frac{(\ln |y_2|_+)^{d + 7/2}}{|y_2|_+^{d-1}}.
\end{equation*}
In order to establish this result, we first present two additional tools pertaining to the Helffer-Sj\"{o}strand equation: a dynamical interpretation of the Helffer-Sj\"{o}strand equation, and the second-order Helffer-Sj\"{o}strand equation. Once equipped with these two tools, the strategy of the proof is straightforward but fairly technical to implement: the idea is to note that the map $(y_1 , y_2 , \varphi) \mapsto \partial_{y_2} \mathcal{G}_0(y_1 ,\varphi)$ solves a second-order Helffer-Sj\"{o}strand equation, and to analyse this function using the results of Section~\ref{sec:additionaltools}.

\subsection{Two additional tools} \label{sec:additionaltools}

\subsubsection{Dynamical interpretation of the Helffer-Sj\"{o}strand equation} \label{sec:sec264dynamic} 
Beside the definition of the Helffer-Sj\"{o}strand operator, there is a second (more dynamical) approach to the Helffer-Sj\"ostrand equation which originally appeared in~\cite{NS} and is one of the key tools of~\cite{GOS} to identify the scaling limit of the model. We provide in this section a brief account of the approach by stating the main definitions and result, and refer to~\cite{GOS} for the proofs. We note that the results of this section are used in Section~\ref{sec:newsec4} (and specifically in Section~\ref{subsecB23}).

We first introduce the Langevin dynamic associated with the measure $\mu_\beta$. 

\begin{definition}[Langevin dynamic]
    Let $\left\{ B_t(x) \, : \, t \geq 0, \, x \in \Zd \right\}$ be a collection of independent Brownian motions valued in $\R^{{d \choose 2}}$ and let $\varphi \in \Omega$ be an initial condition. We define the Langevin dynamic to be the solution of the system of stochastic differential equations
 \begin{equation*}
 \left\{ \begin{aligned}
 d\varphi_t(x) 
& = -\sum_{n\geq 0} \frac{1}{2\beta} \frac{1}{\beta^{n/2}}(-\Delta)^{n+1} \varphi_t(x) dt -\sum_{q \in \mathcal{Q}} 2\pi z(\beta, q) \sin{ \left( 2\pi \left( q, \varphi_t \right) \right)} q(x) \,dt + \sqrt{2} \,dB_t(x) & \mbox{in} ~ (0 , \infty) \times \Zd, \\
\varphi_0(x) & = \varphi(x) & \mbox{in} ~ \Zd.
\end{aligned} \right.
\end{equation*}
Given an initial condition $\varphi \in \Omega$, we denote by $\mathbb{P}_{\varphi}$ the law of the Langevin dynamic started from $\varphi \in \Omega$ and by $\E_\varphi$ the corresponding expectation.
\end{definition}

\begin{remark}
    Let us make two remarks about this definition:
    \begin{itemize}
    \item The existence and uniqueness of a solution to the system of stochastic differential equations when the initial condition $\varphi \in \Omega$ satisfies the assumption $\sum_{x \in \Zd} \left| \varphi(x) \right|^2 e^{- |x|} < \infty$ (which holds $\mu_\beta$ almost surely) are guaranteed by standard arguments (see, e.g.,~\cite[Section 2.1.3]{GOS} and~\cite[Section 2.2]{FS}).
    \item The Langevin dynamic is stationary and reversible with respect to the measure $\mu_\beta$.
    \item Given an exponent $p \in [1, \infty)$ and a function $F$ which depends on the Langevin dynamic, i.e., $F : \Zd \times C([0, \infty], \R) \to \R^k$ where $C([0, \infty], \R)$ is the set of continuous functions defined on $[0, \infty]$ and valued in $\R$ (N.B. this will typically be used with the heat kernel and its codifferential defined below), we denote by
    \begin{equation*}
        \left\| F \right\|_{L^p(\Omega \, ; \, \R^k)} := \left| \left\langle \E_\varphi \left[ \left| F \right|^p \right] \right\rangle_{\mu_\beta} \right|^{1/p}
    \end{equation*}
    the $L^p$-norm with respect to the Langevin dynamic when the initial condition is sampled according to the Gibbs measure $\mu_\beta$.
    \end{itemize}
\end{remark}

The Langevin dynamic can be used to solve the Helffer-Sj\"{o}strand equation. Specifically, we have the following proposition for which we refer to~\cite{NS, GOS}.

\begin{proposition}[Dynamical formula for the Helffer-Sj\"{o}strand Green's matrix] \label{prop:prop2.26}
   For any $\mathbf{f} \in L^2(\Omega \, ; \, \R )$, we have the identity, for $(x , y , \varphi) \in \Zd \times \Zd \times \Omega$,
    \begin{equation*} 
        \mathcal{G}_{\mathbf{f}}(x , \varphi ; y) := \E_{\varphi} \left[ \int_0^\infty \mathbf{f}(\varphi_t) ( P^{\varphi_\cdot} (t , y ; x))^T \, dt \right],
    \end{equation*}
    where, given a realisation of the Langevin dynamic $(\varphi_t)_{t \geq 0}$, $P^{\varphi_\cdot} : (0 , \infty) \times \Zd \times \Zd \to \R^{\binom d2 \times \binom d2}$ is the solution of the parabolic system of equations
    \begin{equation} \label{eq:eqdefheqtkernel}
        \left\{ \begin{aligned}
        \partial_t P^{\varphi_\cdot} (\cdot , \cdot ; y) + \mathcal{L}_{\mathrm{spat}}^{\varphi_\cdot}  P^{\varphi_\cdot}(\cdot , \cdot ; y) & =0 ~\mbox{in}~ (0 , \infty) \times \Zd , \\
        P^{\varphi_\cdot} \left(0,\cdot ; y \right) & = \delta_{y} ~\mbox{in}~ \Zd,
        \end{aligned}
        \right.
    \end{equation}
    and $( P^{\varphi_\cdot} (t , y ; x))^T$ is the transposed of the matrix $ P^{\varphi_\cdot} (t , y ; x)$.
\end{proposition}

\begin{remark} \label{rem:reversinghtetime}
For later purposes, let us note that:
\begin{itemize}
    \item If, given a realisation of the Langevin dynamic $(\varphi_t)_{t \geq 0}$ and a time $t \in (0 , \infty)$, we introduce the time reversed dynamic
$$(\bar \varphi^t_s)_{0 \leq s \leq t} : = (\varphi_{t-s})_{0 \leq s \leq t}$$
then we have the identity, for any $x , y \in \Zd$
\begin{equation*}
    P^{\varphi_\cdot} (t , y ; x) = (P^{\bar \varphi_\cdot^t} (t , x ; y) )^T
\end{equation*}
where the right-hand side is the transposed matrix (N.B. Let us remark that the heat kernel at time $t$ on the right-hand side can be rigorously defined even though the dynamic $\bar \varphi^t$ has only been defined for the times $s \in [0,t]$).
    \item Using a computation similar to the one of Lemma~\ref{lemma.lemma2.5}, one can show that
    \begin{equation} \label{eq:PisLipschitz}
        \left| \partial_t P^{\varphi_\cdot}(t , x , y)  \right| = \left| \mathcal{L}_{\mathrm{spat}}^{\varphi_\cdot}  P^{\varphi_\cdot}(t , x , y)  \right| \leq C \sum_{z \in \Zd} e^{-c (\ln \beta) |z - x|} \left| \nabla P^{\varphi_\cdot}(t , z , y) \right|.
    \end{equation}
    (N.B. There are two contributions in this upper bound: a term coming from the iteration of the Laplacian yielding a decay of the form $e^{- c (\ln \beta) |x - y|}$ and another one coming from the charges which would yield a faster decay of the form $e^{- c \sqrt{\beta} |x - y|}$).
\end{itemize}
\end{remark}

The following statement quantifies the decay of the heat kernel and its codifferential. The first inequality is the celebrated Nash-Aronson estimate and the second and third ones are regularity estimates. We recall the notation $t_+ := (t+1)$, refer to~\cite[Proposition 3.10]{DW} for a proof of the inequality~\eqref{eq:estimateheat} and to Appendix~\ref{app.CZreg} for a proof of~\eqref{eq:decaygradheat} and~\eqref{eq:decaygradgradheat}.

\begin{proposition}[Upper bound on the heat kernel and its codifferential] \label{prop.prop4.11chap4HK}
For any exponent $p \in (1 , \infty)$, there exists an inverse temperature $\beta_1 (d , p) < \infty$ such that for any $\beta > \beta_1$ the following result holds. There exists a constant $C(d, \beta, p ) < \infty$ such that, for any realisation of the Langevin dynamic $(\varphi_t)_{t \geq 0}$ and any $x,y \in \Zd$,
\begin{equation} \label{eq:estimateheat}
\left| P^{\varphi_\cdot} (t , x ; y) \right| \leq \frac{C}{t_+^{d/2}} \exp \left( - \frac{|x - y|}{C \sqrt{t_+}}\right).
\end{equation}
Additionally, for any $x , y \in \Zd$,
\begin{equation} \label{eq:decaygradheat}
 \left\| \di^*_1 P^{\cdot} (t , x ; y) \right\|_{L^p \left(\Omega \, ; \, \R^{d \times {d \choose 2}} \right)} + \left\| \di^*_2 P^{\cdot} (t , x ; y) \right\|_{L^p\left(\Omega \, ; \, \R^{{d \choose 2} \times d}\right)} \leq \frac{C \left( \ln t_+ \right)^{d+2}}{t_+^{d/2 + 1/2}} \exp \left( - \frac{|x - y|}{C \sqrt{t_+}}\right)
\end{equation}
and
\begin{equation}  \label{eq:decaygradgradheat}
    \left\| \di^*_1 \di^*_2 P^{\cdot} (t , x ; y) \right\|_{L^p(\Omega \, ; \, \R^{d \times d })} \leq \frac{C \left( \ln t_+ \right)^{d+2} }{t_+^{d/2 + 1}} \exp \left( - \frac{|x - y|}{C \sqrt{t_+}}\right).
\end{equation}
\end{proposition}

\begin{remark}
    Proposition~\ref{prop.prop4.11chap4} is a direct consequence of Proposition~\ref{prop:prop2.26} and Proposition~\ref{prop.prop4.11chap4HK}.
\end{remark}

\subsubsection{The second order Helffer-Sj{\"o}strand  equation} \label{sec:secsecondorderHS}

The final ingredient which we will need in this article is the second-order Helffer-Sj\"{o}strand equation (following Conlon  and Spencer~\cite{CS14}). The motivation underlying the introduction of this second equation is as follows: in the proof of Theorem ~\ref{th:mainth}, it will be needed to estimate a covariance of the form 
\begin{equation} \label{eq:covdoubleHS}
    \cov_{\mu_\beta} \left[ \mathcal{U}_i(x , \cdot) , \mathbf{f} \right],
\end{equation}
where:
\begin{itemize}
    \item $x$ is a vertex of $\Zd$ and $i \in \{1 , \ldots, {d \choose 2}\}$ is an index;
    \item $\mathbf{f}: \Omega \to \R$ is an explicit function;
    \item $\mathcal{U} : \Zd \times \Omega \to \R^{d \choose 2}$ is the solution of the Hellfer-Sj\"{o}strand equation $\mathcal{L}\mathcal{U} = G$ in $\Zd \times \Omega$ for an explicit function $G : \Zd \times \Omega \to \R^{d \choose 2}$.
\end{itemize}
To estimate the covariance~\eqref{eq:covdoubleHS}, we will use the Helffer-Sj\"{o}strand representation formula (Proposition~\ref{prop:HSrepresetnation}), which yields the identity
\begin{equation*}
    \cov_{\mu_\beta} \left[ \mathcal{U}_i(x , \cdot) , \mathbf{f} \right] = \sum_{y \in \Zd} \left\langle \partial_y \mathcal{U}_i(x , \cdot) \cdot \mathcal{H}(x , \cdot) \right\rangle_{\mu_\beta} ~~\mbox{with}~~ \mathcal{L} \mathcal{H} = \partial \mathbf{f} ~\mbox{in}~ \Zd \times \Omega.
\end{equation*}
(N.B. On the right-hand side, both functions are valued in $\R^{\binom d2}$ and the dot represents the Euclidean scalar product on this space). We will thus have to study the properties of the function
$$\partial \mathcal{U} : (x , y , \varphi) \to \partial_y \mathcal{U}(x , \varphi) \in \R^{\binom d2 \times \binom d2}.$$
The strategy is then to find an equation satisfied by the map $\partial \mathcal{U}$. In this direction, we may apply (formally) the operator $\partial_{x}$ to both the left- and right-hand sides of the identity $\mathcal{L} \mathcal{U} = G$. We obtain the identity
\begin{equation} \label{eq:derivHS1}
 \partial_y \mathcal{L} \mathcal{U} (x , \varphi) = \partial_y G (x , \varphi)  ~\mbox{for}~ (x , y , \varphi) \in \Zd \times \Zd \times \Omega 
\end{equation}
(N.B. Both sides of this identity are valued in $\R^{{d \choose 2} \times {d \choose 2}}$).

Following the insight of Conlon and Spencer~\cite{CS14}, one may deduce from~\eqref{eq:derivHS1} that the map $\partial \mathcal{U}$ is the solution of a second-order Helffer-Sj\"{o}strand equation defined as follows
\begin{equation*}
    \mathcal{L}_{\mathrm{sec}} \partial \mathcal{U} (x , y , \varphi) = \partial_y G(x , \varphi) +  8\pi^3 z\left( \beta , q \right) \sin \left(2\pi\left( \varphi , q \right)\right) \left( \di^* \mathcal{U}(\cdot , \varphi) , n_q \right) \left( q \otimes q\right) (x, y) ~~\mbox{for}~~ (x , y , \varphi) \in \Zd \times \Zd \times \Omega
\end{equation*}
where:
\begin{itemize}
    \item $\mathcal{L}_{\mathrm{sec}}$ is an operator, called the second-order Helffer-Sj\"{o}strand operator, which acts on functions defined on $\Zd \times \Zd \times \Omega$ and valued in $\R^{{d \choose 2} \times {d \choose 2}}$. This operator takes the following form
    \begin{equation*}
\mathcal{L}_{\mathrm{sec}} := -\Delta_{\varphi} + \mathcal{L}_{\mathrm{spat}, x }^{\varphi} + \mathcal{L}_{\mathrm{spat}, y}^{\varphi}
\end{equation*}
    where $\Delta_{\varphi}$ is the Laplacian defined in~\eqref{eq:TV13000102} and $\mathcal{L}_{\mathrm{spat}, x }$ and $\mathcal{L}_{\mathrm{spat}, y}$ are
    extensions of the operator $\mathcal{L}_{\mathrm{spat}}^{\varphi}$ to functions $\mathcal{H} : \Zd \times \Zd \times \Omega \to \R^{{d \choose 2} \times {d \choose 2}}$. Specifically given such a function, we define the function $\mathcal{L}_{\mathrm{spat}, x }^{\varphi} \mathcal{H} : \Zd \times \Zd \times \Omega \to \R^{{d \choose 2} \times {d \choose 2}}$ as follows:
        \begin{itemize}
            \item  We first set a convention: we will use two indices $i , j \in \left\{ 1 , \ldots, {d \choose 2}\right\}$ to refer to the components of a vector of $\R^{{d \choose 2} \times {d \choose 2}}$.
            \item We next fix a vertex $y \in \Zd$, an index $j \in \left\{ 1 , \ldots, {d \choose 2}\right\}$ and consider the function $\mathcal{H}_{\cdot j} (\cdot , y , \cdot) : \Zd \times \Omega \to \R^{{d \choose 2}}$. We will denote this function by $\mathcal{H}_{\cdot j}^y$ below.
            \item We may then apply the operator $\mathcal{L}_{\mathrm{spat}}^{\varphi}$ to the function $\mathcal{H}_{\cdot j}^y$ to obtain the function $\mathcal{L}_{\mathrm{spat}, x }^{\varphi} \mathcal{H}$. More precisely, for $i , j \in \left\{ 1 , \ldots, {d \choose 2}\right\}$, we define
            \begin{equation*}
                \left( \mathcal{L}_{\mathrm{spat}, x }^{\varphi} \mathcal{H} (x , y , \varphi)\right)_{ij} :=  \left( \mathcal{L}_{\mathrm{spat}}^{\varphi} \mathcal{H}_{\cdot j}^y(x , \varphi)\right)_{i}.
            \end{equation*}
        \end{itemize}
        For the operator $\mathcal{L}_{\mathrm{spat}, x }^{\varphi}$ we proceed similarly but freeze the first index $i \in \{ 1 , \ldots , {d \choose 2}\}$ and the variable~$x$.
    \item We note that the operators $\mathcal{L}_{\mathrm{spat}, x }^{\varphi}$ and $\mathcal{L}_{\mathrm{spat}, y}^{\varphi}$ commute (because they act on distinct variables and distinct coordinates). This property will play an important role in Proposition~\ref{prop:dynamicalsecondorderHS} below.
    \item The two terms in the sum on the right-hand side are both valued in $\R^{{d \choose 2} \times {d \choose 2}}$: the partial derivative $\partial G$ is defined in Section~\ref{sec:defWitten} and the tensor product $(q \otimes q)$ is defined in~\eqref{def:tensorproduct}.
\end{itemize}

The existence and uniqueness of solutions of the second-order Helffer-Sj\"{o}strand equation can be obtained as a consequence of the Lax-Milgram Theorem, using the same techniques as the ones (briefly) discussed in Section~\ref{sec:wittenetHS}.

As it was the case for the Helffer-Sj\"{o}strand operator, we will be interested in the fundamental solution (or Green's matrix) associated with the second-order Helffer-Sj\"{o}strand operator. It is introduced in the following definition.

\begin{definition}[Green's matrix for the second-order Helffer-Sj\"{o}strand equation]
For any $( x_1, y_1) \in \Zd \times \Zd$, we let $\delta_{(x_1 , y_1)} : \Zd \times \Zd \to \R^{\binom d2^2 \times \binom d2^2}$ be the discrete Dirac mass defined by the formula
$$\delta_{(x_1,y_1)}((x,y)) := \left( \indc_{\{(x, y) = (x_1,y_1)\}} \indc_{\{(i,j)=(k,l)\}} \right)_{1 \leq i, j , k , l \leq \binom d2} \in \R^{{d \choose 2}^2 \times {d \choose 2}^2}.$$
For any function $\f : \Omega \to \R$ satisfying $\f \in L^2 \left( \Omega \, ; \, \R \right)$ and any $(x_1 , y_1) \in \Zd \times \Zd$, we let 
$$\G_{\mathrm{sec},\f}(\cdot ; x_1 , y_1 ):= \left( \G_{\mathrm{sec}, \f, ijkl}(\cdot ; x_1 , y_1) \right)_{1\leq i,j, k , l\leq {\binom d2}} : \Zd \times \Zd \times \Omega \to \R^{\binom d2^2 \times \binom d2^2}$$
be the solution of the second-order Helffer-Sj\"{o}strand equation
\begin{equation}
\label{e.Greender}
\mathcal{L}_{\mathrm{sec}} \G_{\mathrm{sec},\f}(\cdot , \cdot,  \cdot; x_1 , y_1) =\f \delta_{(x_1,y_1)} ~~\mbox{in}~~ \Omega \times \Zd \times \Zd.
\end{equation}
\end{definition}

\begin{remark}
    Let us make a few remarks on the previous definition:
    \begin{itemize}
    \item Since the second-order Helffer-Sj\"{o}strand operator acts on functions valued in $\R^{{d \choose 2}^2}$, the Green's matrix is valued in $\R^{\binom d2^2 \times \binom d2^2}$, i.e., it is a matrix of size ${d \choose 2}^2 \times {d \choose 2}^2$. We identify the space $\R^{{d \choose 2}^2}$ with the space of matrices of size ${d \choose 2} \times {d \choose 2}$ and will use two indices $i , j \in \{ 1 , \ldots, {d \choose 2}\}$ to refer to the components of a vector of $\R^{{d \choose 2}^2}$ (and thus four indices $i , j , k , l\in \{ 1 , \ldots, {d \choose 2}\}$ to refer to the components of a vector of $\R^{{d \choose 2}^2 \times {d \choose 2}^2}$).
    \item The identity~\eqref{e.Greender} is understood as follows: for any $k , l \in \{ 1 , \ldots, {d \choose 2} \}$, the function $\G_{\mathrm{sec} , \f, \cdot kl}(\cdot ; x_1 , y_1):= \left( \G_{\mathrm{sec},\f, ijkl}(\cdot ; x_1 , y_1) \right)_{1\leq i, j \leq \binom d2} : \Zd \times \Zd \times \Omega \to \R^{{\binom d2}^2}$ solves
    \begin{equation*}
    \mathcal{L}_{\mathrm{sec}} \mathcal{G}_{\mathrm{sec} , \f, \cdot kl}(x , y, \varphi; x_1 , y_1) = \begin{pmatrix}
           0 & \ldots & 0 & \ldots & 0 \\
           \vdots & & \vdots &  &\vdots \\
           0 & \ldots &   \textbf{f} (\varphi)   \indc_{\{(x,y) = (x_1, y_1)\}} & \ldots & 0 \\
           \vdots && \vdots && \vdots\\
           0 & \ldots & 0 & \ldots &0
         \end{pmatrix} \hspace{3mm} \mbox{for}~ (x , y , \varphi) \in \Zd \times \Zd \times \Omega,
    \end{equation*}
    where the right-hand side is a ${d \choose 2} \times {d \choose 2}$ matrix whose only non-zero entry is on the $k$th-row and $l$-th column.
    \item We define the codifferential in the first variable, and denote it by $\di^*_1 \mathcal{G}_{\mathrm{sec},\mathbf{f}}$, as follows: for each fixed triplet $j , k , l \in \{1 , \ldots, \binom d2 \}$, vertices $y , x_1 , y_1 \in \Zd$ and $\varphi \in \Omega$, we consider the function $x \mapsto \mathcal{G}_{\mathrm{sec}, \mathbf{f}, \cdot jkl}(x , y , \varphi ; x_1 , y_1) \in \R^{d \choose 2}$. This function can be seen as a $2$-form and one can apply the codifferential $\di^*$ as in Definition~\ref{def:codifferential2form}. We thus obtain a $1$-form $x \mapsto \di^*_1 \mathcal{G}_{\mathrm{sec}, \mathbf{f}, \cdot jkl}(x , y, \varphi ; x_1 ,  y_1) \in \R^{d}$. We then define the function 
    $$\di^*_1 \mathcal{G}_{\mathrm{sec},\mathbf{f}} : (x , y, \varphi ; x_1 , y_1) \mapsto \left( \di^*_1 \mathcal{G}_{\mathrm{sec},\mathbf{f}, i jkl}(x , y , \varphi ; x_1 , y_1) \right)_{1 \leq i \leq d, 1 \leq j , k , l \leq \binom d2 } \in \R^{d \times {\binom d2}^3}.$$
    We may similarly define the codifferential in the second variable, denoted by $\di^*_2 \mathcal{G}_{\mathrm{sec},\mathbf{f}}$, by freezing the first, third and fourth variables and the components $i , k , l \in \{1 , \ldots, {d \choose 2} \}$. 
    We may similarly define the codifferential in the  third variable and in the fourth variable as well as the mixed codifferentials by applying this procedure to more than one variable.
    
    More specifically, we will have to use the mixed codifferential with respect to the first, third and fourth variables
    \begin{equation*}
        \begin{aligned}
        \di^*_1  \di^*_{3}  \di^*_{4} \G_{\mathrm{sec}, \f} := \left( \di^*_1  \di^*_{3}  \di^*_{4}  \mathcal{G}_{\mathrm{sec} , \mathbf{f}, i j k l} \right)_{1 \leq i , k , l \leq d, 1 \leq j \leq   \binom d2} : \Zd \times \Zd \times \Omega \times \Zd \times \Zd \to \R^{ d \times {d \choose 2} \times d \times d}.
        \end{aligned}
    \end{equation*}
    \item We will use the previous definitions in conjunction with the following definition, for each pair $(x , y , \varphi) \in \Zd \times \Zd \times \Omega$ and each pair of charges $q , q' \in \mathcal{Q}$,
    \begin{equation*}
        \begin{aligned}
        \left( \mathcal{G}_{\mathrm{sec}, \mathbf{f}} (x, y, \varphi ; \cdot, \cdot), q \otimes q' \right) & = \sum_{x_1 , y_1 \in \Zd}  \sum_{k , l = 1}^{d \choose 2}\mathcal{G}_{\mathrm{sec}, \mathbf{f}, \cdot kl}(x, y, \varphi ; x_1, y_2) q_{k} (x_1) q'_l(y_1) \in \R^{{d \choose 2} \times {d \choose 2}}, \\
        \left( \di^*_{3} \di^*_{4} \mathcal{G}_{\mathrm{sec}, \mathbf{f}} (x, y, \varphi ; \cdot, \cdot), n_q \otimes n_{q'} \right) & = \sum_{x_1 , y_1 \in \Zd}  \sum_{k , l = 1}^{d}\di^*_{3} \di^*_{4} \mathcal{G}_{\mathrm{sec}, \mathbf{f}, \cdot kl}(x, y, \varphi ; x_1, y_2) n_{q,k} (x_1) n_{q',l}(y_1) \in \R^{{d \choose 2} \times {d \choose 2}},
        \end{aligned}
    \end{equation*}
    (N.B. The sums are finite because both $q$ and $n_q$ have finite support and these two terms are equal due to Lemma~\ref{lem:lem2.4} and~\eqref{identity:dandd*}). We finally define the codifferential in the first variable of the above quantity
    \begin{equation*}
        \begin{aligned}
        \left( \di^*_1 \mathcal{G}_{\mathrm{sec}, \mathbf{f}} (x, y, \varphi ; \cdot, \cdot), q \otimes q' \right) & = \sum_{x_1 , y_1 \in \Zd}  \sum_{k , l = 1}^{d \choose 2} \di^*_1  \mathcal{G}_{\mathrm{sec}, \mathbf{f}, \cdot kl}(x, y, \varphi ; x_1, y_2) q_{k} (x_1) q'_l(y_1) \in \R^{d\times {d \choose 2}}, \\
        \left( \di^*_1 \di^*_{3} \di^*_{4} \mathcal{G}_{\mathrm{sec}, \mathbf{f}} (x, y, \varphi ; \cdot, \cdot), n_q \otimes n_{q'} \right) & = \sum_{x_1 , y_1 \in \Zd}  \sum_{k , l = 1}^{d} \di^*_1  \di^*_{3} \di^*_{4} \mathcal{G}_{\mathrm{sec}, \mathbf{f}, \cdot kl}(x, y, \varphi ; x_1, y_2) n_{q,k} (x_1) n_{q',l}(y_1) \in \R^{d  \times {d \choose 2}}.
        \end{aligned}
    \end{equation*}
    \end{itemize}
\end{remark}

In the following proposition, we let $P^{\varphi_\cdot}$ be the heat kernel defined in~\eqref{eq:eqdefheqtkernel}.

\begin{proposition}[Dynamical formula for the second-order Helffer-Sj\"{o}strand Green's matrix] \label{prop:dynamicalsecondorderHS}
    One has the identity, for $(x , y , \varphi , x_1 , y_1) \in \Zd \times \Zd \times \Omega \times \Zd \times \Zd$,
    \begin{equation} \label{eq:dynamicalidentityGreender}
        \mathcal{G}_{\mathrm{sec},\mathbf{f}}(x , y , \varphi ; x_1 , y_1) := \E_{\varphi} \left[ \int_0^\infty \mathbf{f}(\varphi_t) (P^{\varphi_\cdot} (t , x_1  , x ))^T \otimes (P^{\varphi_\cdot} (t , y_1  , y ))^T \, dt \right].
    \end{equation}
\end{proposition}

\begin{remark}
    To be precise, the tensor product and the transposition on the right-hand side of~\eqref{eq:dynamicalidentityGreender} can be written more explicitly by the formula, for any $i , j , k , l \in \{ 1 , \ldots , {d \choose 2} \}$,
    \begin{equation*}
        \mathcal{G}_{\mathrm{sec},\mathbf{f}, ijkl}(x , y , \varphi ; x_1 , y_1) := \E_{\varphi} \left[ \int_0^\infty \mathbf{f}(\varphi_t) P^{\varphi_\cdot}_{ki} (t , x_1  , x ) P^{\varphi_\cdot}_{lj} (t , y_1  , y ) \, dt \right].
    \end{equation*}
    The fact that the right-hand side of~\eqref{eq:dynamicalidentityGreender} displays a tensor product of two heat kernels (which solve the simpler equation~\eqref{eq:eqdefheqtkernel}) is an important property which is a consequence of the observation that the operators $\mathcal{L}_{\mathrm{spat}, x }^{\varphi}$ and $\mathcal{L}_{\mathrm{spat}, y }^{\varphi}$ defined above act on different variables and commute.
\end{remark}

We finally complete this section by stating two upper bounds on the Green's matrix associated with the second-order Helffer-Sj\"{o}strand operator and its mixed codifferential. The proof of this result is omitted here but can be obtained by combining Proposition~\ref{prop.prop4.11chap4HK} together with Proposition~\ref{prop:dynamicalsecondorderHS}.

\begin{proposition} \label{cor:corollary4.14}
For any exponent $p \in (1 , \infty)$, there exist an inverse temperature $\beta_0\left(d , p \right) < \infty$ and a constant $C(d , \beta, p) <\infty$ such that the following statement holds. For any inverse temperature $\beta > \beta_0$ and each $ (x , y , x_1 , y_1) \in \left(\Zd\right)^4$, one has the estimate
\begin{equation*}
\left\| \G_{\mathrm{sec}, \f} (x , y, \cdot ; x_1 , y_1)\right\|_{L^p \left(\Omega \, ; \, \R^{{d \choose 2}^2 \times {d \choose 2}^2} \right)} \leq \frac{C \left\|\f \right\|_{L^p \left( \Omega \, ; \, \R \right)}}{|x - x_1|^{2d-2}_+ + |y - y_1|^{2d-2}_+}.
\end{equation*}
Additionally, one has the estimate
\begin{equation*}
\left\| \di^*_1  \di^*_{3}  \di^*_{4} \G_{\mathrm{sec}, \f} (x , y, \cdot ; x_1 , y_1) \right\|_{L^p \left( \Omega \, ; \, \R^{d \times {d \choose 2} \times d \times d} \right)} \leq \frac{C\left( \ln \left( |x - x_1|_+ + |y - y_1|_+ \right) \right)^{2d+4} \left\|\f \right\|_{L^{2p} \left( \Omega \, ; \, \R  \right)}}{|x - x_1|^{2d + 1}_+ + |y - y_1|^{2d + 1}_+}.
\end{equation*} 
\end{proposition}

\subsection{Notation}

The rest of this section is devoted to the proof of the identity~\eqref{eq:identitydstarGHW} and the inequalities~\eqref{eq:estverticalderivative} and~\eqref{eq:estverticalderivativebis}, and we first set some notation. We let
\begin{equation*}
    \mathcal{K} : 
    \left\{ \begin{aligned}
        \Zd \times \Zd \times \Omega & \to  \R^{{d \choose 2} \times {d \choose 2}} \\
        (y_1 , y_2 , \varphi) & \mapsto   \partial_{y_2} \mathcal{G}_0 (y_1, \varphi).
    \end{aligned} \right.
\end{equation*}
The proof of these results starts from the observation that the map $\mathcal{K}$ solves the second-order Helffer-Sj\"{o}strand equation
\begin{align*}
\mathcal{L}_{sec} \mathcal{K} (y_1, y_2, \varphi) & = - \sum_{q \in \mathcal{Q}} \left( \nabla_q^* \cdot \left( \partial_{y_2} \a_q \right) \nabla_q \mathcal{G}_0 \right)(y_1,\varphi) + 
 \exp \left( U_0 +  U_{\cos , 0} \right) \left( \partial_{y_1} U_0 \right) \otimes \left( \partial_{y_2} U_0 \right) \\ 
 & \quad  + \exp \left( U_0 +  U_{\cos , 0} \right) \left( \partial_{y_1} U_0 \right) \otimes \left( \partial_{y_2} U_{\cos , 0} \right) + \exp \left( U_0 + U_{\cos , 0} \right) \left( \partial_{y_1} \partial_{y_2} U_0 \right)
\end{align*}
where we used the following notation
\begin{equation*}
    - \sum_{q \in \mathcal{Q}} \left(\nabla_q^* \cdot (\partial_{y_2} \a_q) \nabla_q \mathcal{G}_0\right)(y_1,\varphi)
    =
   2\pi \sum_{q \in \mathcal{Q}}  z\left( \beta , q \right) \sin \left( 2\pi\left( \di^* \varphi , n_q \right) \right) \left( \di^* \mathcal{G}_0 , n_q \right) \left( q \otimes q \right) (y_1, y_2) \in \R^{{d \choose 2} \times {d \choose 2}},
\end{equation*}
and we recall from \eqref{formuladerUx} that the following identities hold
\begin{align*}
    \partial_{y_1} U_0(\varphi) & = 2\pi \sum_{q \in \mathcal{Q}}  z(\beta , q)  \cos (2\pi (\di^* \varphi , n_q)) \sin (2\pi(\nabla G, n_q))  q(y_1) \in \R^{d \choose 2}, \\
    \partial_{y_2} U_0(\varphi) & = 2\pi  \sum_{q \in \mathcal{Q}} z(\beta , q)  \cos (2\pi (\di^* \varphi , n_q)) \sin (2\pi(\nabla G , n_q))  q(y_2) \in \R^{d \choose 2}, \notag \\
    \partial_{y_2} U_{\cos, 0}(\varphi) & = - 2\pi  \sum_{q \in \mathcal{Q}} z(\beta , q)  \sin (2\pi (\di^* \varphi , n_q)) \left( \cos (2\pi(\nabla G , n_q)) - 1 \right)  q(y_2) \in \R^{d \choose 2}, \notag \\
 \partial_{y_1} \partial_{y_2} U_0(\varphi) & = - (2\pi)^2 \sum_{q \in \mathcal{Q}}  z(\beta , q)  \sin (2\pi (\di^* \varphi , n_q)) \sin (2\pi(\nabla G , n_q))  \left( q \otimes q \right) (y_1, y_2) \in \R^{{d \choose 2} \times {d \choose 2}}. \notag
\end{align*}
Using linearity,  we write $\mathcal{K} = \mathcal{H}_1 + \mathcal{H}_2 + \mathcal{H}_3 + \mathcal{H}_4$,  where 
\begin{align}
\label{e.v1}
\mathcal{L}_{sec}  \mathcal{H}_1 (y_1, y_2, \varphi) & = - \sum_{q \in \mathcal{Q}} \left(\nabla_q^* \cdot (\partial_{y_2} \a_q) \nabla_q \mathcal{G}_0 \right)(y_1, \varphi), \\
\mathcal{L}_{sec} \mathcal{H}_2(y_1, y_2, \varphi) & = \exp \left( U_0(\varphi) +  U_{\cos , 0}(\varphi)  \right) \left( \partial_{y_1} \partial_{y_2} U_0(\varphi) \right) ,\notag \\
\mathcal{L}_{sec} \mathcal{H}_3(y_1, y_2, \varphi) & =  
 \exp \left( U_0(\varphi) +   U_{\cos , 0}(\varphi)  \right) \left( \partial_{y_1} U_0(\varphi) \right) \otimes \left( \partial_{y_2} U_0 (\varphi) \right). \notag \\
 \mathcal{L}_{sec} \mathcal{H}_4(y_1, y_2, \varphi) & =  
 \exp \left( U_0(\varphi)  +  U_{\cos , 0}(\varphi)  \right) \left( \partial_{y_1} U_0(\varphi) \right) \otimes \left( \partial_{y_2} U_{\cos,0} (\varphi) \right). \notag
\end{align}
We will prove the following estimates on the functions $\mathcal{H}_1$, and $\mathcal{H}_2$: for any pair $y_1 , y_2 \in \Zd$,
\begin{align*}
     \| \di^*_{1}\mathcal{H}_1(y_1, y_2, \cdot) \| _{L^2\left(\Omega \, ; \, \R^{d \times \binom d2}\right)} & \leq C \frac{\left(\ln (|y_1|_+ + |y_2|_+) \right)^{3d+7} }{|y_1|_+^{d-1} + |y_2|_+^{d-1}}\frac{1}{|y_1 -y_2|_+^{d + 1}}, \\
     \| \di^*_{1} \mathcal{H}_2(y_1, y_2, \cdot) \|_{L^2\left(\Omega \, ; \, \R^{d \times \binom d2}\right)}
    & \leq
     C \frac{\left( \ln (|y_1|_+ + |y_2|_+) \right)^{2d+4} }{|y_1|_+^{d-1} + |y_2|_+^{d-1}} \frac{1}{|y_1 -y_2|_+^{d +1}}.
\end{align*}
The term $\mathcal{H}_3$ and $\mathcal{H}_4$ are the most complex to estimate, and we will prove the following result: there exist two functions $\mathcal{W}_3 , \mathcal{W}_4 :  \Zd \times \Zd \times \Omega \to \R^{d \times d}$ such that
\begin{equation*}
     \mathcal{H}_3(y_1 , y_2 , \varphi)  = \di_{2} \mathcal{W}_3(y_1 , y_2 , \varphi) \hspace{4mm} \mbox{and} \hspace{4mm} \mathcal{H}_4 (y_1 , y_2 , \varphi) = \di_{2} \mathcal{W}_4(y_1 , y_2 , \varphi).
\end{equation*}
together with the inequality
\begin{equation*}
     \| \di^*_{1}\mathcal{W}_3(y_1, y_2, \cdot) \|_{L^2(\Omega \, ; \, \R^{d \times d})} + \| \di^*_{1}\mathcal{W}_4(y_1, y_2, \cdot) \|_{L^2(\Omega \, ; \, \R^{d \times d})} \leq C \frac{(\ln |y_1|_+)^{d + 7/2}}{|y_1|_+^{d-1}} \times \frac{(\ln |y_2|_+)^{d + 7/2}}{|y_2|_+^{d-1}}.
\end{equation*}
The identity~\eqref{eq:identitydstarGHW} follows by setting
\begin{equation*}
    \mathcal{H} := \mathcal{H}_1 + \mathcal{H}_2 \hspace{4mm} \mbox{and}  \hspace{4mm} \mathcal{W} := \mathcal{W}_3 + \mathcal{W}_4.
\end{equation*}

\subsection{Estimating the term $\mathcal{H}_1$}

We first solve the equation~\eqref{e.v1} using the second-order Green's matrix. Specifically, we have the identity
\begin{align*}
    \mathcal{H}_1(y_1, y_2, \varphi) & = (2 \pi) \sum_{q \in \mathcal{Q}} z\left( \beta , q \right)  \left( \mathcal{G}_{\mathrm{sec}, \mathbf{f}_q } \left( y_1 , y_2 , \varphi ; \cdot , \cdot \right) , q \otimes q \right) \\
    & = (2 \pi) \sum_{q \in \mathcal{Q}} z\left( \beta , q \right)  \left( \di_{3}^*  \di_{4}^*  \mathcal{G}_{\mathrm{sec}, \mathbf{f}_q } \left( y_1 , y_2 , \varphi ; \cdot , \cdot \right) , n_q \otimes n_q \right)
\end{align*}
with
\begin{align*}
\f_{q} (\varphi) =  \sin (2\pi\left(  \di^* \varphi , n_q  \right))  \left( \di^* \mathcal{G}_0 (\cdot , \varphi), n_q \right).
\end{align*}
We apply the decay estimates for the second order Green's matrix stated Proposition~\ref{cor:corollary4.14}, the inequality~\eqref{eq:boundonG0} on the map $\di^* \mathcal{G}_0$ (N.B. to be fully rigorous, the inequality~\eqref{eq:boundonG0} should be established for the $L^4(\Omega \, ; \, \R^{d})$-norm instead of the $L^2(\Omega \, ; \, \R^{d})$-norm and this can be achieved by increasing the value of $\beta$ if necessary) and Lemma~\ref{lemma.lemma2.5} to obtain
\begin{align*}
   \|  \di^*_{1} \mathcal{H}_1(y_1, y_2, \cdot) \| _{L^2\left(\Omega \, ; \, \R^{d \times \binom d2}\right)}
    & \leq \sum_{q \in \mathcal{Q}}\frac{C_q  |z\left( \beta , q \right)|  \left( \ln (|y_1-z_q|_+ + |y_2 - z_q|_+) \right)^{2d+4} }{|y_1-z_q|_+^{2d + 1} + |y_2 - z_q|_+^{2d + 1}} \frac{\left( \ln | z_q|_+ \right)^{d+3}}{|z_q|^{d-1}_+} \\
    & \leq \sum_{z \in \Zd} \frac{C \left( \ln (|y_1-z|_+ + |y_2 - z|_+) \right)^{2d+4}}{|y_1-z|_+^{2d+1} + |y_2 - z|_+^{2d+1}}  \frac{\left( \ln | z|_+ \right)^{d+3}}{|z|^{d-1}_+}.
\end{align*}
Using Proposition~\ref{prop:appCdiscretesum} of Appendix~\ref{App.sumonlattice} (to estimate the sum on the lattice), we obtain
\begin{align*}
     \|  \di^*_{1} \mathcal{H}_1(y_1, y_2, \cdot) \| _{L^2\left(\Omega \, ; \, \R^{d \times \binom d2}\right)}
    & \leq
     C \frac{\left(\ln (|y_1|_+ + |y_2|_+) \right)^{3d+7} }{|y_1|_+^{d-1} + |y_2|_+^{d-1}}\frac{1}{|y_1 -y_2|_+^{d + 1}}.  \notag
\end{align*}

\subsection{Estimating the term $\mathcal{H}_2$}

We first start with the definition of $\mathcal{H}_2$
\begin{equation*}
    \mathcal{L}_{sec} \mathcal{H}_2(y_1, y_2, \varphi) =  
 - (2\pi)^2 \exp \left( U_0 + U_{\cos , 0} \right) \sum_{q \in \mathcal{Q}}  z(\beta , q)  \sin (2\pi (\di^* \varphi , n_q)) \sin (2\pi(\nabla G , n_q)) ( q \otimes q) (y_1 , y_2).
\end{equation*}
We then deduce the following formula
\begin{align*}
    \mathcal{H}_2(y_1, y_2, \varphi) & = - (2 \pi)^2 \sum_{q \in \mathcal{Q}} z(\beta , q) \sin (2\pi(\nabla G , n_q)) \left( \mathcal{G}_{sec, \mathbf{f}_{q}}(y_1, y_2, \varphi ; \cdot, \cdot), q \otimes q \right) \\
    & = - (2 \pi)^2 \sum_{q \in \mathcal{Q}} z(\beta , q) \sin (2\pi(\nabla G , n_q)) \left( \di^*_{w_1} \di^*_{w_2} \mathcal{G}_{sec, \mathbf{f}_{q}}(y_1, y_2, \varphi ; \cdot, \cdot), n_q \otimes n_q \right)
\end{align*}
with
\begin{equation*}
    \mathbf{f}_{q}(\varphi) := \exp(U_0(\varphi) + U_{\cos , 0}(\varphi)) \sin (2\pi (\di^* \varphi , n_q)).
\end{equation*}
Applying the regularity estimate for the second-order Helffer-Sj\"{o}strand equation together with Proposition~\ref{prop:appCdiscretesum} of Appendix~\ref{App.sumonlattice}, we deduce that
\begin{align*}
     \|  \di^*_{1} \mathcal{H}_2(y_1, y_2, \cdot) \|_{L^2\left(\Omega \, ; \, \R^{d \times \binom d2}\right)}
   & \leq
     \sum_{z \in \Zd} \frac{C\left( \ln (|y_1-z|_+ + |y_2 - z|_+) \right)^{2d+4} }{|y_1-z|_+^{2d + 1} + |y_2 - z|_+^{2d+1}} \frac{1}{|z|^{d-1}_+} \\
   & \leq
     C \frac{\left( \ln (|y_1|_+ + |y_2|_+) \right)^{2d+4} }{|y_1|_+^{d-1} + |y_2|_+^{d-1}} \frac{1}{|y_1 -y_2|_+^{d +1}}.
\end{align*}

\subsection{Estimating the term $\mathcal{H}_3$} \label{subsecB23}
This term is the most intricate to estimate and we first recall the definition of the function $\mathcal{H}_3$
\begin{multline*}
    \mathcal{L}_{sec} \mathcal{H}_3(y_1, y_2, \varphi) =  
 \exp \left( U_0 (\varphi) + U_{\cos , 0} (\varphi) \right) \left( \sum_{q \in \mathcal{Q}} 2\pi z(\beta , q)  \cos (2\pi (\di^* \varphi , n_q)) \sin (2\pi(\nabla G , n_q))  q(y_1) \right) \\
 \otimes \left( \sum_{q \in \mathcal{Q}} 2\pi z(\beta , q)  \cos (2\pi (\di^* \varphi , n_q)) \sin (2\pi(\nabla G , n_q))  q(y_2) \right).
\end{multline*}
We will prove that there exists a function $\mathcal{W}_3 : \Zd \times \Zd \times \Omega \to \R^{{d \choose 2} \times d }$ such that, for any $y_1 , y_2 \in \Zd$ and $\varphi \in \Omega$,
\begin{equation} \label{eq:defWdd}
    \mathcal{H}_3(y_1, y_2, \varphi) = \di_{2} \mathcal{W}_3 (y_1 , y_2 , \varphi)
\end{equation}
and
\begin{equation} \label{eq:boundWdd}
\left\| \di^*_{1} \mathcal{W}_3(y_1 , y_2 , \cdot)  \right\|_{L^2\left(\Omega \, ; \, \R^{d \times d} \right)} \leq C \frac{(\ln |y_1|_+)^{d + 7/2}}{|y_1|_+^{d-1}} \times \frac{(\ln |y_2|_+)^{d + 7/2}}{|y_2|_+^{d-1}}.
\end{equation}
We first introduce some notation. For each pair of charges $q , q' \in \mathcal{Q}$, let us set
\begin{equation*}
    \mathbf{f}_{q , q'}(\varphi) := (2 \pi)^2 \exp \left( U_0(\varphi) + U_{\cos , 0}(\varphi)\right) \cos (2\pi (\di^* \varphi , n_q)) \cos (2\pi (\di^* \varphi , n_{q'})).
\end{equation*}
We note that, if the inverse temperature $\beta$ is chosen sufficiently large (see Remark~\ref{rem:remark3.3}), then we have the bound, for any pair of charges $q , q' \in \mathcal{Q}$,
\begin{equation} \label{eq:boundonfqq'}
    \left\| \mathbf{f}_{q , q'} \right\|_{L^4(\Omega \, ; \, \R)} \leq C.
\end{equation}
We additionally let $\mathcal{H}_{3, q , q'} : \Zd \times \Zd \times \Omega \to \R^{\binom d2 \times \binom d2}$ be the solution of the equation 
\begin{equation*}
    \mathcal{L}_{sec} \mathcal{H}_{3, q , q'}(y_1, y_2, \varphi) = \mathbf{f}_{q , q'}(\varphi) (q \otimes q')(y_1 , y_2).
\end{equation*}
Let us note that these functions are defined so as to have the identity
\begin{equation*}
    \mathcal{H}_3 = \sum_{q , q' \in \mathcal{Q}} z(\beta , q) z(\beta , q') \sin (2\pi(\nabla G , n_q)) \sin (2\pi(\nabla G , n_{q'})) \mathcal{H}_{3, q , q'}.
\end{equation*}
We will show that, for each pair of charges $q , q' \in \mathcal{Q}$, there exists a function 
\begin{equation} \label{eq:defWqq'obj}
    \mathcal{H}_{3, q , q'}(y_1, y_2, \varphi) = \di_{2} \mathcal{W}_{q , q'} (y_1 , y_2 , \varphi)
\end{equation}
together with the inequality (where $C_q$ and $C_{q'}$ are two constants which may depend on $q$ and $q'$ but may only grow polynomially fast in $\| q \|_1$ and $\| q' \|_1$)
\begin{equation} \label{eq:boundWqq'obj}
\left\| \di^*_{1} \mathcal{W}_{q,q'}(y_1 , y_2 , \cdot)  \right\|_{L^2\left(\Omega \, ; \, \R^{d \times d} \right)} \leq  C_q C_{q'} \frac{(\ln (|y_1 - z_q |_+)^{d + 5/2}}{|y_1 - z_q |_+^{d} } \times \frac{  ( \ln |y_2 - z_{q'}|_+))^{d+5/2} }{|y_2 - z_{q'}|_+^{d}}.
\end{equation}
We first show how to complete the proof of the identity~\eqref{eq:defWdd} and the bound~\eqref{eq:boundWdd} using~\eqref{eq:defWqq'obj} and~\eqref{eq:boundWqq'obj}. We set
\begin{equation*}
    \mathcal{W}_3 := \sum_{q , q' \in \mathcal{Q}} z(\beta , q) z(\beta , q') \sin (2\pi(\nabla G , n_q)) \sin (2\pi(\nabla G , n_{q'})) \mathcal{W}_{ q , q'}
\end{equation*}
and compute, using~\eqref{eq:boundWqq'obj} and Lemma~\ref{lemma.lemma2.5},
\begin{align*}
    \lefteqn{\left\|  \di^*_{1} \mathcal{W}_3(y_1 , y_2 , \cdot)  \right\|_{L^2\left(\Omega \, ; \, \R^{d \times d} \right)} } \qquad & \\ & \leq \sum_{q , q' \in \mathcal{Q}} |z(\beta , q)| |z(\beta , q')| C_q C_{q'} \left| (\nabla G , n_{q}) \right| \left| (\nabla G , n_{q'}) \right|  \frac{(\ln (|y_1 - z_q |_+)^{d + 5/2}}{|y_1 - z_q |_+^{d} } \times \frac{  ( \ln |y_2 - z_{q'}|_+))^{d+5/2} }{|y_2 - z_{q'}|_+^{d}} \\
    & \leq C \sum_{z , z' \in \Zd} \frac{1}{|z|_+^{d-1}} \frac{1}{|z'|_+^{d-1}} \frac{(\ln (|y_1 - z |_+)^{d + 5/2}}{|y_1 - z |_+^{d} } \times \frac{  ( \ln |y_2 - z'|_+))^{d+5/2} }{|y_2 - z'|_+^{d}} \\
    & = C \left(  \sum_{z \in \Zd} \frac{1}{|z|_+^{d-1}} \frac{(\ln (|y_1 - z |_+)^{d + 5/2}}{|y_1 - z |_+^{d}} \right) \left(  \sum_{ z' \in \Zd} \frac{1}{|z'|_+^{d-1}} \frac{(\ln (|y_2- z'|_+)^{d + 5/2}}{|y_2 - z'|_+^{d}} \right) \\
    & \leq C \frac{(\ln |y_1|_+)^{d + 7/2}}{|y_1|_+^{d-1}} \times \frac{(\ln |y_2|_+)^{d + 7/2}}{|y_2|_+^{d-1}}.
\end{align*}
This completes the proofs of~\eqref{eq:defWdd} and~\eqref{eq:boundWdd}. 

The rest of the proof is devoted to the identity~\eqref{eq:defWqq'obj} and the bound~\eqref{eq:boundWqq'obj}. We appeal to the dynamical interpretation of the Helffer-Sj\"{o}strand equation presented in Section~\ref{sec:sec264dynamic}. Applying Proposition~\ref{prop:prop2.26} and Remark~\ref{rem:reversinghtetime} (together with the linearity of the Helffer-Sj\"{o}strand equation), we obtain the identity
\begin{equation*}
   \mathcal{H}_{3, q , q'}(y_1, y_2, \varphi) = \int_{0}^\infty  \E_{\varphi} \left[   \mathbf{f}_{q , q'}(\varphi_t)   P^{\bar  \varphi_\cdot^t}_{q_1 , q_2} (t , y_1 , y_2)\right] \, dt,  
\end{equation*}
where, given a trajectory $(\varphi_t)_{t \geq 0}$ of the Langevin dynamic, we recall the definition of the time reversed dynamic
$$(\bar \varphi^t_s)_{0 \leq s \leq t} : = (\varphi_{t-s})_{0 \leq s \leq t}$$
and let $P^{\bar \varphi_\cdot^t}_{q_1, q_2}(\cdot , \cdot , \cdot) : [0 , t] \times \Zd \times \Zd \to \R^{\binom d2 \times \binom d2}$ be the solution of the parabolic system
\begin{equation} \label{eq:TV20240}
        \left\{ \begin{aligned}
        \partial_t P^{\bar  \varphi_\cdot^t}_{q_1 , q_2}+ \left(\mathcal{L}_{\mathrm{spat} , x}^{\bar  \varphi_\cdot^t} + \mathcal{L}_{\mathrm{spat} , y}^{\bar  \varphi_\cdot^t}\right)  P^{\bar  \varphi_\cdot^t}_{ q_1 , q_2} & =0 &~\mbox{in}~ &[0 , t] \times \Zd \times \Zd, \\
        P^{\bar  \varphi_\cdot^t}_{ q_1 , q_2} \left(0,y , z\right) & = (q_1\otimes q_2)(y,z) &~\mbox{for}~ &y , z \in \Zd.
        \end{aligned} \right.
\end{equation}
We note that from the definition of the operators $\mathcal{L}_{\mathrm{spat} , x}^{\bar  \varphi_\cdot^t}$  and $\mathcal{L}_{\mathrm{spat} , y}^{\bar  \varphi_\cdot^t}$ (and since they commute), the solution of the equation~\eqref{eq:TV20240} factorizes and one can write 
$$ P^{\bar  \varphi_\cdot^t}_{ q_1 , q_2}(t , y , z) = P^{\bar  \varphi_\cdot^t}_{ q_1} (t , y ) \otimes  P^{\bar  \varphi_\cdot^t}_{ q_2} (t , z), $$ 
where the maps $ P^{\bar  \varphi_\cdot^t}_{ q_1} : [0 , t] \times \Zd \to \R^{\binom d2}$ and $ P^{\bar  \varphi_\cdot^t}_{ q_2} : [0 , t] \times \Zd \to \R^{\binom d2}$ are the solutions of the parabolic systems
\begin{equation*}
     \left\{ \begin{aligned}
        \partial_t P^{\bar  \varphi_\cdot^t}_{q_1}+ \mathcal{L}_{\mathrm{spat}}^{\bar  \varphi_\cdot^t}  P^{\bar  \varphi_\cdot^t}_{ q_1} & =0 &~\mbox{in}~& [0 , t] \times \Zd , \\
        P^{\bar  \varphi_\cdot^t}_{ q_1 } \left(0,\cdot\right) & = q_1 &~\mbox{in}~& \Zd,
        \end{aligned} \right.
        \hspace{5mm}\mbox{and}  \hspace{5mm}
         \left\{ \begin{aligned}
        \partial_t P^{\bar  \varphi_\cdot^t}_{q_2}+ \mathcal{L}_{\mathrm{spat}}^{\bar  \varphi_\cdot^t} P^{\bar  \varphi_\cdot^t}_{ q_2} & =0 &~\mbox{in}~& [0 , t] \times \Zd, \\
        P^{\bar  \varphi_\cdot^t}_{ q_2 } \left(0,\cdot\right) & = q_2 &~\mbox{in}~& \Zd,
        \end{aligned} \right.
\end{equation*}
where we have set, for each charge $q \in \mathcal{Q}$,
    \begin{equation*}
     \left\{ \begin{aligned}
        \partial_t P^{\bar  \varphi_\cdot^t}_{q_1}+ \mathcal{L}_{\mathrm{spat}}^{\bar  \varphi_\cdot^t}  P^{\bar  \varphi_\cdot^t}_{q} & =0 ~\mbox{in}~ [0 , t] \times \Zd , \\
        P^{\bar  \varphi_\cdot^t}_{q} \left(0,\cdot\right) & = q ~\mbox{in}~ \Zd.
        \end{aligned} \right.
\end{equation*}
    Let us note that, by the definition of the heat kernel (see Proposition~\ref{prop:prop2.26}), we have the identity
    \begin{equation*}
        P^{\bar  \varphi_\cdot^t}_{q} \left(t,x\right) = \left( P^{\bar  \varphi_\cdot^t}(t , x ; \cdot) , q \right) = \left( \di^*_2 P^{\bar  \varphi_\cdot^t}(t , x ; \cdot) , n_q \right).
    \end{equation*}
    Using the upper bound on the heat kernel stated in Proposition~\ref{prop.prop4.11chap4HK} (with the exponent $p = 8$), we have the upper bounds
    \begin{equation} \label{eq:ineqonPq}
        t^{1/2}\left\|  P^{\bar  \varphi_\cdot^t}_{q} \left(t,y\right) \right\|_{L^8\left(\Omega \, ; \, \R^{{d \choose 2}}\right)} + t \left\|  \di^* P^{\bar  \varphi_\cdot^t}_{q} \left(t,y\right) \right\|_{L^8(\Omega \, ; \, \R^{d})} \leq \frac{C_q  (\ln t_+)^{d+2}}{t_+^{\frac{d}{2}}} \exp \left( - \frac{|y|}{C \sqrt{t_+}} \right).
    \end{equation}
    As a consequence, we have the inequality
    \begin{equation} \label{eq:ineqonPqbis2}
        \left\| \nabla_{q'} P^{\bar  \varphi_\cdot^t}_{q}(t)  \right\|_{L^8(\Omega \, ; \, \R)} = \left\|  \left( \di^* P^{\bar  \varphi_\cdot^t}_{q}(t , \cdot) , n_{q'} \right)  \right\|_{L^8(\Omega \, ; \, \R)} \leq \frac{C_q C_{q'}(\ln t_+)^{d+2} }{t_+^{\frac{d}{2} + 1}} \exp \left( - \frac{|z_{q'} - z_{q}|}{C\sqrt{t_+}} \right).
    \end{equation}
    The strategy is then to define the function $Q^{\varphi_\cdot}_{q}$ to be the solution of the parabolic system
\begin{equation} \label{eq:TV20150}
    \left\{ \begin{aligned}
        \partial_t Q^{\bar  \varphi_\cdot^t}_{q} - \left( \frac{1}{2\beta} \Delta - \frac{1}{2\beta}\sum_{n \geq 1} \frac{1}{\beta^{ \frac n2}} (-\Delta)^{n+1} \right) Q^{\bar  \varphi_\cdot^t}_{ q} & = - \sum_{q' \in \mathcal{Q}}z \left( \beta , q' \right) \cos \left( 2\pi\left( \bar \varphi_t, q' \right) \right) \nabla_{q'} P^{\bar  \varphi_\cdot^t}_{q} n_{q'}  &\mbox{in}~ [0,t] \times \Zd,\\
        Q^{\bar  \varphi_\cdot^t}_{ q} \left(0,\cdot\right) & =  n_{q} &\mbox{in}~\Zd.
        \end{aligned} \right.
\end{equation}
By applying the exterior derivative to both sides of the identities above, one sees that, for each $y \in \Zd$ and each time $t \geq 1$,
    \begin{equation} \label{eq:TV16500}
        P^{\bar  \varphi_\cdot^t}_{q}(t,y) = \di Q^{\bar  \varphi_\cdot^t}_{q}(t,y).
    \end{equation}
We will then prove the inequality
\begin{equation} \label{eq:functionQ_q}
     \left\| Q^{\bar  \varphi_\cdot^t}_{q}(t , y) \right\|_{L^8(\Omega \, ; \, \R^d)} \leq \frac{C_q(\ln t_+)^{d+2}}{t_+^{\frac d2}} \exp \left( - \frac{|y - z_{q}|}{C \sqrt{t_+}} \right).
\end{equation}
In order to simplify the argument, we assume, without loss of generality, that $z_{q} = 0$. We will first use the Duhamel principle to solve the equation~\eqref{eq:TV20150}. To this end, we introduce the fundamental solution associated with the system~\eqref{eq:TV20150}, i.e., the matrix valued function $\mathbf{Q} : (0 , \infty) \times \Zd \to \R^{ d\times d }$ which solves the equation
\begin{equation*} 
    \left\{ \begin{aligned}
        \partial_t \mathbf{Q} - \left( \frac{1}{2\beta} \Delta - \frac{1}{2\beta}\sum_{n \geq 1} \frac{1}{\beta^{ \frac n2}} (-\Delta)^{n+1} \right) \mathbf{Q} & = 0 &\mbox{in}~ (0 , \infty) \times \Zd,\\
        \mathbf{Q} \left(0,\cdot \right) & =  \delta_0 &\mbox{in}~\Zd.
        \end{aligned} \right.
\end{equation*}
We note that the operator on the left-hand side (i.e., the Laplacian and the sum of the iterations of the Laplacian) is a deterministic operator with constant coefficient. In this case, the matrix $\mathbf{Q}$ is always a multiple of the identity matrix, and its coefficients can be upper bounded as follows (N.B. this is the same estimate as for the standard heat kernel as the iterations of the Laplacian can be proved to be higher-order terms which have a negligible contribution for large times $t$)
\begin{equation} \label{upperboundheatkernel}
    \forall (t , y) \in (0, \infty) \times \Zd \times \Zd, \hspace{5mm} \left| \mathbf{Q}(t ,y) \right| + t_+^{\frac12} \left| \nabla \mathbf{Q}(t ,y) \right| \leq \frac{C}{t_+^{\frac{d}{2}}} \exp \left( - \frac{|y|}{C \sqrt{t_+}} \right).
\end{equation}
Using the Duhamel principle, we have the identity
\begin{equation*}
    Q^{\bar  \varphi_\cdot^t}_{q}(t , y) = \left( \mathbf{Q}(t , y - \cdot) ,  n_q \right) - \sum_{q' \in \mathcal{Q}} z \left( \beta , q' \right) \int_0^t  \cos \left( 2\pi\left( \bar \varphi_s^t, q' \right) \right) \left( \di^* P^{\bar  \varphi_\cdot^t}_{q}(s , \cdot) , n_{q'} \right)  \left( \mathbf{Q}(t-s , y - \cdot) , n_{q'} \right) \, ds.
\end{equation*}
We then estimate the two terms on the right-hand side by using the upper bound~\eqref{eq:ineqonPqbis2} on the map $P_q^{\bar \varphi^t_\cdot}$ and the upper bound~\eqref{upperboundheatkernel}. We thus obtain
\begin{multline*}
    \left\| Q^{\bar  \varphi_\cdot^t}_{q}(t , y) \right\|_{L^8(\Omega \, ; \, \R^d)} \\
    \leq \frac{C_q}{t_+^{\frac{d}{2}}} \exp \left( - \frac{|y|}{C\sqrt{t_+}} \right) + \sum_{q' \in \mathcal{Q}} C_{q'} z \left( \beta , q' \right) \int_0^t \frac{(\ln s_+)^{d+2} }{s_+^{\frac{d}{2} + 1}} \exp \left( - \frac{|z_{q'}|}{C\sqrt{s_+}}\right) \times \frac{1}{(t-s)_+^{\frac{d}{2}}} \exp \left( - \frac{|y- z_{q'}|}{C\sqrt{(t-s)_+}}\right) \, ds.
\end{multline*}
We next estimate the second term on the right-hand side. Using the same computation as in Lemma~\ref{lemma.lemma2.5}, we may first write
\begin{align*}
    \lefteqn{\sum_{q' \in \mathcal{Q}} z \left( \beta , q' \right) \frac{C_{q'}}{s_+^{\frac{d}{2}}} \exp \left( - \frac{|z_{q'}|}{C\sqrt{s_+}}\right)  \frac{C_{q'}}{(t-s)_+^{\frac{d}{2}}} \exp \left( - \frac{|y- z_{q'}|}{C\sqrt{(t-s)_+}}\right) } \qquad & \\ & 
     \leq C  \sum_{z \in \Zd} \frac{1}{s_+^{\frac{d}{2}}} \exp \left( - \frac{|z|}{C\sqrt{s_+}}\right)  \frac{1}{(t-s)_+^{\frac{d}{2}}} \exp \left( - \frac{|y - z|}{C\sqrt{(t-s)_+}}\right) \\ 
    & \leq \frac{C}{t_+^{\frac{d}{2}}}  \exp \left( - \frac{|y|}{C \sqrt{t_+}} \right).
\end{align*}
We may then combine the two previous displays so as to obtain 
\begin{align*}
    \left\| Q^{\bar  \varphi_\cdot^t}_{q}(t , y) \right\|_{L^8(\Omega \, ; \, \R^d)} & \leq \frac{C_q}{t_+^{\frac{d}{2}}} \exp \left( - \frac{|y|}{C\sqrt{t_+}} \right) +  \left( \int_0^t  \frac{(\ln s_+)^{d+2}}{s_+} \, ds \right) \frac{C}{t_+^{\frac{d}{2}}}  \exp \left( - \frac{|y|}{C \sqrt{t_+}} \right) \\
    & \leq \frac{C_q}{t_+^{\frac{d}{2}}} \exp \left( - \frac{|y|}{C\sqrt{t}} \right) + \frac{C (\ln t_+)^{d+3}}{t_+^{\frac{d}{2}}} \exp \left( - \frac{|y|}{C \sqrt{t_+}} \right) \\
    & \leq \frac{C_q (\ln t_+)^{d+3}}{t_+^{\frac{d}{2}}} \exp \left( - \frac{|y|}{C\sqrt{t_+}} \right) .
\end{align*}
This completes the proof of the inequality~\eqref{eq:functionQ_q}.

Equipped with the identities~\eqref{eq:TV16500} and~\eqref{eq:functionQ_q}, we may set
\begin{equation*}
    \mathcal{W}_{q , q'} (y_1, y_2, \varphi) := \int_{0}^\infty  \E_{\varphi} \left[   \mathbf{f}_{q , q'}(\varphi_t)  P^{\bar  \varphi_\cdot^t}_{q} (t , y_1) \otimes Q^{\bar  \varphi_\cdot^t}_{q'} (t , y_2) \right] \, dt.
\end{equation*}
We next verify that this function satisfies the identity~\eqref{eq:defWqq'obj} and the bound~\eqref{eq:boundWqq'obj}. For~\eqref{eq:defWqq'obj}, we note that
\begin{align*}
    \di_{2} \mathcal{W}_{q , q'} (y_1 , y_2 , \varphi) & = \int_{0}^\infty  \E_{\varphi} \left[   \mathbf{f}_{q , q'}(\varphi_t)   P^{\bar  \varphi_\cdot^t}_{q}  (t , y_1 ) \otimes  \di Q^{\bar  \varphi_\cdot^t}_{q'} (t , y_2 ) \right] \, dt \\
    & =  \int_{0}^\infty  \E_{\varphi} \left[   \mathbf{f}_{q , q'}(\varphi_t)   P^{\bar  \varphi_\cdot^t}_{q} (t , y_1 ) \otimes   P^{\bar  \varphi_\cdot^t}_{q'} (t , y_2 ) \right] \, dt \\
    & = \mathcal{H}_{3, q , q'}(y_1, y_2, \varphi).
\end{align*}
For the upper bound~\eqref{eq:boundWqq'obj}, we use the bounds~\eqref{eq:boundonfqq'},~\eqref{eq:ineqonPq} and~\eqref{eq:functionQ_q} and the Jensen inequality to deduce that
\begin{align*}
    \left\| \di^*_{1} \mathcal{W}_{q,q'}(y_1 , y_2 , \cdot)  \right\|_{L^2\left(\Omega \, ; \, \R^{d \times d} \right)} & = 
    \left\| \int_{0}^\infty  \E_{\varphi} \left[   \mathbf{f}_{q , q'}(\varphi_t)   \di^* P^{\bar  \varphi_\cdot^t}_{q}(t , y_1 ) \otimes  Q^{\bar  \varphi_\cdot^t}_{q'} (t , y_2 ) \right] \, dt \right\|_{L^2\left(\Omega \, ; \, \R^{d \times d} \right)} \\
    & \leq \int_0^\infty \frac{C_q(\ln t_+)^{d+2}}{t_+^{\frac{d}{2} + 1}} \exp \left( - \frac{|y_1 - z_q |}{C \sqrt{t_+}} \right)  \frac{C_{q'}(\ln t_+)^{d+3}}{t_+^{\frac d2}} \exp \left( - \frac{|y_2 - z_{q'}|}{C \sqrt{t_+}} \right) \, dt.
\end{align*}
The right-hand side can be estimated as follows
\begin{align*}
    \left\| \di^*_{1} \mathcal{W}_{q,q'}(y_1 , y_2 , \cdot)  \right\|_{L^2\left(\Omega \, ; \, \R^{d \times d} \right)} 
    & \leq \int_0^\infty \frac{C_q C_{q'} (\ln t_+)^{2d+5}}{t_+^{d + 1}} \exp \left( - \frac{|y_1 - z_q | + |y_2 - z_{q'}|}{C \sqrt{t_+}}\right) \, dt \\
    & \leq \frac{C_q C_{q'}(\ln (|y_1 - z_q |_+ + |y_2 - z_{q'}|_+))^{2d+5} }{\left( |y_1 - z_q |_+ + |y_2 - z_{q'}|_+ \right)^{2d}} \\
    & \leq C_q C_{q'} \frac{(\ln (|y_1 - z_q |_+)^{d + 5/2}}{|y_1 - z_q |_+^{d} } \times \frac{  ( \ln |y_2 - z_{q'}|_+))^{d+5/2} }{|y_2 - z_{q'}|_+^{d}},
\end{align*}
where in the last inequality we used the lower bound: there exists $c > 0$ such that for all $a , b \geq 1$,
$$(\ln (a+b) )^{2d+5} (a + b)^{2d} \geq c (\ln a )^{d+5/2} a^{d} \times (\ln b )^{d+5/2} b^{d}.$$

\subsection{Estimating the term $\mathcal{H}_4$}
The estimation of this term is essentially identical to (and in fact simpler than) the one of the term $\mathcal{H}_3$. We first recall the definition of the function $\mathcal{H}_4$
\begin{multline*}
    \mathcal{L}_{sec} \mathcal{H}_4(y_1, y_2, \varphi) =  
 \exp \left( U_0 (\varphi) + U_{\cos , 0} (\varphi) \right) \left( \sum_{q \in \mathcal{Q}} 2\pi z(\beta , q)  \cos (2\pi (\di^* \varphi , n_q)) \sin (2\pi(\nabla G , n_q))  q(y_1) \right) \\
 \otimes \left( - \sum_{q \in \mathcal{Q}} 2\pi z(\beta , q)  \sin (2\pi (\di^* \varphi , n_q)) \left( \cos (2\pi(\nabla G , n_q)) -1 \right)  q(y_2) \right).
\end{multline*}
We may then show that there exists a function $\mathcal{W}_4 : \Zd \times \Zd \times \Omega \to \R^{{d \choose 2} \times d }$ such that, for any $y_1 , y_2 \in \Zd$ and $\varphi \in \Omega$,
\begin{equation*} 
    \mathcal{H}_4(y_1, y_2, \varphi) = \di_{2} \mathcal{W}_4 (y_1 , y_2 , \varphi)
\end{equation*}
and
\begin{equation*} 
\left\| \di^*_{1} \mathcal{W}_4(y_1 , y_2 , \cdot)  \right\|_{L^2\left(\Omega \, ; \, \R^{d \times d} \right)} \leq C \frac{(\ln |y_1|_+)^{d + 7/2}}{|y_1|_+^{d-1}} \times \frac{(\ln |y_2|_+)^{d + 7/2}}{|y_2|_+^{d-1}}.
\end{equation*}
The proof of this identity and inequality follows exactly the same reasoning as that of~\eqref{eq:defWdd} and~\eqref{eq:boundWdd}, and is actually simpler since we can use the bound $1- \cos a \sim a^2/2 \ll a \sim \sin a$ for $a$ small in the term involving the gradient of the Green’s function. Therefore, we omit further technical details here.

\section{Proof of the lower bound} \label{sec:shortsec5}

The lower bound in Theorem \ref{th:mainth} follows from the asymptotics of the transversal two-point functions established by the authors in \cite{DW}, and a correlation inequality of Dunlop-Newman type. We first state the Dunlop-Newman inequality for the XY model on finite graphs. For any finite rooted graph $G = (V , E, \rho)$ and $\beta>0$, let $\mu_{\mathrm{XY}, \beta,G}$ and $\mu_{\mathrm{Vil}, \beta,G}$ be the Gibbs measure of the XY and Villain models in $G$ with inverse temperature $\beta$ and Dirichlet boundary condition at $\rho$. They are defined to be the probability distributions on the space $\Omega_G := \left\{ \theta : V \to (- \pi , \pi] \, : \, \theta(\rho) = 0 \right\}$ given by
\begin{equation*} 
        \mu_{ \mathrm{XY},\beta,G}(d \theta) := \frac{1}{Z_{\mathrm{XY},\beta,G}}  \exp \left( \beta \sum_{(x,y) \in E}  \cos \left( \theta_x -\theta_y \right) \right)\prod_{x \in V \setminus \rho} d\theta(x)
\end{equation*}
and
\begin{equation*}
    \mu_{ \mathrm{Vil}, \beta,G}(d \theta) := \frac{1}{Z_{\mathrm{Vil},\beta,G}}  \prod_{(x,y) \in E} v_\beta \left( \theta_x -\theta_y \right)  \prod_{x \in V \setminus \rho} d\theta(x).
\end{equation*}
We denote by $\langle \cdot \rangle_{\mathrm{XY}, \beta,G}$ and $\langle \cdot \rangle_{ \mathrm{Vil}, \beta,G}$ the expectation with respect to these measures and note that this formalism is sufficiently general to contain the Villain model on a box of $\mathbb{Z}^d$ with zero boundary condition introduced in~\eqref{e.finitedirichlet}.

\begin{theorem}[Theorem 13 of \cite{DN}]
\label{t.dunlopnewman}
Let $G$ be a finite rooted graph, then for any $\beta>0$ and any $x , y \in V$,
    \begin{multline*}
        \langle \sin \theta_x \sin \theta_y \rangle_{\mathrm{XY}, \beta,G, }^2
        \\ \leq
       \left( \langle \cos \theta_x \cos \theta_y \rangle_{\mathrm{XY}, \beta, G} - \langle \cos \theta_x \rangle_{\mathrm{XY}, \beta,G} \langle \cos \theta_y \rangle_{\mathrm{XY}, \beta,G} \right)
        \left( \langle \cos \theta_x \cos \theta_y \rangle_{\mathrm{XY}, \beta,G} + \langle \cos \theta_x \rangle_{\mathrm{XY},\beta,G} \langle \cos \theta_y \rangle_{\mathrm{XY}, \beta, G} \right).
    \end{multline*}
\end{theorem}

We also notice that the Villain model on a finite graph $G$ can be written as the distribution limit of the rescaled limit of the XY models on metric graphs $G_n$. Given a finite graph $G= (\mathcal{V}(G), \mathcal{E}(G))$, we replace each edge $e=(v,w) \in  \mathcal{E}(G)$ by $n-1$ new vertices denoted as $v(e,1), \ldots, v(e,n-1)$ (we also denote by $v(e,0)= v$ and $v(e,n) =w$), and a collection of $n$ new edges $\{v, v(e,1)\}, \ldots, \{v(e,n-1),w\}$. This defines a new graph $G_n = (\mathcal{V}(G_n), \mathcal{E}(G_n))$. 

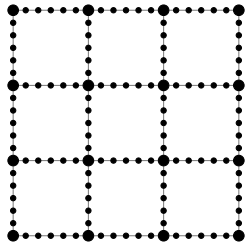
\begin{figure}[h]
\begin{center}
\begin{tikzpicture}
    \def\N{3} 
    \def\subdiv{6} 
    \def\step{1/\subdiv} 
    
    \foreach \x in {0,...,\N} {
        \foreach \y in {0,...,\N} {
            \ifnum \x<\N
                \foreach \i in {1,...,\subdiv} {
                    \draw[gray] (\x+\i*\step,\y) -- (\x+\i*\step-\step,\y);
                }
            \fi
            \ifnum \y<\N
                \foreach \i in {1,...,\subdiv} {
                    \draw[gray] (\x,\y+\i*\step) -- (\x,\y+\i*\step-\step);
                }
            \fi
        }
    }
    
    \foreach \x in {0,...,\N} {
        \foreach \y in {0,...,\N} {
            \filldraw[black] (\x,\y) circle (2pt);
        }
    }
    
    \foreach \x in {0,...,\N} {
        \foreach \y in {0,...,\N} {
            \ifnum \x<\N
                \foreach \i in {1,...,\subdiv} {
                    \filldraw[black] (\x+\i*\step,\y) circle (1pt);
                }
            \fi
            \ifnum \y<\N
                \foreach \i in {1,...,\subdiv} {
                    \filldraw[black] (\x,\y+\i*\step) circle (1pt);
                }
            \fi
        }
    }
    \end{tikzpicture}
    \end{center}
    \caption{The metric graph $G_6$, where $G$ is a $3 \times 3$ cube in $\Z^2$}

\end{figure}

The following result, obtained in \cite{NW} and \cite{AHPS}, shows that the correlation function on a rescaled sequence of the XY model on metric graphs converges to that of the Villain model.

\begin{theorem}
Let $G$ be a finite graph, $\beta>0$ and $n\in\N$, let $\beta_n = n\beta$ and $G_n$ be the metric graph obtained by $G$. Then we have 
    \begin{equation*}
      \lim_{n\to \infty}  \langle \cos \theta_x \cos \theta_y \rangle_{\mathrm{XY},\beta_n,G_n}
      =  \langle \cos \theta_x \cos \theta_y \rangle_{ \mathrm{Vil}, \beta,G}.
    \end{equation*}
    Likewise,
     \begin{equation*}
      \lim_{n\to \infty}  \langle \sin \theta_x \sin \theta_y \rangle_{\mathrm{XY}, \beta_n,G_n}
      =  \langle \sin \theta_x \sin \theta_y \rangle_{ \mathrm{Vil}, \beta,G}.
    \end{equation*}
\end{theorem}

Combined with Theorem \ref{t.dunlopnewman}, and the Ginibre inequality which gives the infinite volume limit of the Villain model, we obtain the corresponding Dunlop-Newman inequality for the Villain model in the thermodynamic limit. 
\begin{theorem}[Dunlop-Newman inequality for Villain model]
    For any $\beta>0$, we have
    \begin{multline*}
        \langle \sin \theta_x \sin \theta_y \rangle_{\mu_{\mathrm{Vil}, \beta}}^2
        \leq \\
       \left( \langle \cos \theta_x \cos \theta_y \rangle_{\mu_{\mathrm{Vil}, \beta}} - \langle \cos \theta_x \rangle_{\mu_{\mathrm{Vil}, \beta}} \langle \cos \theta_y \rangle_{\mu_{\mathrm{Vil}, \beta}} \right)
        \left( \langle \cos \theta_x \cos \theta_y \rangle_{\mu_{\mathrm{Vil}, \beta}} + \langle \cos \theta_x \rangle_{\mu_{\mathrm{Vil}, \beta}} \langle \cos \theta_y \rangle_{\mu_{\mathrm{Vil}, \beta}} \right).
    \end{multline*}
\end{theorem}
Together with the asymptotics of the transversal correlation $\langle \sin \theta_x \sin \theta_y \rangle_{\mu_{\mathrm{Vil}, \beta}}$ obtained in \eqref{e.2ptfirstfirst}, we obtain the desired lower bound 
\begin{equation*}
    \langle \cos \theta_x \cos \theta_y \rangle_{\mu_{\mathrm{Vil}, \beta}} - \langle \cos \theta_x \rangle_{\mu_{\mathrm{Vil}, \beta}} \langle \cos \theta_y \rangle_{\mu_{\mathrm{Vil}, \beta}}
    \ge
    \frac{c}{|x-y|^{2(d-2)}}.
\end{equation*}

\appendix

\section{Notation} \label{app.notation}

In this appendix, we show that the random variables introduced  Sections~\ref{sec:sec3.1} and~\ref{sec:sec3.2} are well-defined random variables (N.B. This is not necessarily obvious due to the sum over the charges $q \in \mathcal{Q}$). Specifically, we prove that:
\begin{itemize}
    \item The random variable $U_x$ is well-defined in $L^2(\Omega \, ; \, \R)$. While we will not prove it to minimise the technicality of this section, we note that the argument written below can be refined so as to show that the random variable $\exp (8U_x)$ is integrable (N.B. In fact one can show that the random variable $\exp ( p_\beta U_x)$ for a large exponent $p_\beta$ which depends on $\beta$, and tends to infinity as $\beta \to \infty$, is integrable).
    \item The sum over the charges in the definitions of the five other random variables (introduced in~\eqref{eq:def3.5}) converge absolutely for any value of $\varphi \in \Omega$. This implies that these random variables are well-defined (and in fact are bounded uniformly over $\varphi \in \Omega$).
\end{itemize}
Throughout this section, we will make use of the results of Appendix~\ref{App.sumonlattice} (specifically, the upper bound~\eqref{eq:truineqB1}).

\subsection{Showing that $U_{x}$ is a well-defined random variable in $L^2(\Omega \, ; \, \R)$.}

Given a large integer $L \in \N$, let us denote by
\begin{equation*}
    \mathcal{Q}_L := \left\{ q \in \mathcal{Q} \, : \, \supp q \subseteq  \Lambda_L\right\}.
\end{equation*}
We then let $U_L$ be the random variable
\begin{equation*}
     U_L(\varphi) := \sum_{q \in \mathcal{Q}_L} z(\beta , q)  \sin \left(2\pi(\di^* \varphi , n_q)\right)  \sin \left(2\pi(\nabla G_x , n_q)\right) \in \R.
\end{equation*}
Let us note that the sum on the right-hand side of the previous display is absolutely convergent for any realisation of the field $\varphi \in \Omega$ and that this is in fact a smooth function (depending on finitely many coordinates) of $\varphi \in \Omega$ (due to the same argument as in the proof of Lemma~\ref{lemma.lemma2.5}). We will then show the inequality
\begin{equation*} 
    \left\| U_{2L} - U_L \right\|_{L^2(\Omega \, ; \, \R)} \leq \frac{C}{L}.
\end{equation*}
This result implies that, for any pair of integers $k , l \in \N$ with $l \leq k$,
\begin{equation*}
    \left\| U_{2^k} - U_{2^l} \right\|_{L^2(\Omega \, ; \, \R)} \leq \frac{C}{2^l}.
\end{equation*}
The sequence of random variables $\left( U_{2^k} \right)_{k \in \N}$ is thus Cauchy in $L^2(\Omega \, ; \, \R)$; it thus converges in this space (and in fact almost surely) toward a random variable. This limiting random variable is our definition for $U_x$.

Applying the Helffer-Sj\"{o}strand representation formula (noting that $U_L$ is a smooth function of $\varphi$), we have that 
\begin{align*}
    \mathrm{Var}_{\mu_{\beta}} \left[ U_{2L} - U_L \right] & = \sum_{x \in \Zd} \left\langle \partial_x \left( U_{2L} - U_L \right)(\varphi) \cdot \mathcal{G}(x , \varphi) \right\rangle_{\mu_\beta} \\
    & = 2 \pi \sum_{q \in \mathcal{Q}_{2L} \setminus \mathcal{Q}_L} z(\beta , q)  \sin \left(2\pi(\nabla G , n_q)\right)  \left\langle \cos \left(2\pi(\di^* \varphi , n_q)\right)  \left( \sum_{x \in \Zd} \mathcal{G}(x, \varphi) \cdot q(x) \right) \right\rangle_{\mu_\beta} \\
    & = 2 \pi \sum_{q \in \mathcal{Q}_{2L} \setminus \mathcal{Q}_L} z(\beta , q)  \sin \left(2\pi(\nabla G , n_q)\right)  \left\langle \cos \left(2\pi(\di^* \varphi , n_q)\right)  \left(  \mathcal{G}(\cdot, \varphi) , q \right) \right\rangle_{\mu_\beta} \\
    & = 2 \pi \sum_{q \in \mathcal{Q}_{2L} \setminus \mathcal{Q}_L} z(\beta , q)  \sin \left(2\pi(\nabla G , n_q)\right)  \left\langle \cos \left(2\pi(\di^* \varphi , n_q)\right)  \left(  \di^* \mathcal{G}(\cdot, \varphi) , n_q \right) \right\rangle_{\mu_\beta}
\end{align*}
where 
\begin{align*}
    \mathcal{L} \mathcal{G}(x , \varphi) & =  \partial_x (U_{2L} - U_L) = 2 \pi \sum_{q \in \mathcal{Q}} z(\beta , q)  \sin \left(2\pi(\nabla G , n_q)\right) \cos \left(2\pi(\di^* \varphi , n_q)\right) q(x) \in \R^{d \choose 2}.
\end{align*}
Using the definition of the Green's matrix for the Helffer-Sj\"{o}strand equation, we have the identity
\begin{equation*}
    \mathcal{G}(x , \varphi) = 2 \pi \sum_{q \in \mathcal{Q}} z(\beta , q)  \sin \left(2\pi(\nabla G , n_q)\right)  (\di^*_2 \mathcal{G}_{\mathbf{f}_q}(x , \varphi; \cdot), n_q) \hspace{3mm} \mbox{with} \hspace{3mm} \mathbf{f}_q(\varphi) = \cos \left(2\pi(\di^* \varphi , n_q)\right).
\end{equation*}
This leads to the upper bound (which is uniform over $x \in \Zd$)
\begin{equation*}
    \left\| \di^* \mathcal{G}(x , \cdot) \right\|_{L^2(\Omega \, ; \, \R)} \leq \sum_{z \in \Lambda_{2L} \setminus \Lambda_L} \frac{C}{|z|_+^{d-1}} \frac{\left(\ln |x - z|_+\right)^{d+2}}{|x-z|_+^{d}} \leq \frac{C \left(\ln L \right)^{d+2}}{L^{d-1}}.
\end{equation*}
Combining the previous inequality with the Cauchy-Schwarz inequality and the same argument as in the proof of Lemma~\ref{lemma.lemma2.5}, we deduce that (denoting by $C_q$ a constant which may grow polynomially fast in $\left\| q \right\|_1$)
\begin{align*}
    \mathrm{Var}_{\mu_{\beta}} \left[ U_{2L} - U_L \right] \leq C \sum_{q \in \mathcal{Q}_{2L} \setminus \mathcal{Q}_L} |z(\beta,q)| \left| (\nabla G , n_q) \right|  \left| \left\langle \left( \di^* \mathcal{G}(\cdot, \varphi) , n_q \right) \right\rangle_{\mu_\beta} \right| & \leq  \sum_{q \in \mathcal{Q}_{2L} \setminus \mathcal{Q}_L} |z(\beta,q)| \frac{C_q}{L^{d-1}}  \frac{C_q (\ln L)^{d+2}}{L^{d-1}}  \\
    & \leq \frac{(\ln L)^{d+2}}{L^{2d-2}} \underset{\leq C L^d}{\underbrace{\sum_{q \in \mathcal{Q}_{2L} \setminus \mathcal{Q}_L} z(\beta,q) C_q}} \\
    & \leq \frac{C (\ln L)^{d+2}}{L^{d-2}}.
\end{align*}

\subsection{Estimating the $L^\infty(\Omega \, ; \, \R)$-norm of $U_{\cos , x}$}
Using the upper bounds $\left| \cos (a) \right| \leq 1$, $1 - \cos (a) \leq a^2/2$ for $a \in \R$, the observation $2d - 2 > d$ (since $d \geq 3$) and the same argument as in Lemma~\ref{lemma.lemma2.5}, we may write: for any $x \in \Zd$ and any $\varphi \in \Omega$,
\begin{align*}
    \left|  U_{\cos , x}(\varphi) \right|  \leq \sum_{q \in \mathcal{Q}} |z(\beta , q)| \left|\cos \left(2\pi(\di^* \varphi , n_q)\right)\right|  \left| \cos \left(2\pi(\nabla G , n_q)\right) - 1 \right| \leq C \sum_{q \in \mathcal{Q}} |z(\beta , q)| \left| (\nabla G , n_q) \right|^2
    & \leq   \sum_{q \in \mathcal{Q}} |z(\beta , q)| \frac{C_q}{|z_q|_+^{2d-2}} \\
    & \leq C \sum_{z \in \Zd} \frac{1}{|z|_+^{2d-2}} \\
    & \leq C.
\end{align*}

\subsection{Estimating the $L^\infty(\Omega \, ; \, \R)$-norm of $U_{\sin \cos , x}$}
Using the upper bounds $\left| \cos (a) \right| \leq 1$, $\left| \sin (a) \right| \leq |a|$, $1 - \cos (a) \leq a^2/2$ for $a \in \R$ and the same argument as in Lemma~\ref{lemma.lemma2.5}, we may write: for any $x \in \Zd$ and any $\varphi \in \Omega$,
\begin{align*}
    \left|  U_{\sin \cos , x}(\varphi) \right| & \leq \sum_{q \in \mathcal{Q}} |z(\beta , q)| \left|\cos \left(2\pi(\di^* \varphi , n_q)\right)\right| \left| \sin \left(2\pi(\nabla G_x , n_q)\right) \right|  \left| \cos \left(2\pi(\nabla G , n_q)\right) - 1 \right| \\
    & \leq C \sum_{q \in \mathcal{Q}} |z(\beta , q)| \left| (\nabla G_x , n_q) \right| \left| (\nabla G , n_q) \right|^2 \\
    & \leq  C \sum_{q \in \mathcal{Q}} |z(\beta , q)|\frac{C_q}{|x - z_q|_+^{d-1} |z_q|_+^{2d-2}} \\
    & \leq C \sum_{z \in \Zd} \frac{1}{|x - z|_+^{d-1}|z|_+^{2d-2}} \\
    & \leq \frac{C}{|x|_+^{d-1}}.
\end{align*}

\subsection{Estimating the $L^\infty(\Omega \, ; \, \R)$-norm of $U_{\cos \sin , x}$}
Using the same computation as for the term $U_{\cos \sin}$, we have that: for any $x \in \Zd$ and any $\varphi \in \Omega$,
\begin{align*}
    \left|  U_{\cos \sin , x}(\varphi) \right|  \leq C \sum_{q \in \mathcal{Q}} |z(\beta , q)| \left| (\nabla G_x , n_q) \right|^2 \left| (\nabla G , n_q) \right| \leq  C \sum_{q \in \mathcal{Q}} |z(\beta , q)|\frac{C_q}{|x - z_q|_+^{2d-2} |z_q|_+^{d-1}}
    & \leq C \sum_{z \in \Zd} \frac{1}{|x - z|_+^{2d-2}|z|_+^{d-1}} \\
    & \leq \frac{C}{|x|_+^{d-1}}.
\end{align*}

\subsection{Estimating the $L^\infty(\Omega \, ; \, \R)$-norm of $U_{\cos \cos , x}$}
Using the upper bounds $\left| \cos (a) \right| \leq 1$ and $1 - \cos (a) \leq a^2/2$ for $a \in \R$, we may write: for any $x \in \Zd$ and any $\varphi \in \Omega$,
\begin{align*}
    \left|  U_{\cos \cos , x} (\varphi) \right|  \leq C \sum_{q \in \mathcal{Q}} |z(\beta , q)| \left| (\nabla G_x , n_q) \right|^2 \left| (\nabla G , n_q) \right|^2
     \leq  C \sum_{q \in \mathcal{Q}} |z(\beta , q)|\frac{C_q}{|x - z_q|_+^{2d-2} |z_q|_+^{2d-2}}
    & \leq C \sum_{z \in \Zd} \frac{1}{|x - z|_+^{2d-2}|z|_+^{2d-2}} \\
    & \leq \frac{C}{|x|_+^{2d-2}}.
\end{align*}

\subsection{Estimating the $L^\infty(\Omega \, ; \, \R)$-norm and the variance of $U_{\sin \sin , x}$}
Using the upper bound $\left| \sin (a) \right| \leq |a|$ for $a \in \R$, we may write: for any $x \in \Zd$ and any $\varphi \in \Omega$,
\begin{align*}
    \left|  U_{\sin \sin , x} (\varphi) \right|  \leq C \sum_{q \in \mathcal{Q}} |z(\beta , q)| \left| (\nabla G_x , n_q) \right| \left| (\nabla G , n_q) \right| \leq  C \sum_{q \in \mathcal{Q}} |z(\beta , q)|\frac{C_q}{|x - z_q|_+^{d-1} |z_q|_+^{d-1}}
    & \leq C \sum_{z \in \Zd} \frac{1}{|x - z|_+^{d-1}|z|_+^{d-1}} \\
    & \leq \frac{C}{|x|_+^{d-2}}.
\end{align*}
The variance of the random variable $U_{\sin \sin , x}$ is more technical to estimate. We first collect the identity, for each $y \in \Zd$,
\begin{equation*}
  \partial_y U_{\sin \sin , x} (\varphi) = - \sum_{q \in \mathcal{Q}}2\pi z(\beta , q) \sin \left(2\pi(\di^* \varphi , n_q)\right) \sin \left(2\pi(\nabla G , q)\right) \sin \left( 2\pi(\nabla G_x , q)\right) q(y) \in \R^{d \choose 2}.
\end{equation*}
By applying the Helffer-Sj\"{o}strand representation formula, we have the identity
\begin{align*}
     \mathrm{Var}_{\mu_\beta} \left[ U_{\sin \sin , x} \right] & = \sum_{y \in \Zd} \left\langle \partial_y U_{\sin \sin , x} (\varphi)  \cdot \mathcal{G}(y , \varphi) \right\rangle_{\mu_\beta} \\
     & = 2\pi \sum_{q \in \mathcal{Q}}  z(\beta , q)  \sin \left( 2\pi(\nabla G , n_q)\right) \sin \left( 2\pi(\nabla G_x , q) \right)\left\langle  \sin \left( 2\pi(\di^*\varphi , n_q) \right) \left( \di^* \mathcal{G}(\cdot , \phi) , n_q \right) \right\rangle_{\mu_\beta}
\end{align*}
with
\begin{equation*}
    \mathcal{L} \mathcal{G}(z , \varphi) = 2 \pi \sum_{q \in \mathcal{Q}} z(\beta , q) \sin\left( 2\pi(\di^* \varphi , q) \right) \sin \left(2\pi(\nabla G , n_q) \right) \sin \left(2\pi(\nabla G_x , n_q)\right) q(z).
\end{equation*}
This function can be decomposed using the Green's matrix as follows
\begin{equation*}
     \mathcal{G}(z , \varphi) = 2 \pi \sum_{q \in \mathcal{Q}} z(\beta , q)  \sin \left(2\pi(\nabla G , n_q) \right) \sin \left(2\pi(\nabla G_x , n_q)\right) \left( \di^*_2 \mathcal{G}_{\mathbf{f}_q}(z , \varphi , \cdot ) ,  n_q \right)
\end{equation*}
with $\mathbf{f}_q(\varphi) = \sin\left( 2\pi(\di^* \varphi , q) \right)$. By Lemma~\ref{lemma.lemma2.5} and Proposition~\ref{prop.prop4.11chap4}, this leads to the upper bound
\begin{equation*}
    \left\| \di^* \mathcal{G} (z , \cdot) \right\|_{L^2 \left( \Omega \, ; \, \R^{d} \right)} \leq \sum_{z' \in \Zd} \frac{C}{|z'|_+^{d-1} |z' - x|_+^{d-1}} \frac{(\ln |z - z'|_+)^{d+2}}{|z' - z|_+^{d}}.
\end{equation*}
We thus obtain, using Lemma~\ref{lemma.lemma2.5} again together with the inequality~\eqref{ineqB4} of Appendix~\ref{App.sumonlattice}
\begin{equation*}
    \mathrm{Var}_{\mu_\beta} \left[ U_{\sin \sin , x} \right]  \leq C \sum_{z , z' \in \Zd} \frac{1}{|z|_+^{d-1} |z - x|_+^{d-1}}  \frac{1}{|z'|_+^{d-1} |z' - x|_+^{d-1}} \frac{(\ln |z - z'|_+)^{d+2}}{|z - z'|_+^{d}}  \leq \frac{C}{|x|^{2d-2}_+}.
\end{equation*}

\section{Annealed regularity estimates for the heat kernel} \label{app.CZreg}

This section contains the proof of the inequalities~\eqref{eq:decaygradheat} and~\eqref{eq:decaygradgradheat} of Proposition~\ref{prop.prop4.11chap4HK}. The result is restated below.

\begin{proposition}[$L^p$-annealed regularity] \label{prop:Lpannealedreg}
    For any exponent $p \in (1 , \infty),$ there exists an inverse temperature $\beta_1 := \beta_1(p , d) < \infty$ such that for any $\beta \geq \beta_1$ and any pair of vertices $x , y \in \Zd$,
    \begin{equation*}
        \left\| \di^*_1 P^{\cdot}(t , x  , y ) \right\|_{L^p\left(\Omega \, ; \, \R^{d \times {d \choose 2}} \right)} +  \left\| \di^*_2 P^{\cdot}(t , x  , y ) \right\|_{L^p\left(\Omega \, ; \, \R^{{d \choose 2} \times d} \right)} \leq \frac{C (\ln t_+)^{(d+2)}}{t^{\frac d2 + \frac{1}{2}}} \exp \left( - \frac{|x-y|}{C\sqrt{t_+}} \right)
    \end{equation*}
    and
    \begin{equation*}
        \left\| \di^*_1 \di^*_2 P^{\cdot}(t , x  , y ) \right\|_{L^p\left(\Omega \, ; \, \R^{d \times d } \right)} \leq \frac{C (\ln t_+)^{(d+2)}}{t^{\frac d2 + 1}} \exp \left( - \frac{|x-y|}{C\sqrt{t_+}} \right).
    \end{equation*}
\end{proposition}

\begin{remark}
Let us make two remarks about the previous statement:
\begin{itemize}
    \item More generally, the result holds for the gradients and mixed derivative of the heat kernel (instead on the codifferentials and mixed codifferential).
    \item To prove Theorem~\ref{th:mainth}, we need to choose the inverse temperature $\beta$ sufficiently large so that this result holds with the exponent $p = 8$.
    \item Since the Langevin dynamic is reversible, it is equivalent to prove the result for the dynamic or the time-reversed dynamic (and we will only prove it in the first case).
\end{itemize}
\end{remark}

The proof of Proposition~\ref{prop:Lpannealedreg} is divided into three parts. Section~\ref{sectprelimappB} is a preliminary section that introduces additional notation and auxiliary results. Section~\ref{sec:CZsec} contains the analytical core of the argument: exploiting the fact that the elliptic operator $\mathcal{L}_{\mathrm{spat}}^{\varphi\cdot}$ has small ellipticity contrast, we establish accurate estimates on the $L^p$-norm of the gradient of the heat kernel, for some exponent $p \in (1 , \infty)$ that tends to infinity as $\beta \to \infty$. This part of the argument is deterministic and relies on an adaptation of Calderón–Zygmund regularity theory. Finally, Section~\ref{sec:DDannealedregsec} introduces probabilistic elements: combining the regularity results from Section~\ref{sec:CZsec} with the stationarity of the Langevin dynamics—following the approach of Delmotte and Deuschel~\cite{DD05}—we complete the proof of Proposition~\ref{prop:Lpannealedreg}.

\subsection{Preliminaries} \label{sectprelimappB}

\subsubsection{Notation} 
Given an integer $L \in \N$ and $(t , x) \in (0 , \infty) \times \Zd$, we denote by $Q_L$ and $Q_L(t , x)$ the parabolic cylinders
\begin{equation*}
    Q_L := (- L^2 , 0) \times \Lambda_L \hspace{5mm} \mbox{and} \hspace{5mm} Q_L(t , x) = ( t - L^2 , t) \times ( x + \Lambda_L).
\end{equation*} 
We introduce the (discrete) parabolic boundary of $Q_L(t , x)$ according to
$$\partial_{\mathrm{par}} Q_L(t , x) := \left( \left\{ t - L^2 \right\} \times (x + \Lambda_L) \right) \cup \left( (t - L^2 , t) \times \partial (x + \Lambda_L)  \right)$$ 
Given a function $f : Q_L(t , x) \to \R$ and an exponent $p \in [1 , \infty)$, we introduce its average $L^p$-norm and $L^\infty$-norm according to the formulae
\begin{equation*}
    \left\| f \right\|_{\underline{L}^p(Q_L(t , x))}^p := \frac{1}{L^2} \int_{t - L^2}^t \frac{1}{\left| \Lambda_L \right|} \sum_{y \in (x + \Lambda_L)} \left| f(s , y) \right|^p \, ds \hspace{3mm} \mbox{and} \hspace{3mm} \left\| f \right\|_{L^\infty(Q_L(t , x))} := \sup_{(s , y) \in Q_L(t , x)} \left| f(s , y) \right|.
\end{equation*}
We next introduce the notation: for $C_0 \in (1 , \infty)$, $t \in (0 , \infty)$ and $x \in \Zd$, we let
    \begin{equation*}
        \Phi_{C_0} (t , x) := \frac{1}{t_+^{d/2}} \exp \left( - \frac{|x|}{C_0 \sqrt{t_+}} \right).
    \end{equation*}
Let us note that there exists a constant $C < \infty$ depending only on $C_0$ such that, for any $t \in (0 , \infty)$
\begin{equation} \label{ineq:almostprobab}
        \sum_{x \in \Zd} \Phi_{C_0} (t , x) \leq C.
\end{equation}

\subsubsection{Two preliminary results}

In this section, we collect two preliminary results which will be used in the proof of Proposition~\ref{prop:Lpannealedreg}. The first one is an optimal $L^2$-estimate for discrete gradient of the heat kernel.

\begin{proposition}[$L^2$-regularity estimate for the heat kernel] \label{prop:Caccioppoli}
    There exists a constant $C := C(d , \beta) < \infty$ such that for any $(t , x) \in (0 , \infty)$ and any realisation of the Langevin dynamic $(\varphi_t)_{t \geq 0}$,
    \begin{equation*}
        \left\| \nabla P^{\cdot}(\cdot , \cdot , 0) \right\|_{\underline{L}^2(Q_{\sqrt{t}/2}(t , x))} \leq \frac{C}{t_+^{\frac{d}{4} + \frac12}}\exp \left( - \frac{|x|}{C \sqrt{t_+} }\right).
    \end{equation*}
\end{proposition}

\begin{remark}
This estimate is obtained by combining the Caccioppoli inequality (see~\cite[Chapter 5, Section 1]{DW} but the argument is the standard one) and the upper bound on the heat kernel stated in Proposition~\ref{prop.prop4.11chap4HK}.
\end{remark}

The second result is a general statement asserting that if an $L^2$ function can be well-approximated on all scales by a bounded function, then it necessarily lies in some $L^p$-space, with a corresponding quantitative bound on its $L^p$-norm.

\begin{proposition} \label{prop:CalderonZygmundLpestimate}
For each $p \in  (2, \infty)$ and $A \geq 1$, there exists $\delta_0(p, A, d) > 0$
and $C(p, A, d) < \infty$ such that the following holds for every $\delta \in (0, \delta_0]$. Let $f : [0 , \infty) \times \Zd \to \R$ and a constant $K \in (0, \infty)$ be such that, for any $(s , y) \in [0 , \infty) \times \Zd$ and any $r \in \N$ such that $8r^2 \leq s$, there exists a function $f_{(s , y) , r} :  Q_r(s , y) \to \R $ satisfying
\begin{equation*}
     \left\| f - f_{(s , y) , r} \right\|_{\underline{L}^2\left( Q_r(s , y) \right)}  \leq \delta \left\| f \right\|_{\underline{L}^2\left( Q_{2r}(s , y) \right)} + K
\end{equation*}
and 
\begin{equation*}
    \left\|  f_{(s , y) , r} \right\|_{L^\infty \left( Q_r(s , y) \right)} \leq A \left\| f \right\|_{\underline{L}^2\left( Q_{2r}(s , y) \right)}.
\end{equation*}
Then for any $(t , x) \in (0 , \infty) \times \Zd$ and any $R \in \N$ with $8R^2 \leq t$,
\begin{equation*}
    \left\| f \right\|_{\underline{L}^p(Q_{R}(t , x))} \leq C  \left\| f \right\|_{\underline{L}^2(Q_{2R}(t , x))} + CK.
\end{equation*}
\end{proposition}

\begin{remark}
The proof of this result (in the continuous setting and for functions depending only on the spatial variable) can be found in~\cite[Lemma 7.2]{AKM} (by considering the function $g$ there to be the constant function equal to $K$). The extension to the discrete and parabolic setting (considered above) is mostly notational (and omitted here to reduce the technical complexity of this appendix).
\end{remark}

\subsection{Proof of the Calder\'on-Zygmund regularity} \label{sec:CZsec}

In this section, we establish a Calderón–Zygmund-type regularity estimate for the heat kernel $P^{\varphi_\cdot}$. Specifically, we derive an almost optimal bound on its $L^p$-norm for some exponent $p \in (1 , \infty)$, which depends on $\beta$ and tends to infinity as $\beta \to \infty$. At a high level, we rely on the observation that, when $\beta$ is large, the elliptic operator $\mathcal{L}_{\mathrm{spat}}^{\varphi_\cdot}$ has a small ellipticity contrast. This implies that the heat kernel is well-approximated in $L^2$ on every scale scales by a solution to a heat equation with strong regularity properties. This allows us to apply Proposition~\ref{prop:CalderonZygmundLpestimate} and obtain an (almost) sharp estimate on the $L^p$-norm of the heat kernel. We mention that the suboptimal logarithmic correction appears in this step of the proof and is caused by the infinite range of the elliptic operator $\mathcal{L}_{\mathrm{spat}}^{\varphi_\cdot}$.

\begin{proposition}[Calder\'{o}n-Zygmund regularity for the heat kernel] \label{prop:CalderonZygmund}
For any exponent $p \in (1 , \infty),$ there exists an inverse temperature $\beta_1 := \beta_1(p , d) < \infty$ such that for any $\beta \geq \beta_1$, there exists a constant $C := C(d , \beta , p) < \infty$ such that for any $t \in (0, \infty)$, and any realisation of the Langevin dynamic $(\varphi_t)_{t \geq 0}$, one has the inequality
\begin{equation*}
    \int_{t/4}^{t} \sum_{y \in \Zd} \left| \nabla P^{\varphi_\cdot} (s , y , 0 )  \right|^p e^{\frac{|y|}{C \sqrt{t_+}}} \, ds \leq C (\ln t_+)^{\frac{p(d+2)}{2}} t_+^{\frac{d+2}{2}  - \frac{p}{2} - \frac{pd}{2}}.
\end{equation*}
\end{proposition}

\begin{proof}[Proof of Proposition~\ref{prop:CalderonZygmund}]
We fix a realisation of the Langevin dynamic $(\varphi_t)_{t \geq 0}$ (N.B. this argument is deterministic and works for any realisation of the dynamic) and split the proof into several steps:
\begin{itemize}
\item \textbf{Step 1.} We use that the operator $\mathcal{L}^{\varphi}_{\mathrm{spat}}$ is a perturbation of the Laplacian to prove that the heat kernel $P^{\varphi_\cdot}$ is well-approximated by a solution $\bar u$ of the heat equation:
\begin{equation} \label{eq:caloricheateq}
\partial_t \bar u - \frac{1}{2\beta} \Delta \bar u = 0.
\end{equation}
More specifically, we prove the following result:
for each $\delta > 0$, there exists an inverse temperature $\beta_1(d , \delta)$ such that for each $\beta \geq \beta_1$, each $r \in \N$ and each $(t , x) \in (0 , \infty) \times \Zd$ with $8r^2 \leq t$, there exists a function $\bar u$ which solves the heat equation~\eqref{eq:caloricheateq} in the parabolic cylinder $Q_{2r}(t , x)$ and satisfies
\begin{equation} \label{ineq:wellapproxcalderon}
    \left\| \nabla P^{\varphi_\cdot} - \nabla \bar u \right\|_{\underline{L}^2 (Q_r(t , x))} \leq \delta \left\| \nabla P^{\varphi_\cdot} \right\|_{\underline{L}^2(Q_{2r}(t , x)) } + e^{- c \left( \ln \beta \right) r} \exp \left( - \frac{|x|}{C \sqrt{t_+}}\right)
\end{equation}
together with the estimate
\begin{equation} \label{eq:Sh95616}
   \left\| \nabla \bar u \right\|_{L^\infty(Q_r(t , x))} \leq C \left\| \nabla P^{\varphi_\cdot} \right\|_{\underline{L}^2(Q_{2r}(t , x))}.
\end{equation}
\item \textbf{Step 2.} We apply Proposition~\ref{prop:CalderonZygmundLpestimate} and the $L^2$-regularity estimate for the heat kernel (Proposition~\ref{prop:Caccioppoli}) to deduce that: for any exponent $p \geq 1$, there exists an inverse temperature $\beta_1 := \beta_1(p , d) < \infty$ such that for any $\beta \geq \beta_1$, any $t \in (0, \infty)$,
\begin{equation} \label{eq:mainresulstep2CZ}
    \left\| \nabla P^{\varphi_\cdot} (\cdot , \cdot , 0)  \right\|_{\underline{L}^p(Q_{\sqrt{t}/4} (t , x))} \leq  \frac{C(\ln t_+)^{\frac{d+2}{2}}}{t_+^{\frac{1}{2} + \frac{d}{2}}}  \exp \left( - \frac{|x|}{C \sqrt{t_+}}\right).
\end{equation}
\item \textbf{Step 3.} We post-process the result of Step 2 to obtain the estimate: there exists a constant $C_0 := C_0(d , \beta,p) < \infty$ such that
\begin{equation*}
    \int_{t/4}^{t} \sum_{y \in \Zd} \left| \nabla P^{\varphi_\cdot} (s , y , 0 )  \right|^p e^{\frac{|y|}{C_0 \sqrt{t_+}}} \, ds \leq C_0 (\ln t_+)^{\frac{p(d+2)}{2}} t_+^{\frac{d+2}{2}  - \frac{p}{2}- \frac{pd}{2}}.
\end{equation*}
\end{itemize}

\textit{Step 1.} Let us fix $(t , x) \in (0 , \infty) \times \Zd$ and $r \in \N$ such that $8r^2 \leq t$. For any index $i , j \in \left\{1 , \ldots, {d \choose 2} \right\}$, we then let $\bar u_{i j} : Q_{2r}(t , x) \to \R$ be the solution of the parabolic equation
\begin{equation} \label{eq:defbaruheat}
    \left\{ \begin{aligned}
    \partial_t \bar u_{i j} - \frac{1}{2\beta}\Delta \bar u_{i j} & = 0 &~~\mbox{in} ~~& Q_{2r}(t , x), \\
    \bar u_{i j} & = P_{i j}^{\varphi_\cdot} \left( \cdot , \cdot , 0\right) &~~\mbox{in} ~~&  \partial_{\mathrm{par}} Q_{2r}(t , x),
    \end{aligned} \right.
\end{equation}
and set $u := \left( u_{ij} \right)_{1 \leq i , j \leq {d \choose 2}}$. An energy estimate for the solutions of the heat equation implies the upper bound
\begin{equation*}
    \left\| \nabla \bar u \right\|_{L^2 \left( Q_{2r}(t , x) \right)} \leq C \left\| \nabla P^{\varphi_\cdot} (\cdot , \cdot , 0) \right\|_{L^2 \left( Q_{2r}(t , x) \right)}.
\end{equation*}
Additionally, as a solution of the heat equation, the function $\bar u $ possesses good regularity properties, in particular we have the $L^\infty$-$L^2$ estimate on the gradient of $\bar u$
\begin{equation*}
    \left\| \nabla \bar u \right\|_{L^\infty \left( Q_{r}(t , x) \right)} \leq C \left\| \nabla \bar u \right\|_{\underline{L}^2 \left( Q_{2r}(t , x) \right)} \leq C \left\| \nabla P^{\varphi_\cdot} (\cdot , \cdot , 0) \right\|_{\underline{L}^2 \left( Q_{2r}(t , x) \right)}.
\end{equation*}
We then note that the function $w := P - u$ solves the parabolic systems of equations
\begin{equation*}
    \left\{ \begin{aligned}
    \partial_t w - \frac{1}{2\beta} \Delta w & = \frac{1}{2\beta} \Delta P^{\varphi_\cdot}(\cdot , \cdot , 0) - \mathcal{L}^{\varphi_\cdot}_{\mathrm{spat}} P^{\varphi_\cdot}(\cdot , \cdot , 0) &~~\mbox{in} ~~& Q_{2r}(t , x), \\
    w & = 0  &~~\mbox{in} ~~&  \partial_{\mathrm{par}} Q_{2r}(t , x).
    \end{aligned} \right.
\end{equation*}
We may then use an energy estimate for the heat equation and use that the operator
$$
\mathcal{L}^{\varphi_\cdot}_{\mathrm{spat}} P^{\varphi_\cdot}(\cdot , \cdot , 0) - \frac{1}{2\beta} \Delta P^{\varphi_\cdot}(\cdot , \cdot , 0) =  \sum_{n \geq 1} \frac{1}{2\beta} \frac{1}{\beta^{n/2}} \Delta^{n+1} P^{\varphi_\cdot} + \sum_{q \in \mathcal{Q}} \mathbf{a}_q(\varphi_\cdot) (P^{\varphi_\cdot} (\cdot , \cdot , 0) , q) q
$$
is a small perturbation of $\frac{1}{2\beta} \Delta$ (since $\beta \gg 1$) to conclude that (N.B. in the following identity we extend the function $w$ to $(t-4r^2 , t) \times \Zd$ by setting $w(s , y) = 0$ for $y \in \Zd \setminus (x + \Lambda_{2r})$)
\begin{multline*}
    \sum_{y \in \Zd} w(t , y)^2 + \frac{1}{2\beta} \left\| \nabla w \right\|_{L^2(Q_{2r}(t , x))}^2 \\ 
    = \int_{t - 4r^2}^t   \sum_{n \geq 1} \frac{1}{2\beta} \frac{1}{\beta^{n/2}} ( \Delta^{n+1} P^{\varphi_\cdot}(s , \cdot) , w(s , \cdot)) + \sum_{q \in \mathcal{Q}} \mathbf{a}_q(\varphi_s) (P^{\varphi_\cdot}(s , \cdot) , q) (w(s , \cdot) , q) \, ds.
\end{multline*}
By performing integrations by parts (and using the identity~\eqref{identity:dandd*}), we see that the right-hand side is a linear function of the gradient of $w$. By applying the Cauchy-Schwarz inequality together with a computation similar to the one performed in the proof of Lemma~\ref{lemma.lemma2.5}, we obtain the inequality
\begin{equation*}
     \frac{1}{2\beta} \left\| \nabla w \right\|_{L^2(Q_{2r}(t , x))}^2 \leq \frac{C}{\beta \sqrt{\beta }} \left\| \nabla P^{\varphi_\cdot}(\cdot , \cdot , 0) \right\|_{L^2(Q_{2r}(t , x))}^2 + \frac{1}{\beta} \int_{t - 4r^2}^t \sum_{y \in \Zd \setminus (x + \Lambda_{2r})} e^{-c \left( \ln \beta \right) |x - y|} \left| \nabla P^{\varphi_\cdot}(s , y , 0) \right|^2 \, ds.
\end{equation*}
We then simplify the second term on the right-hand side: by using the definition of the discrete gradient (which implies the inequality $\left| \nabla P^{\varphi_\cdot}(s , y , 0) \right| \leq \sum_{y' \sim y} \left| P^{\varphi_\cdot}(s , y' , 0) \right| + \left| P^{\varphi_\cdot}(s , y , 0) \right|$) together with the upper bound on the heat kernel stated in Proposition~\ref{prop.prop4.11chap4HK}, we deduce that
\begin{align*}
    \int_{t - 4r^2}^t \sum_{y \in \Zd \setminus (x + \Lambda_{2r})} e^{-c \left( \ln \beta \right) |x - y|} \left| \nabla P^{\varphi_\cdot}(s , y , 0) \right|^2 \, ds & \leq C  \int_{t - 4r^2}^t \sum_{y \in \Zd \setminus (x + \Lambda_{2r})} e^{-c \left( \ln \beta \right) |x - y|} \left| P^{\varphi_\cdot}(s , y , 0) \right|^2 \, ds \\
    & \leq  C \int_{t - 4r^2}^t \sum_{y \in \Zd \setminus (x + \Lambda_{2r})} e^{-c \left( \ln \beta \right) |x - y|}  s_+^{- d} \exp \left( - \frac{|y|}{C \sqrt{s_+}} \right) \, ds.
\end{align*}
Using the assumption $8r^2 \leq t$, we have $s \geq t/2$ for any time $s$ in the interval of integration $(t - 4r^2 , t)$. Combining this observation with the (trivial) upper bound $s_+ \geq 1$, we thus deduce that
\begin{align*}
    \int_{t - 4r^2}^t \sum_{y \in \Zd \setminus (x + \Lambda_{2r})} e^{-c \left( \ln \beta \right) |x - y|} \left| \nabla P^{\varphi_\cdot}(s , y , 0) \right|^2 \, ds & \leq C \exp \left( - \frac{|x|}{C \sqrt{t_+}} \right)  \underset{\leq e^{- c (\ln \beta) r}}{\underbrace{\left( \int_{t - 4r^2}^t \sum_{y \in \Zd \setminus (x + \Lambda_{2r})} e^{-c \left( \ln \beta \right) |x - y|} \right)}} \\
    & \leq C  e^{-c (\ln \beta) r} \exp \left( - \frac{|x|}{C \sqrt{t_+}} \right).
\end{align*}
A combination of the three previous inequalities implies that
\begin{equation*}
    \left\| \nabla w \right\|_{L^2(Q_{2r}(t , x))}^2 \leq \frac{C}{\sqrt{\beta }} \left\| \nabla P^{\varphi_\cdot}(\cdot , \cdot , 0) \right\|_{L^2(Q_{2r}(t , x))}^2 + C  e^{-c (\ln \beta) r} \exp \left( - \frac{|x|}{C \sqrt{t_+}} \right).
\end{equation*}
For any $\delta \in (0 , 1)$, we may thus select the inverse temperature $\beta \in (1 , \infty)$ sufficiently large so that $C /\sqrt{\beta} \leq \delta^2$. This implies
\begin{equation*}
    \left\| \nabla w \right\|_{L^2(Q_{2r}(t , x))}^2 \leq \delta^2 \left\| \nabla P^{\varphi_\cdot}(\cdot , \cdot , 0) \right\|_{L^2(Q_{2r}(t , x))}^2 + C e^{-c (\ln \beta) r}  \exp \left( - \frac{|x|}{C \sqrt{t_+}} \right).
\end{equation*}
We then divide by the volume of the parabolic cylinder $Q_{2r}(t , x)$ and take the square-root on both sides of this inequality (and use $\sqrt{a + b} \leq \sqrt{a} + \sqrt{b}$ for $a , b \geq 0$) to complete the proof of Step 1.

\medskip

\textit{Step 2.} Let us fix an exponent $p \in (1 , \infty)$, let $\delta_0$ be the parameter associated with the exponent $p$ given by Proposition~\ref{prop:Lpannealedreg}. We then select an inverse temperature $\beta$ sufficiently large so that the inequality~\eqref{ineq:wellapproxcalderon} holds with $\delta = \delta_0$. In order to treat the error term coming from the infinite range of the elliptic operator (i.e., the second term on the right-hand side of~\eqref{ineq:wellapproxcalderon}), we introduce a coarsening of the heat kernel on a scale of size $(\ln t_+)$. To be more precise, we fix a time $t \in (1 , \infty)$, a vertex $x \in \Zd$ and partition the parabolic cylinder $Q_{\sqrt{t}/2} (t , x)$ into cylinders of size (approximately) $(\ln \, t_+)$ and denote this partition by $\mathcal{P}$, e.g.,
\begin{equation*}
    \mathcal{P} := \left\{ Q_{ \lfloor \ln t_+ \rfloor}(s' , y') \hspace{1mm} : \hspace{1mm} (s' , y') \in \left( \lfloor \ln t_+ \rfloor^2 \N \times \lfloor \ln t_+ \rfloor \Zd \right) \cap Q_{\sqrt{t}/2} (t , x) \right\}.
\end{equation*}
We then let $f$ be the real-valued function defined on $Q_{\sqrt{t}/2} (t , x)$ which is constant equal to the average $L^2$ norm of the heat kernel $P^{\varphi_\cdot}$ on the cylinders of $\mathcal{P}$, i.e.,
\begin{equation*}
    f(s, y) = \left\| \nabla P^{\varphi_\cdot} (\cdot , \cdot , 0)\right\|_{\underline{L}^2\left( Q\right)} \hspace{3mm} \mbox{for} \hspace{3mm} (s , y) \in Q \hspace{2mm} \mbox{with} \hspace{2mm} Q \in \mathcal{P}.
\end{equation*}
We next show that the function $f$ satisfies the assumptions of Proposition~\ref{prop:CalderonZygmundLpestimate}. Specifically, we fix $(s , y) \in [0 , \infty) \times \Zd$, $r \in \N$ such that $8r^2 \leq s$, set
\begin{equation} \label{eq:defofK}
    K:= e^{-c (\ln \beta) (\ln t_+)} \exp \left( - \frac{|y|}{C \sqrt{t_+}} \right).
\end{equation}
Let us note that, when $\beta$ is large, the first term in the definition of $K$ decays like a polynomial in $t$ with a large exponent and is thus smaller than all the other quantities we wish to estimate. We then distinguish two cases.

\medskip

 \textit{\underline{Case 1:}} For $r \geq \ln t_+$, then we set 
    $$f_{(s , y) , r}(s', y') := \left\| \nabla \bar u \right\|_{\underline{L}^2 \left(Q_{\lfloor \ln t_+ \rfloor}(s'' , y'')\right)} \hspace{3mm} \mbox{for} \hspace{3mm}  (s' , y') \in Q_{\lfloor \ln t_+ \rfloor}(s'' , y'') \in \mathcal{P},$$
    where $\bar u$ is the function introduced in Step 1 (see~\eqref{eq:defbaruheat}). By~\eqref{ineq:wellapproxcalderon} (and the triangle inequality), we have the inequality
    \begin{align*}
        \left\| f - f_{(s , y) , r} \right\|_{\underline{L}^2\left( Q_r(s , y) \right)}  \leq \delta \left\| f \right\|_{\underline{L}^2\left( Q_{2r}(s , y) \right)} + C e^{-c (\ln \beta) r}  \exp \left( - \frac{|y|}{C \sqrt{t_+}} \right)  \leq \delta \left\| f \right\|_{\underline{L}^2\left( Q_{2r}(s , y) \right)} + K,
    \end{align*}
    and by~\eqref{eq:Sh95616}, we have the inequality
    \begin{equation*}
        \left\|  f_{(s , y) , r} \right\|_{L^\infty \left( Q_r(s , y) \right)} \leq C \left\| f \right\|_{\underline{L}^2\left( Q_{2r}(s , y) \right)}.
    \end{equation*}
    The assumptions of Proposition~\ref{prop:CalderonZygmundLpestimate} is thus verified in that case.
    
    \smallskip
    
\textit{\underline{Case 2:}} For $r \leq \ln t_+$, we note that, since the functions $f$ is constant on the parabolic cylinders of the partition $\mathcal{P}$ and we are considering here scales which are smaller than the one of the partition $\mathcal{P}$, the function $f$ is (by construction) regular on these small scales and satisfies
    \begin{equation*}
        \left\|  f \right\|_{L^\infty \left( Q_r(s , y) \right)} \leq C \left\| f \right\|_{\underline{L}^2\left( Q_{2r}(s , y) \right)}.
    \end{equation*}
We may thus set $ f_{(s , y) , r} := f$ to verify the assumptions of Proposition~\ref{prop:CalderonZygmundLpestimate}.
\medskip
We may thus apply Proposition~\ref{prop:CalderonZygmundLpestimate} (with $R = \sqrt{t} / 4$) to deduce that, for any exponent $p \in (1 , \infty)$ if the inverse temperature $\beta$ is chosen sufficiently large,
\begin{equation} \label{ineq:SH1201}
    \left\| f \right\|_{\underline{L}^p(Q_{\sqrt{t}/4}(t , x))} \leq C  \left\| f \right\|_{\underline{L}^2(Q_{\sqrt{t}/2}(t , x))} + CK.
\end{equation}
We finally remove the coarsening on both sides of this inequality. For the left-hand side, we have by the definition of the function $f$ and by Proposition~\ref{prop:Caccioppoli}
\begin{equation} \label{eq:ineqB8Sh}
    \left\| f \right\|_{\underline{L}^2(Q_{\sqrt{t}/2}(t , x))}  \leq \left\| \nabla P^{\varphi_\cdot} \right\|_{\underline{L}^2(Q_{\sqrt{t}/2}(t , x))} \leq \frac{C}{t_+^{\frac{d}{2} + \frac{1}{2}}} \exp \left( - \frac{|x|}{C \sqrt{t_+} }\right).
\end{equation}
For the left-hand side of the inequality~\eqref{ineq:SH1201}, we first note that, by using the upper bound~\eqref{eq:PisLipschitz} on the derivative in time of the heat kernel and the Sobolev injection $H^1 \hookrightarrow L^\infty$ in dimension $1$, we have, for any $(s , y) \in Q_{\sqrt{t}/4}(t , x)$,
\begin{align*}
    \left| \nabla P^{\varphi_\cdot}(s , y , 0) \right|^2 & \leq C \int_{s-1}^s \left| \nabla P^{\varphi_\cdot}(s' , y, 0) \right|^2 \, ds' + C \int_{s-1}^s \sum_{z \in \Zd} e^{- c (\ln \beta) |z - y|} \left| \nabla P^{\varphi_\cdot}(s' , z, 0) \right|^2 \, ds' \\
    & \leq C  \int_{s-1}^s \sum_{z \in \Zd} e^{- c (\ln \beta) |z - y|} \left| \nabla P^{\varphi_\cdot}(s' , z, 0) \right|^2 \, ds'.
\end{align*}
We next split the right-hand side into two terms which can be estimated separately (and recall that the constant $K$ was introduced in~\eqref{eq:defofK})
\begin{align*}
    \left| \nabla P^{\varphi_\cdot}(s , y , 0) \right|^2 & \leq C  \int_{s-1}^s \sum_{z \in y + \Lambda_{\lfloor \ln t_+ \rfloor}} \underset{\leq 1}{\underbrace{e^{- c (\ln \beta) |z - y|}}} \left| \nabla P^{\varphi_\cdot}(s' , z, 0) \right|^2 \, ds' \\
    & \qquad + C \underset{\leq C K}{\underbrace{\int_{s-1}^s \sum_{z \in \Zd \setminus (y + \Lambda_{\lfloor \ln t_+ \rfloor})} e^{- c (\ln \beta) |z - y|} \left| \nabla P^{\varphi_\cdot}(s' , z, 0) \right|^2 \, ds'}} \\
    & \leq C  \int_{s-1}^s \sum_{z \in y + \Lambda_{\lfloor \ln t_+ \rfloor}}  \left| \nabla P^{\varphi_\cdot}(s' , z, 0) \right|^2 \, ds' + C K.
\end{align*}
We then increase the interval of integration to write
\begin{align*}
    \left| \nabla P^{\varphi_\cdot}(s , y , 0) \right|^2 & \leq C \int_{s-\lfloor \ln t_+ \rfloor^2}^s \sum_{z \in y + \Lambda_{\lfloor \ln t_+ \rfloor}}  \left| \nabla P^{\varphi_\cdot}(s' , z, 0) \right|^2 \, ds' + C K \\
    & = C \left\| \nabla P^{\varphi_\cdot}(\cdot , \cdot, 0)  \right\|_{L^2(Q_{\lfloor \ln t_+ \rfloor}(s , y))}^2 + C K\\
    & \leq C \left( \ln t_+ \right)^{d+2} \left\| \nabla P^{\varphi_\cdot}(\cdot , \cdot , 0) \right\|_{\underline{L}^2(Q_{\lfloor \ln t_+ \rfloor}(s , y))}^2 + C K.
\end{align*}
From this inequality, we may deduce that
\begin{equation*}
    \left\| \nabla P^{\varphi_\cdot}(\cdot , \cdot , 0) \right\|_{\underline{L}^p(Q_{\sqrt{t}/4}(t , x))} \leq C (\ln t_+)^{\frac{d+2}{2}} \left\| f \right\|_{\underline{L}^p(Q_{\sqrt{t}/4}(t , x))}+ C K .
\end{equation*}
Combining the previous inequality with~\eqref{ineq:SH1201} and~\eqref{eq:ineqB8Sh}, we deduce that
\begin{equation*}
    \left\| \nabla P^{\varphi_\cdot}(\cdot , \cdot , 0) \right\|_{\underline{L}^p(Q_{\sqrt{t}/4}(t , x))} \leq \frac{C(\ln t_+)^{\frac{d+2}{2}}}{t_+^{\frac{1}{2} + \frac{d}{2}}} \exp \left( - \frac{|x|}{C \sqrt{t_+} }\right) + CK.
\end{equation*}
We finally note that, if the inverse temperature $\beta$ is chosen sufficiently large, then the constant $K$ is smaller than the first term on the left-hand side, and we have obtained
\begin{equation*}
\left\| \nabla P^{\varphi_\cdot}(\cdot , \cdot , 0) \right\|_{\underline{L}^p(Q_{\sqrt{t}/4}(t , x))} \leq \frac{C(\ln t_+)^{\frac{d+2}{2}}}{t_+^{\frac{1}{2} + \frac{d}{2}}} \exp \left( - \frac{|x|}{C \sqrt{t_+} }\right).
\end{equation*}

\medskip

\textit{Step 3.} In this step, we post-process the result of Step 2. Specifically, we set $C_1 := 2C$ where $C$ is the constant on the right-hand side of~\eqref{eq:mainresulstep2CZ}. We then partition the space $(t/4 , t) \times \Zd$ by setting (N.B. we assume below for simplicity that $\sqrt{t}/4$ is an integer)
\begin{equation*}
    \mathcal{Z} :=\left( \frac{t}{16}  \N \cap \left[ \frac t4 , t \right] \right) \times \frac{\sqrt{t}}{4}  \Zd  ~~\mbox{and}~~ \mathcal{P}_1 := \left\{ Q_{\sqrt{t}/4} (s , y) \, : \, (s , y) \in \mathcal{Z} \right\}.
\end{equation*}
We then estimate
\begin{align*}
     \int_{t/4}^{t} \sum_{y \in \Zd} \left| \nabla P^{\varphi_\cdot} (s , y , 0 )  \right|^p e^{\frac{|y|}{C_1 \sqrt{t_+}}} \, ds & \leq C \sum_{(s , y) \in \mathcal{Z}} \left\| \nabla P^{\varphi_\cdot} (\cdot , \cdot , 0 )  \right\|_{L^p \left( Q_{\sqrt{t}/4} (s , y) \right)}^p e^{\frac{|y|}{C_1 \sqrt{t_+}}} \\
     & \leq C  t_+^{\frac{d+2}{2}} \sum_{(s , y) \in  \mathcal{Z}} \left\| \nabla P^{\varphi_\cdot} (\cdot , \cdot , 0 )  \right\|_{\underline{L}^p \left( Q_{\sqrt{t}/4} (s , y) \right)}^p e^{\frac{|y|}{C_1 \sqrt{t_+}}} \\
     & \leq  C  t_+^{\frac{d+2}{2}} \sum_{(s , y) \in  \mathcal{Z}}\frac{(\ln t_+)^{\frac{p(d+2)}{2}}}{t_+^{\frac{p}{2} + \frac{pd}{2}}}  \exp \left( - \frac{|y|}{C \sqrt{t_+}} + \frac{|y|}{C_1 \sqrt{t_+}} \right).
\end{align*}
Using the definition $C_1 := 2C$, we eventually obtain
\begin{align*}
    \int_{t/4}^{t} \sum_{y \in \Zd} \left| \nabla P^{\varphi_\cdot} (s , y , 0 )  \right|^p e^{\frac{|x|}{C_1 \sqrt{t_+}}} \, ds & \leq  C  t_+^{\frac{d+2}{2}} \frac{(\ln t_+)^{\frac{p(d+2)}{2}}}{t_+^{\frac{p}{2} + \frac{pd}{2}}}  \underset{\leq C}{\underbrace{\sum_{(s , y) \in \mathcal{Z}} \exp \left( - \frac{|x|}{2C \sqrt{t_+}}  \right)}} \\
    & \leq C  (\ln t_+)^{\frac{p(d+2)}{2}} t_+^{\frac{d+2}{2}  - \frac{p}{2} - \frac{pd}{2}}.
\end{align*}
\end{proof}

\subsection{Proof of the $L^p$-annealed regularity} \label{sec:DDannealedregsec}

We complete in this section the proof of Proposition~\ref{prop:Lpannealedreg} by combining the Calder\'{o}n-Zygmund regularity estimate proved in the previous section with the techniques of Delmotte and Deuschel~\cite{DD05} (which rely on the space and time stationarity of the Langevin dynamic).

\begin{proof}[Proof of Proposition~\ref{prop:Lpannealedreg}]
We first prove the result with the mixed codifferential (which is the most challenging case) and then point out how the argument can be adapted to obtain the proof with only one codifferential.
    
Given $s \geq 0$ and a realisation $(\varphi_t)_{t \geq 0}$ of the Langevin dynamic, let us denote by
        \begin{equation*}
        \left\{ \begin{aligned}
        \partial_t P^{\varphi_\cdot} (\cdot , \cdot ;s ,  y) + \mathcal{L}_{\mathrm{spat}}^{\varphi_\cdot}  P^{\varphi_\cdot}(\cdot , \cdot ; s , y) & =0 ~\mbox{in}~ (0 , \infty) \times \Zd , \\
        P^{\varphi_\cdot} \left(s,\cdot ; s ,  y \right) & = \delta_{y} ~\mbox{in}~ \Zd.
        \end{aligned}
        \right.
    \end{equation*}
From this definition, we may write the identity, for any $t , s \in (0 , \infty)$ with $s \leq t$,
\begin{equation*}
     \di^*_1 \di^*_2 P^{\varphi_\cdot} (t , x ;  0) = \sum_{y \in \Zd} \di^*_1 P^{\varphi_\cdot} (t , x ; s, y) \di^*_2 P^{\varphi_\cdot} (s , y ;  0).
\end{equation*}
We let $C_1 <\infty$ be a large constant which is allowed to depend on $d , \beta$ and $p$ and whose specific value will be selected below. From this identity, the Cauchy-Schwarz inequality and the inequality $ab \leq a^2 + b^2$, we deduce that, for any $s \in (0, t)$,
\begin{align} \label{eq:convpropheatkernel}
    \left| \di^*_1 \di^*_2 P^{\varphi_\cdot} (t , x ;  0) \right| e^{|x|/\left(C_1 \sqrt{t_+}\right)} & \leq \sum_{y \in \Zd} \left| \di^*_1 P^{\varphi_\cdot} (t , x ; s, y) \right| \left| \di^*_2 P^{\varphi_\cdot} (s , y ;  0) \right| e^{|x|/\left( C_1 \sqrt{t_+}\right)}
    \\
    & \leq \sum_{y \in \Zd} \left| \di^*_1 P^{\varphi_\cdot} (t , x ; s, y) \right|e^{|x-y|/\left(C_1 \sqrt{t_+}\right)}  \left| \di^*_2 P^{\varphi_\cdot} (s , y ;  0) \right| e^{|y|/\left(C_1 \sqrt{t_+}\right)} \notag \\
    & \leq \sum_{y \in \Zd} \left| \di^*_1 P^{\varphi_\cdot} (t , x ; s, y) \right|^2 e^{2|x-y|/\left(C_1 \sqrt{t_+}\right)} + \sum_{y \in \Zd} \left| \di^*_2 P^{\varphi_\cdot} (s , y ;  0) \right|^2 e^{2|y|/\left(C_1 \sqrt{t_+}\right)}. \notag
\end{align}
We next use the Jensen inequality (and noting that a constant $C$ appears on the second line because the sum on the left-hand side of~\eqref{ineq:almostprobab} is not necessarily equal to $1$) to deduce that
\begin{align*} 
   \left(  \sum_{y \in \Zd} \left| \di^*_1 P^{\varphi_\cdot} (t , x ; s, y) \right|^2 e^{2|x-y|/\left(C_1 \sqrt{t_+}\right)} \right)^p & =   \left(  \sum_{y \in \Zd} \left| \di^*_1 P^{\varphi_\cdot} (t , x ; s, y) \right|^2 e^{2|x-y|/\left(C_1 \sqrt{t_+}\right)}  \Phi_{C_1}(t , x - y)^{-1} \Phi_{C_1}(t , x- y) \right)^p \\
   & \leq C \sum_{y \in \Zd} \left| \di^*_1 P^{\varphi_\cdot} (t , x ; s, y) \right|^{2p} e^{2p|x-y|/\left(C_1 \sqrt{t_+}\right)} \Phi_{C_1}(t , x - y)^{-p} \Phi_{C_1}(t , x- y) \notag \\
   & \leq C\sum_{y \in \Zd} \left| \di^*_1 P^{\varphi_\cdot} (t , x ; s, y) \right|^{2p} \Phi_{C_2}(t , x- y)^{1-p}, \notag
\end{align*}
where we set $C_2 := C_1 / (1 + 2p/(p-1))$ in order to absorb the exponential factor. In the rest of the proof, we denote by $\E$ the expectation with respect to the stationary Langevin dynamic, i.e., $\E \left[ \cdot \right] := \langle \E_\varphi \left[ \cdot \right]\rangle_{\mu_\beta}$. Taking the expectation on the right-hand side and using the space and time stationarity of the dynamic, we deduce that
\begin{equation*}
    \E \left[ \left| \di^*_1 P^{\varphi_\cdot} (t , x ; s, y) \right|^{2p} \right] =   \E \left[ \left| \di^*_1 P^{\varphi_\cdot} (t - s , x - y ; 0, 0) \right|^{2p}  \right] .
\end{equation*}
By combining the two previous inequalities, we deduce that
\begin{align} \label{eq:111612}
    \E \left[ \left(  \sum_{y \in \Zd} \left| \di^*_1 P^{\varphi_\cdot} (t , x ; s, y) \right|^2 e^{2|x-y|/\left(C_1 \sqrt{t_+}\right)} \right)^p \right] & \leq C \sum_{y \in \Zd} \E \left[ \left| \di^*_1 P^{\varphi_\cdot} (t , x ; s, y) \right|^{2p} \right] \Phi_{C_2}(t , y)^{1-p} \\
    &  \leq C \sum_{y \in \Zd} \E \left[ \left| \di^*_1 P^{\varphi_\cdot} (t - s , x - y ; 0, 0) \right|^{2p} \right] \Phi_{C_2}(t , x - y)^{1-p} \notag \\
    & = C \sum_{y \in \Zd} \E \left[ \left| \di^*_1 P^{\varphi_\cdot} (t - s , y ; 0) \right|^{2p} \right] \Phi_{C_2}(t , y)^{1-p} \notag \\
    & =C \E \left[ \sum_{y \in \Zd}  \left| \di^*_1 P^{\varphi_\cdot} (t - s , y ; 0) \right|^{2p} \Phi_{C_2}(t , y)^{1-p} \right]. \notag
\end{align}
A similar computation shows that
\begin{equation*} 
    \E \left[ \left( \sum_{y \in \Zd} \left| \di^*_2 P^{\varphi_\cdot} (s , y ;  0) \right|^2 e^{2|y|/\left(C_1 \sqrt{t_+}\right)} \right)^p \right] \leq C \E \left[  \sum_{y \in \Zd} \left| \di^*_2 P^{\varphi_\cdot} (s , 0 ; y) \right|^{2p} \Phi_{C_2}(t , y)^{1-p} \right].
\end{equation*}
We then use Remark~\ref{rem:reversinghtetime} together with the time reversibility of the Langevin dynamic to deduce that, for any $y \in \Zd$,
\begin{equation*}
    \E \left[ \left| \di^*_2 P^{\varphi_\cdot} (s , 0 ; y) \right|^{2p} \right] = \E \left[  \left| \di^*_1 P^{\varphi_\cdot} (s , y ; 0) \right|^{2p} \right].
\end{equation*}
Consequently
\begin{equation} \label{eq:111712}
    \E \left[ \left( \sum_{y \in \Zd} \left| \di^*_2 P^{\varphi_\cdot} (s , y ;  0) \right|^2 e^{2|y|/\left(C_1 \sqrt{t_+}\right)} \right)^p \right] \leq C \E \left[  \sum_{y \in \Zd} \left| \di^*_1 P^{\varphi_\cdot} (s , y ; 0) \right|^{2p} \Phi_{C_2}(t , y)^{1-p} \right].
\end{equation}
Combining the identity~\eqref{eq:convpropheatkernel} with the inequalities~\eqref{eq:111612} and~\eqref{eq:111712}, we deduce that, for any $s \in (0 , t)$,
\begin{multline*}
    \E \left[ \left| \di^*_1 \di^*_2 P^{\varphi_\cdot} (t , x ;  0) \right|  \right] e^{|x|/\left(C_1 \sqrt{t_+}\right)} \\
    \leq C \E \left[ \sum_{y \in \Zd}  \left| \di^*_1 P^{\varphi_\cdot} (t - s , y ; 0) \right|^{2p} \Phi_{C_2}(t , y)^{1-p} \right] +  C \E \left[ \sum_{y \in \Zd}  \left| \di^*_1 P^{\varphi_\cdot} (s , y ; 0) \right|^{2p} \Phi_{C_2}(t , y)^{1-p} \right].
\end{multline*}
Since the time $s \in (0,t)$ is a free parameter, we may integrate over $s \in \left( \frac{t}{4}, \frac{3t}{4} \right)$ to deduce that
\begin{equation*}
    \E \left[ \left| \di^*_1 \di^*_2 P^{\varphi_\cdot} (t , x ;  0) \right|^p \right] e^{|x|/\left(C_1 \sqrt{t_+}\right)} \leq C \E \left[ t^{-1}\int_{t/4}^{3t/4} \sum_{y \in \Zd}  \left| \di^*_1 P^{\varphi_\cdot} (s , y ; 0) \right|^{2p} \Phi_{C_2}(t , y)^{1-p} \, ds \right].
\end{equation*}
Selecting the inverse temperature $\beta$ sufficiently large so that the Calder\'{o}n-Zygmund regularity (Proposition~\ref{prop:CalderonZygmund}) holds with the exponent $2p$ (and selecting the constant $C_1$, and thus $C_2$, sufficiently large depending on $d , \beta, p$), we deduce that
\begin{align*}
     \E \left[ t^{-1} \int_{t/4}^{3t/4} \sum_{y \in \Zd}  \left| \di^*_1 P^{\varphi_\cdot} (s , y ; 0) \right|^{2p} \Phi_{C_2}(t , y)^{1-p} \, ds \right] & = \E \left[ t^{-1 - \frac{(1-p)d}{2}} \int_{t/4}^{3t/4} \sum_{y \in \Zd}  \left| \di^*_1 P^{\varphi_\cdot} (s , y ; 0) \right|^{2p}  e^{\frac{(p-1)|y|}{C_2 \sqrt{t_+}}} \, ds \right] \\
     & \leq C \E \left[ t^{-1 -\frac{(1-p)d}{2}} \int_{t/4}^{3t/4} \sum_{y \in \Zd}  \left| \nabla P^{\varphi_\cdot} (s , y ; 0) \right|^{2p}  e^{\frac{(p-1)|y|}{C_2 \sqrt{t_+}}} \, ds \right] \\
     & \leq  \frac{C (\ln t_+)^{p(d+2)}}{t^{p(1 + \frac{d}{2})}}
\end{align*}
(N.B. in the second inequality, we used the first item of Remark~\ref{rem:remark25} to deduce that the norm of the codifferential is smaller than the norm of the gradient). Combining the two previous inequalities, we obtain
\begin{equation*}
    \left\| \di^*_1 \di^*_2 P^{\varphi_\cdot} (t , x ;  0)  \right\|_{L^p (\Omega \, ; \, \R^{d \times d})}^p =  \E \left[ \left| \di^*_1 \di^*_2 P^{\varphi_\cdot} (t , x ;  0) \right|^p \right] \leq  \frac{C (\ln t_+)^{p(d+2)}}{t^{p(1 + \frac{d}{2})}} \exp \left( - \frac{|x|}{C_1 \sqrt{t_+}} \right)
\end{equation*}
and thus completes the proof of Proposition~\ref{prop:Lpannealedreg} in the case of the mixed codifferential of the heat kernel.

In the case where there is only one codifferential (we will only treat the case when the codifferential is applied to the first variable), we start similarly by writing the identity, for any $t , s \in (0 , \infty)$ with $s \leq t$,
\begin{equation*}
     \di^*_1 P^{\varphi_\cdot} (t , x ;  0) = \sum_{y \in \Zd} \di^*_1 P^{\varphi_\cdot} (t , x ; s, y) P^{\varphi_\cdot} (s , y ;  0).
\end{equation*}
We then deduce that, for any $s \in (0, t)$,
\begin{align} \label{eq:convpropheatkernel}
    \left| \di^*_1 P^{\varphi_\cdot} (t , x ;  0) \right| e^{|x|/\left(C_1 \sqrt{t_+}\right)} & \leq \sum_{y \in \Zd} \left| \di^*_1 P^{\varphi_\cdot} (t , x ; s, y) \right|e^{|x-y|/\left(C_1 \sqrt{t_+}\right)}  \left| P^{\varphi_\cdot} (s , y ;  0) \right| e^{|y|/\left(C_1 \sqrt{t_+}\right)} \notag \\
    & \leq \sum_{y \in \Zd} \sqrt{t_+} \left| \di^*_1 P^{\varphi_\cdot} (t , x ; s, y) \right|^2 e^{2|x-y|/\left(C_1 \sqrt{t_+}\right)} + \sum_{y \in \Zd} \frac{1}{\sqrt{t_+}} \left| P^{\varphi_\cdot} (s , y ;  0) \right|^2 e^{2|y|/\left(C_1 \sqrt{t_+}\right)}. \notag
\end{align}
The first term on the right-hand side can be treated using the same argument as in the case of the mixed codifferential, and the second term can be estimated using Proposition~\ref{prop.prop4.11chap4HK} (which holds for any realisation of the Langevin dynamic).
\end{proof}

\section{Sums on the lattice} \label{App.sumonlattice}

In this appendix, we provide upper bounds for sum of polynomially decaying functions on the lattice which are frequently used in the proofs above.

\begin{proposition} \label{prop:appCdiscretesum}
    For any exponents $\alpha, \gamma, a, c \in [0, \infty)$ with $\alpha + \delta > d$, there exists a constant $C := C(d, \alpha, \gamma, a, b) < \infty$, such that, for any $x \in \Zd$,
    \begin{equation} \label{eq:truineqB1}
\sum_{y \in \Zd} \frac{( \ln |y-x|_+ )^{a}}{|y - x|_+^{\alpha}}  \frac{( \ln |y|_+ )^{c}}{|y|_+^{\gamma}}
\le
\left\{ \begin{aligned}
\frac{C( \ln |x|_+ )^{a + c} }{|x|_+^{\alpha + \gamma - d}} & \hspace{3mm} \mbox{if}~& \alpha, \delta \in (0,d),\\
\frac{C ( \ln |x|_+ )^{a + c + 1}}{|x|_+^{\alpha}} &\hspace{3mm} \mbox{if}~& \alpha \in (0,d] ~\mbox{and}~ \gamma = d,\\
 \frac{ \left( \ln |x|_+ \right)^{a}}{|x|_+^{\alpha}} &\hspace{3mm} \mbox{if}~& \gamma > d ~\mbox{and} ~ \alpha < \gamma  \\
\end{aligned} \right.
\end{equation}
and
\begin{equation} \label{ineqB4}
     \sum_{z , z' \in \Zd} \frac{1}{|z|_+^{d-1} |z - x|_+^{d-1}}  \frac{1}{|z'|_+^{d-1} |z' - x|_+^{d-1}} \frac{(\ln |z - z'|)^a}{|z - z'|_+^{d}}  \leq \frac{C}{|x|^{2d-2}_+}.
\end{equation}
There exists a constant $C := C(d, a , c) < \infty$ such that, for any $y_1 , y_2 \in \Zd$,
        \begin{equation}
\label{e.sumlattice1}
    \sum_{z \in \Zd} \frac{\left( \ln \left( |y_1-z|_+ + |y_2 - z|_+ \right) \right)^a}{|y_1-z|_+^{2d+1} + |y_2 - z|_+^{2d+1}}  \frac{\left( \ln \left( |z|_+  \right) \right)^c}{|z|^{d-1}_+} \leq \frac{C \left( \ln \left( |y_1|_+ + |y_2|_+ \right) \right)^{a + c} }{|y_1|_+^{d-1} + |y_2|_+^{d-1}}  \frac{1}{|y_1-y_2|^{d+1}_+}.
\end{equation}
\end{proposition}

The proof of this proposition is split into 4 different proofs.

\begin{proof}[Proof of the inequality~\eqref{eq:truineqB1}]
The inequality~\eqref{eq:truineqB1} can be established using standard arguments; we refer the reader to~\cite[Appendix C]{DW} for a complete proof (among other possible references). We note that not all possible cases for the exponents are treated here. A change of variables $y \mapsto x-y$ interchanges the roles of $\alpha$ and $\delta$, thereby covering several additional cases. The remaining cases, namely those with $\alpha = \delta$ and $\alpha , \delta \geq d$, can be handled similarly with little additional effort. However, these are not required for the purposes of this article.
\end{proof}

\begin{proof}[Proof of the inequality~\eqref{ineqB4}]
To simplify the presentation of the argument, we only present the proof in the case $a = 0$. For the inequality~\eqref{ineqB4}, we will first prove the upper bound, for any $x , z \in \Zd$,
\begin{equation} \label{intermediateB2}
    \sum_{z' \in \Zd} \frac{1}{|z'|_+^{d-1} |z' - x|_+^{d-1}} \frac{1}{|z - z'|_+^{d}}  \leq  \frac{C}{|x|_+^{d-1}} \frac{\ln |z|_+ }{\left| z \right|_+^{d-1}} +  \frac{C}{|x|_+^{d-1}} \frac{\ln |x - z|_+}{|x - z|_+^{d-1}} + \frac{C}{|x|^{2d-2}_+}.
\end{equation}
This inequality is shown by considering two different cases.

\textbf{Case I:} $|z|_+ \leq \frac{|x|_+}{10}$. In that case, we partition the space into two regions
\begin{equation*}
    \left\{ z' \in \Zd \, : \, \left| z' \right| \leq \frac{|x|_+}{2} \right\} ~~\mbox{and}~~ \left\{ z' \in \Zd \, : \, \left| z' \right| > \frac{|x|_+}{2} \right\}.
\end{equation*}
In the first region, we may use the inequality $|z' - x|_+ \geq |x|_+ / 2$ and the inequality~\eqref{eq:truineqB1} to write
\begin{align*}
    \sum_{\substack{z' \in \Zd \\ |z'|_+ \leq |x|_+ / 2 }} \frac{1}{|z'|_+^{d-1} |z' - x|_+^{d-1}} \frac{1}{|z - z'|_+^{d}} & \leq \frac{C}{|x|_+^{d-1}} \sum_{\substack{z' \in \Zd \\ |z'|_+ \leq |x|_+ / 2 }} \frac{1}{|z'|_+^{d-1}} \frac{1}{|z - z'|_+^{d}} \\
    & \leq \frac{C}{|x|_+^{d-1}} \sum_{z' \in \Zd } \frac{1}{|z'|_+^{d-1}} \frac{1}{|z - z'|_+^{d}} \\
    & \leq \frac{C}{|x|_+^{d-1}} \frac{ \ln |z|_+ }{\left| z \right|_+^{d-1}}.
\end{align*}
In the second region, we use the inequality $|z' - z|_+ \geq |x|_+/4$ and the inequality~\eqref{eq:truineqB1} to write
\begin{align*}
    \sum_{\substack{z' \in \Zd \\ |z'|_+ \geq |x|_+ / 2 }} \frac{1}{|z'|_+^{d-1} |z' - x|_+^{d-1}} \frac{1}{|z - z'|_+^{d}} & \leq \frac{C}{|x|_+^{d}} \sum_{\substack{z' \in \Zd \\ |z'|_+ \geq |x|_+ / 2 }} \frac{1}{|z'|_+^{d-1}} \frac{1}{|x - z'|_+^{d-1}} \\
    & \leq \frac{C}{|x|_+^{d}} \sum_{z' \in \Zd} \frac{1}{|z'|_+^{d-1}} \frac{1}{|x - z'|_+^{d-1}} \\
    & \leq \frac{C}{|x|_+^{2d-2}}.
\end{align*}
A combination of the two previous inequalities implies that
\begin{equation} \label{resultstep1B2}
    \sum_{z' \in \Zd} \frac{1}{|z'|_+^{d-1} |z' - x|_+^{d-1}} \frac{1}{|z - z'|_+^{d}} \leq \frac{C}{|x|_+^{d-1}} \frac{\ln |z|_+}{\left| z \right|_+^{d-1}} + \frac{C}{|x|^{2d-2}_+}.
\end{equation}

\textbf{Case II:} $|z|_+ \geq \frac{|x|_+}{10}$. In this case, we similarly split the space into two regions
\begin{equation*}
     \left\{ z' \in \Zd \, : \, \left| z' \right| \leq \frac{|x|_+}{20} \right\} ~~\mbox{and}~~ \left\{ z' \in \Zd \, : \, \left| z' \right| > \frac{|x|_+}{20} \right\}.
\end{equation*}
In the first region, we use that $|z' - z|_+ \geq c |x|_+$ and $|z' - x|_+ \geq c |x|_+$ (for some explicit $c > 0$) to deduce that
\begin{align*}
    \sum_{\substack{z' \in \Zd \\ |z'|_+ \leq |x|_+ / 20 }} \frac{1}{|z'|_+^{d-1} |z' - x|_+^{d-1}} \frac{1}{|z - z'|_+^{d}} & \leq \frac{C}{|x|_+^{2d-1}} \sum_{\substack{z' \in \Zd \\ |z'|_+ \leq |x|_+ / 20 }} \frac{1}{|z'|_+^{d-1}} \\
    & \leq \frac{C}{|x|_+^{2d-2}}.
\end{align*}
In the second region, we use the inequality~\eqref{eq:truineqB1} to write
\begin{align*}
    \sum_{\substack{z' \in \Zd \\ |z'|_+ \geq |x|_+ / 20 }} \frac{1}{|z'|_+^{d-1} |z' - x|_+^{d-1}} \frac{1}{|z - z'|_+^{d}} & \leq \frac{C}{|x|_+^{d-1}} \sum_{\substack{z' \in \Zd \\ |z'|_+ \geq |x|_+ / 20 }} \frac{1}{|z'-x|_+^{d-1}} \frac{1}{|z - z'|_+^{d}} \\
    & \leq \frac{C}{|x|_+^{d-1}} \sum_{z' \in \Zd} \frac{1}{|z' - x|_+^{d-1}} \frac{1}{|z - z'|_+^{d}} \\
    & \leq \frac{C}{|x|_+^{d-1}} \frac{ \ln |x - z|_+}{|x - z|_+^{d-1}}.
\end{align*}
A combination of the two previous inequalities implies that
\begin{equation*}
    \sum_{z' \in \Zd} \frac{1}{|z'|_+^{d-1} |z' - x|_+^{d-1}} \frac{1}{|z - z'|_+^{d}} \leq \frac{C}{|x|_+^{2d-2}} + \frac{C}{|x|_+^{d-1}} \frac{\ln |x - z|_+}{|x - z|_+^{d-1}}.
\end{equation*}
Combining this inequality with~\eqref{resultstep1B2} completes the proof of~\eqref{intermediateB2}.

There only remains to show that the inequality~\eqref{intermediateB2} implies~\eqref{ineqB4}. To this end, we make use of~\eqref{eq:truineqB1}, the inequality $\ln |z|_+ \leq C |z|_+^{1/2}$ and the inequality $2d - 2 -1/2 > d$ (since $d \geq 3$) to write
\begin{align*}
    \sum_{z , z' \in \Zd} \frac{1}{|z|_+^{d-1} |z - x|_+^{d-1}}  \frac{1}{|z'|_+^{d-1} |z' - x|_+^{d-1}} \frac{1}{|z - z'|_+^{d}} & \leq \frac{C}{|x|^{d-1}_+}\sum_{z \in \Zd} \frac{\ln |z|_+}{|z|_+^{2d-2} |z - x|_+^{d-1}}  + \frac{C}{|x|^{d-1}_+}\sum_{z \in \Zd} \frac{\ln |x - z|_+}{|z|_+^{d-1} |z - x|_+^{2d-2}} \\
    & \quad + \frac{C}{|x|^{2d-2}_+} \sum_{z \in \Zd} \frac{1}{|z|_+^{d-1} |z - x|_+^{d-1}} \\
    & \leq \frac{C}{|x|_+^{2d-2}}.
\end{align*}
The proof of the inequality~\eqref{ineqB4} is complete.
\end{proof}

\begin{proof}[Proof of the inequality~\eqref{e.sumlattice1}]
Once again, in order to simplify the presentation of the argument, we only present the proof in the case $a = c = 0$. We first use the parallelogram law to write
\begin{align*}
    |y_1-z|_+^{2d+1} + |y_2 - z|_+^{2d+1} \geq c \left( |y_1-z|_+^{2} + |y_2 - z|_+^{2} \right)^{\frac{2d+1}{2}}
    & \geq c \left( |y_1- y_2|_+^{2} + \left|\frac{y_2 + y_1}{2} - z\right|_+^{2} \right)^{\frac{2d+1}{2}} \\
    & \geq c \left( |y_1- y_2|_+^{2d+1} + \left|\frac{y_2 + y_1}{2} -  z\right|_+^{2d+1} \right).
\end{align*}
Setting $x = y_1- y_2$ and $y = \frac{y_2 + y_1}{2}$, we see that (by the parallelogram law and a similar computation)
\begin{equation*}
   |x|_+^{2d+1} + |y|_+^{2d+1} \geq c \left( |y_1|_+^{2d+1} + |y_2|_+^{2d+1}  \right).
\end{equation*}
From the two previous inequalities, we see that~\eqref{e.sumlattice1} is equivalent to the following inequality
\begin{equation} \label{eq:ineqB2}
    \sum_{z \in \Zd} \frac{1}{|x|_+^{2d+1} + |y - z|_+^{2d+1}}  \frac{1}{|z|^{d-1}_+} \leq \frac{C}{|x|_+^{d-1} + |y|_+^{d-1}}  \frac{1}{|x|^{d+1}_+}.
\end{equation}
The rest of the argument is devoted to the proof of this inequality, and we discuss two cases.

\medskip

\textbf{Case I:} $|y|>2|x|$. In this case, the inequality~\eqref{eq:ineqB2} can be restated as follows
\begin{equation*}
    \sum_{z \in \Zd} \frac{1}{|x|_+^{2d+1} + |y - z|_+^{2d+1}}  \frac{1}{|z|^{d-1}_+} \leq \frac{C}{|y|_+^{d-1}}  \frac{1}{|x|^{d+1}_+}.
\end{equation*}
We will prove this inequality by splitting the space into the four regions (see Figure~\ref{fig:4regions})
\begin{equation*}
    \left\{ z \in \Zd \, : \, |y - z| \leq |x| \right\}, ~~  \left\{ z \in \Zd \, : \, |x| < |y - z| \leq \frac{|y|}{2} \right\}~~\mbox{and}~~ \left\{ z \in \Zd \, : \,  |z| \leq \frac{|y|}{2}  \right\}
\end{equation*}
and
\begin{equation*}
    \left\{ z \in \Zd \, : \,  \frac{|y|}{2} < |y - z| ~\mbox{and}~ \frac{|y|}{2} < |z|  \right\}.
\end{equation*}

\begin{figure}
    \centering
    \includegraphics[scale=0.6]{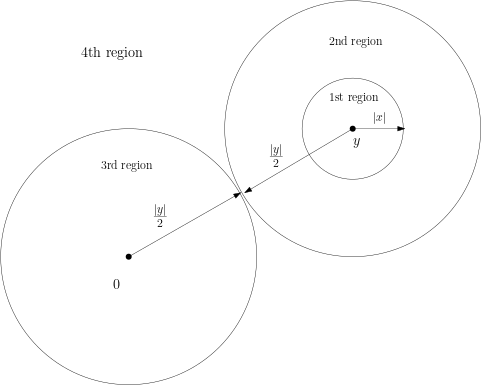}
    \caption{The four regions used in the proof of the inequality~\eqref{e.sumlattice1}.}
    \label{fig:4regions}
\end{figure}

\medskip

\textit{Step 1: the first region.} We may use the inequalities
\begin{equation*}
    |x|_+^{2d+1} + |y - z|_+^{2d+1} \geq |x|_+^{2d+1} \hspace{5mm} \mbox{and} \hspace{5mm} |z|_+ \geq \frac{|y|_+}{2}
\end{equation*}
to obtain
\begin{equation*}
    \sum_{z \in \Zd \, : \, |y - z| < |x| } \frac{1}{|x|_+^{2d+1} + |y - z|_+^{2d+1}}  \frac{1}{|z|^{d-1}_+} \leq \sum_{z \in \Zd \, : \, |y - z| < |x| } \frac{C}{|x|_+^{2d+1}} \frac{1}{|y|^{d-1}} \leq \frac{C}{|x|_+^{d+1}} \frac{1}{|y|_+^{d-1}},
\end{equation*}
where in the last inequality we used that the cardinality of the set $\{ z \in \Zd \, : \, |y - z| \leq |x| \}$ is of order $|x|^d_+$.

\medskip

\textit{Step 2: the second region.} We use the inequalities
\begin{equation*}
    |x|_+^{2d+1} + |y - z|_+^{2d+1} \geq |y - z|_+^{2d+1} \hspace{5mm} \mbox{and} \hspace{5mm} |z|_+ \geq \frac{|y|_+}{2}
\end{equation*}
to obtain that
\begin{align*}
    \sum_{z \in \Zd \, : \, |x| < |y - z| \leq \frac{|y|}{2}  } \frac{1}{|x|_+^{2d+1} + |y - z|_+^{2d+1}}  \frac{1}{|z|^{d-1}_+} & \leq \sum_{z \in \Zd \, : \,  |x| < |y - z| \leq \frac{|y|}{2} } \frac{1}{|y-z|_+^{2d+1}} \frac{1}{|y|_+^{d-1}} \\
    & \leq \frac{C}{|y|_+^{d-1}} \sum_{z \in \Zd \, : \,  |x| < |z| \leq \frac{|y|}{2} } \frac{1}{|z|_+^{2d+1}} \\
    & \leq \frac{C}{|y|_+^{d-1}} \sum_{z \in \Zd \, : \,  |x| < |z| } \frac{1}{|z|_+^{2d+1}}  \\
    & \leq \frac{C}{|y|_+^{d-1}} \frac{1}{|x|_+^{d+1}}.
\end{align*}

\medskip

\textit{Step 3: the third region.}  In this case, we note that there exists $c>0$ such that
\begin{equation*}
    |x|_+^{2d+1} + |y - z|_+^{2d+1} \geq c |y|_+^{2d+1}
\end{equation*}
and use it to deduce that
\begin{align*}
    \sum_{z \in \Zd \, : \, |z| < \frac{|y|}{2} } \frac{1}{|x|_+^{2d+1} + |y - z|_+^{2d+1}}  \frac{1}{|z|^{d-1}_+} \leq \frac{C}{|y|_+^{2d+1}} \sum_{z \in \Zd \, : \,  |z| < \frac{|y|}{2} } \frac{C}{|z|_+^{d-1}} \leq \frac{C}{|y|_+^{2d}} \leq  \frac{C}{|y|_+^{d-1}} \frac{1}{|x|_+^{d+1}}.
\end{align*}

\medskip

\textit{Step 4: the fourth region.} In this case, we use the inequality, for some constant $c > 0$,
\begin{equation*}
     |x|_+^{2d+1} + |y - z|_+^{2d+1} \geq c |z|_+^{2d+1}
\end{equation*}
to deduce that
\begin{align*}
    \sum_{z \in \Zd \, : \,  \frac{|y|}{2} < |y - z| ~\mbox{and}~ \frac{|y|}{2} < |z|} \frac{1}{|x|_+^{2d+1} + |y - z|_+^{2d+1}}  \frac{1}{|z|^{d-1}_+} \leq \sum_{z \in \Zd \, : \,  \frac{|y|}{2} < |y - z| ~\mbox{and}~ \frac{|y|}{2} < |z|} \frac{C}{|z|_+^{3d}} & \leq \sum_{z \in \Zd \, : \, \frac{|y|}{2} < |z|}  \frac{C}{|z|_+^{3d}} \\
    & \leq \frac{C}{|y|_+^{2d}} \\
    & \leq \frac{C}{|y|_+^{d-1}}  \frac{1}{|x|^{d+1}_+}.
\end{align*}

\medskip

\textbf{Case II:} $|y|\le 2|x|$. In this case, the inequality~\eqref{eq:ineqB2} can be restated as follows
\begin{equation*}
    \sum_{z \in \Zd} \frac{1}{|x|_+^{2d+1} + |y - z|_+^{2d+1}}  \frac{1}{|z|^{d-1}_+} \leq \frac{C}{|x|_+^{2d}}.
\end{equation*}
Similarly, we split the sum into two regions:
\begin{equation*}
    \left\{ z \in \Zd \, : \, |z| \leq |x| \right\} \hspace{5mm} \mbox{and} \hspace{5mm} \left\{ z \in \Zd \, : \, |z| > |x| \right\}.
\end{equation*}

\medskip

\textit{Step 1: the first region.} We use the inequality
\begin{equation*}
    |x|_+^{2d + 1} + |y - z|_+^{2d + 1} \geq |x|_+^{2d + 1} 
\end{equation*}
to deduce that
\begin{equation*}
    \sum_{z \in \Zd \, : \, |z| \leq |x| } \frac{1}{|x|_+^{2d+1} + |y - z|_+^{2d+1}}  \frac{1}{|z|^{d-1}_+} \leq \frac{1}{|x|_+^{2d+1}} \sum_{z \in \Zd \, : \, |z| \leq |x| }  \frac{1}{|z|^{d-1}_+} \leq \frac{C}{|x|_+^{2d}}.
\end{equation*}

\medskip

\textit{Step 2: the second region.} We use the inequality
\begin{equation*}
    |x|_+^{2d + 1} + |y - z|_+^{2d + 1} \geq c |z|_+^{2d + 1} 
\end{equation*}
to deduce that
\begin{equation*}
    \sum_{z \in \Zd \, : \, |z| \geq |x| } \frac{1}{|x|_+^{2d+1} + |y - z|_+^{2d+1}}  \frac{1}{|z|^{d-1}_+} \leq \sum_{z \in \Zd \, : \, |z| \geq |x| }  \frac{C}{|z|^{3d}_+} \leq \frac{C}{|x|_+^{2d}}.
\end{equation*}
\end{proof}

{\small
\bibliographystyle{alpha}
\bibliography{villain}
}
\end{document}